\documentclass[english]{article}
\usepackage[T1]{fontenc}
\usepackage[latin9]{inputenc}
\usepackage{geometry}
\usepackage{float}
\usepackage{amsmath}
\usepackage{amsthm}
\usepackage{amssymb}
\usepackage{graphicx}
\usepackage{esint}

\makeatletter

\floatstyle{ruled}
\newfloat{algorithm}{tbp}{loa}
\providecommand{\algorithmname}{Algorithm}
\floatname{algorithm}{\protect\algorithmname}

  \theoremstyle{definition}
  \newtheorem*{condition*}{\protect\conditionname}
\theoremstyle{plain}
\newtheorem{thm}{\protect\theoremname}[section]
  \theoremstyle{plain}
  \newtheorem{prop}[thm]{\protect\propositionname}
  \theoremstyle{definition}
  \newtheorem{defn}[thm]{\protect\definitionname}
  \theoremstyle{plain}
  \newtheorem{lem}[thm]{\protect\lemmaname}
  \theoremstyle{plain}
  \newtheorem{cor}[thm]{\protect\corollaryname}
  \theoremstyle{plain}
  \newtheorem*{assumption*}{\protect\assumptionname}

\usepackage{algorithm,algpseudocode}

\makeatother

\usepackage{babel}
  \providecommand{\assumptionname}{Assumption}
  \providecommand{\conditionname}{Condition}
  \providecommand{\corollaryname}{Corollary}
  \providecommand{\definitionname}{Definition}
  \providecommand{\lemmaname}{Lemma}
  \providecommand{\propositionname}{Proposition}
\providecommand{\theoremname}{Theorem}

\begin{document}

\title{Existence of a calibrated regime switching local volatility model
and new fake Brownian motions}

\author{B. Jourdain$^{*}$, A. Zhou\thanks{Université Paris-Est, Cermics (ENPC), INRIA, F-77455, Marne-la-Vallée,
France. E-mails : benjamin.jourdain@enpc.fr, alexandre.zhou@enpc.fr
- This research benefited from the support of the \textquotedblleft Chaire
Risques Financiers\textquotedblright , Fondation du Risque, and the
French National Research Agency under the program ANR-12-BS01-0019
(STAB).}}
\maketitle
\begin{abstract}
By Gyongy's theorem, a local and stochastic volatility (LSV) model
is calibrated to the market prices of all European call options with
positive maturities and strikes if its local volatility function is
equal to the ratio of the Dupire local volatility function over the
root conditional mean square of the stochastic volatility factor given
the spot value. This leads to a SDE nonlinear in the sense of McKean.
Particle methods based on a kernel approximation of the conditional
expectation, as presented in \cite{GuyonLabordere}, provide an efficient
calibration procedure even if some calibration errors may appear when
the range of the stochastic volatility factor is very large. But so
far, no global existence result is available for the SDE nonlinear
in the sense of McKean. In the particular case where the local volatility
function is equal to the inverse of the root conditional mean square
of the stochastic volatility factor multiplied by the spot value given
this value and the interest rate is zero, the solution to the SDE
is a fake Brownian motion. When the stochastic volatility factor is
a constant (over time) random variable taking finitely many values
and the range of its square is not too large, we prove existence to
the associated Fokker-Planck equation. Thanks to \cite{Figalli},
we then deduce existence of a new class of fake Brownian motions.
We then extend these results to the special case of the LSV model
called regime switching local volatility, where the stochastic volatility
factor is a jump process taking finitely many values and with jump
intensities depending on the spot level. Under the same condition
on the range of its square, we prove existence to the associated Fokker-Planck
PDE. Finally, we deduce existence of the calibrated model by extending
the results in \cite{Figalli}.
\end{abstract}
$\textbf{Keywords}$: local and stochastic volatility models, calibration,
Dupire's local volatility, Fokker-Planck systems, diffusions nonlinear
in the sense of McKean.

\section{Introduction}

The notion of $\textit{fake}$ Brownian motion was introduced by Oleszkiewicz
\cite{Oleszkiewicz} to describe a martingale $\left(X_{t}\right)_{t\geq0}$
such that for any $t\geq0$, $X_{t}\sim\mathcal{N}(0,t)$, but the
process $\left(X_{t}\right)_{t\geq0}$ does not have the same distribution
as the Brownian motion $\left(W_{t}\right)_{t\geq0}$. Madan and Yor
\cite{MadanYor}, Hamza and Klebaner \cite{HamzaKlebaner}, and more
recently Henry-Labordère, Tan and Touzi \cite{HLabordereTanTouzi}
provided construction of discontinuous fake Brownian motions. The
question of the existence of continuous fake Brownian motion was positively
answered by Albin \cite{Albin}, using products of Bessel processes,
and Donati-Martin, Baker and Yor \cite{Baker2011} extended that result
by exhibiting a sequence of continuous martingales with Brownian marginals
and scaling property. Oleszkiewicz gave a simpler example of a continuous
fake Brownian motion, based on the Box Muller transform. Then, Hobson
\cite{Hobson} proved existence of a continous fake exponential Brownian
motion by mixing diffusion processes. More generally, Fan, Hamza and
Klebaner \cite{fan2015} and Hirsch, Profeta, Roynette and Yor \cite{Hirsch2011}
provided construction of self-similar martingales with given marginal
distributions.

One reason for raising interest in the search of processes that have
marginal distributions matching given ones comes from mathematical
finance, where calibration to the market prices of European call options
is a major concern. According to Breeden and Litzenberger \cite{BreedenLitzenberger},
the knowledge of the market prices of those options for a continuum
of strikes and maturities is equivalent to the knowledge of the marginal
distributions of the underlying asset under the pricing measure. Thus,
to be consistent with the market prices, a calibrated model must have
marginal distributions that coincide with those given by the market.
More specifically, in this paper, we address the question of existence
of a special class of calibrated local and stochastic volatility (LSV)
models.

LSV models, introduced by Lipton \cite{Lipton:2002} and by Piterbarg
\cite{Piterbarg}, can be interpreted as an extension of Dupire's
local volatility (LV) model, described in \cite{Dupire}, to get more
consistency with real markets. A typical LSV model has the dynamics
\[
dS_{t}=rS_{t}+f(Y_{t})\sigma(t,S_{t})S_{t}dW_{t}
\]
for the stock under the risk-neutral probability measure, where $\left(Y_{t}\right)_{t\geq0}$
is a stochastic process which may be correlated with $\left(S_{t}\right)_{t\geq0}$
and $r$ is the risk free rate. Meanwhile, according to \cite{Dupire},
given the European call option prices $C(t,K)$ for all positive maturities
$t$ and strikes $K$, the process $\left(S_{t}^{D}\right)_{t\geq0}$,
which follows the dynamics 
\[
dS_{t}^{D}=rS_{t}^{D}dt+\sigma_{Dup}\left(t,S_{t}^{D}\right)S_{t}^{D}dW_{t},
\]
where $\sigma_{Dup}(t,K):=\sqrt{2\frac{\partial_{t}C(t,K)+rK\partial_{K}C(t,K)}{K^{2}\partial_{KK}^{2}C(t,K)}}$
is Dupire's local volatility function, is calibrated to the European
option prices, that is, 
\[
\forall t,K>0,\ C(t,K)=\mathbb{E}\left[e^{-rt}\left(S_{t}^{D}-K\right)^{+}\right].
\]
Under mild assumptions, Gyongy's theorem in \cite{Gyongy} gives that
the choice $\sigma(t,x)=\frac{\sigma_{Dup}(t,x)}{\sqrt{\mathbb{E}\left[f^{2}(Y_{t})|S_{t}=x\right]}}$
ensures that for any $t\geq0$, $S_{t}$ has the same law as $S_{t}^{D}$.
This leads to the following SDE
\[
dS_{t}=rS_{t}dt+\frac{f(Y_{t})}{\sqrt{\mathbb{E}[f^{2}(Y_{t})|S_{t}]}}\sigma_{Dup}(t,S_{t})S_{t}dW_{t},
\]
which is nonlinear in the sense of McKean, because the diffusion term
depends on the law of $(S_{t},Y_{t})$ through the conditional expectation
in the denominator. Getting existence and uniqueness to this kind
of SDE is still an open problem, although some local in time existence
results have been found by Abergel and Tachet \cite{AbergelTachet}
by a perturbative approach. From a numerical point of view, particle
methods based on a kernel approximation of the conditional expectation,
as presented by Guyon and Henry-Labordère \cite{GuyonLabordere},
provide an efficient calibation procedure even if some calibration
errors may appear when the range of the stochastic volatility process
$\left(f(Y_{t})\right)_{t\ge0}$ is very large. 

Recently, advances have been made for analogous models. Alfonsi, Labart
and Lelong \cite{Alfonsi} proved existence and uniqueness of stochastic
local intensity models calibrated to CDO tranche prices, where the
discrete loss process makes the conditional expectation simpler to
handle. Moreover, Guennoun and Henry-Labordère \cite{GuennounLabordere}
showed that in a local volatility model enhanced with jumps, the particle
method applied to a well defined nonlinear McKean SDE with a regularized
volatility function gives call prices that converge to the market
prices as the regularization parameter goes to $0$, for all strikes
and maturities.

In this paper, we first prove existence of a simplified version of
the LSV model. More precisely, we set $r=\frac{1}{2}$, $\sigma_{Dup}\equiv1$,
and consider the dynamics of the asset's log-price, where for simplicity
we also neglect the drift term as its conditional expectation given
the spot is equal to $0$. Moreover, let $d\geq2$ and $Y$ be a random
variable (constant in time) which takes values in $\mathcal{Y}:=\{1,...,d\}$.
We assume that for $i\in\{1,...,d\},\ \alpha_{i}:=\mathbb{P}(Y=i)>0$.
Given a positive function $f$ on $\mathcal{Y}$ and a probability
measure $\mu$ on $\mathbb{R}$, the SDE that we study is thus:

\begin{eqnarray}
dX_{t} & = & \frac{f(Y)}{\sqrt{\mathbb{E}\left[f^{2}(Y)|X_{t}\right]}}dW_{t},\label{eq:ToySDE}\\
X_{0} & \sim & \mu.\nonumber 
\end{eqnarray}
We also suppose that $X_{0}$, $Y$ and $(W_{t})_{t\geq0}$ are independent.
Let us set $\lambda_{i}:=f^{2}(i)>0,i\in\{1,...,d\}$ and denote by
$p_{i}(t,x)$ the conditional density of $X_{t}$ given $\{Y=i\}$
multiplied by $\alpha_{i}$. It means that, for any measurable and
nonnegative function $\phi$, 
\[
\mathbb{E}\left[\phi\left(X_{t}\right)1_{\{Y=i\}}\right]=\int_{\mathbb{R}}\phi(x)p_{i}(t,x)dx.
\]
The Fokker-Planck equations derived from SDE (\ref{eq:ToySDE}) form
a partial differential system (PDS) that writes for $1\le i\leq d$:

\begin{eqnarray}
\partial_{t}p_{i} & = & \frac{1}{2}\partial_{xx}^{2}\left(\frac{\sum_{l=1}^{d}p_{l}}{\sum_{l=1}^{d}\lambda_{l}p_{l}}\lambda_{i}p_{i}\right)\ \text{in \ensuremath{(0,T)\times\mathbb{R}}}\label{eq:toypds2pde}\\
p_{i}(0) & = & \alpha_{i}\mu\ \text{in \ensuremath{\mathbb{R}}},\label{eq:toypds3init}
\end{eqnarray}
where the constant $T>0$ is the finite time horizon. We call this
PDS (FBM). Let us define $\lambda_{min}:=\underset{1\leq i\leq d}{\min}\lambda_{i}$
and $\lambda_{max}:=\underset{1\leq i\leq d}{\max}\lambda_{i}$. We
notice that if $\lambda_{min}=\lambda_{max}=\lambda$, then each $p_{i}$
is a solution to the heat equation with initial condition $\alpha_{i}\mu$.
Results proving existence and uniqueness of solutions to variational
formulations of the heat equation have already been widely developed,
e.g. in \cite{Brezis}, so we focus on the case where $\lambda_{min}<\lambda_{max}$,
that is when the function $f$ is not constant on $\mathcal{Y}$.
We also notice that if we add the equalities (\ref{eq:toypds2pde})
and (\ref{eq:toypds3init}) over the index $i\in\{1,...,d\}$, we
obtain that $\sum_{i=1}^{d}p_{i}$ satisfies the heat equation with
initial condition $\mu$, and it is well known that $\left(\sum_{i=1}^{d}p_{i}\right)(t,x)=\mu*h_{t}(x)$,
with $h$ the heat kernel defined as $h_{t}(x):=\frac{1}{\sqrt{2\pi t}}\exp\left(\frac{-x^{2}}{2t}\right)$
for $t>0$ and $x\in\mathbb{R}$. This observation can also be made
formally, using Gyongy's theorem. Indeed, as we see that 
\begin{equation}
\mathbb{E}\left[\left(\frac{f(Y)}{\sqrt{\mathbb{E}\left[f^{2}(Y)|X_{t}\right]}}\right)^{2}|X_{t}\right]=1,\label{eq:ToySDEmarginal}
\end{equation}
 then if Gyongy's theorem applies, $X_{t}$ has the same law as $X_{0}+W_{t}$,
which has the density $\mu*h_{t}$. 

We introduce the following condition on the family $\left(\lambda_{i}\right)_{1\leq i\leq d}$,
under which we will obtain existence to SDE (\ref{eq:ToySDE}). We
denote by $(e_{1},...,e_{d})$ the canonical basis of $\mathbb{R}^{d}$
and for $x=\left(x_{1},...,x_{d}\right)\in\mathbb{R}^{d}$, we define
$x^{\perp}=\{y=\left(y_{1},...,y_{d}\right)\in\mathbb{R}^{d},\sum_{i=1}^{d}x_{i}y_{i}=0\}$.
\begin{condition*}[C]
 There exists a symmetric positive definite $d\times d$ matrix $\Gamma$
with real valued coefficients such that for $1\leq k\leq d$, the
$d\times d$ matrix $\Gamma^{(k)}$ with coefficients 
\begin{equation}
\Gamma_{ij}^{(k)}=\frac{\lambda_{i}+\lambda_{j}}{2}\left(\Gamma_{ij}+\Gamma_{kk}-\Gamma_{ik}-\Gamma_{jk}\right),\ 1\leq i,j\leq d,\label{eq:Gammakij}
\end{equation}
is positive definite on the space $e_{k}^{\perp}$.
\end{condition*}
Under the Condition $\left(C\right)$, which ensures a coercivity
property that enables to establish energy estimates, existence to
SDE (\ref{eq:ToySDE}) can be proved in three steps:
\begin{enumerate}
\item For $\mu$ having a square integrable density on $\mathbb{R}$, we
define a variational formulation called $V_{L^{2}}(\mu)$ to (FBM)
and we apply Galerkin's method to show that $V_{L^{2}}(\mu)$ has
a solution.
\item For $\mu$ a probability measure on $\mathbb{R}$, we define a weaker
variational formulation called $V(\mu)$. We take advantage of the
fact that if $\left(p_{1},...,p_{d}\right)$ is a solution to $V(\mu)$
then $\sum_{i=1}^{d}p_{i}$ is solution to the heat equation with
initial condition $\mu$. This enables to get control of the explosion
rate of the $L^{2}$ norm of $p_{i},1\leq i\leq d$, as $t\rightarrow0$,
and we extend the results obtained in Step 1 to show existence to
$V(\mu)$. 
\item Thanks to the results in \cite{Figalli}, which give equivalence between
the existence of a solution to a Fokker-Planck equation and the existence
of a solution to the associated martingale problem with time marginals
given by the solution to the Fokker-Planck equation, we show that
the existence result in Step 2 implies existence of a weak solution
to SDE (\ref{eq:ToySDE}). Moreover, if $f$ is non constant on $\mathcal{Y}$
and if $X_{0}=0$, that weak solution is a continuous fake Brownian
motion. 
\end{enumerate}
To get a more realistic financial framework, we then adapt the previous
strategy to obtain existence of a class of calibrated LSV models,
where the process $\left(S_{t}\right)_{t\geq0}$ describing the underlying
asset follows the dynamics 
\begin{equation}
dS_{t}=rS_{t}dt+\frac{f(Y_{t})}{\sqrt{\mathbb{E}[f^{2}(Y_{t})|S_{t}]}}\sigma_{Dup}(t,S_{t})S_{t}dW_{t},\label{eq:LRSVsde}
\end{equation}
the process $\left(Y_{t}\right)_{t\geq0}$ taking values in $\mathcal{Y}$
and 
\[
\forall j\in\{1,...,d\}\backslash\{Y_{t}\},\mathbb{P}\left(Y_{t+dt}=j|\sigma\left(\left(S_{s},Y_{s}\right),0\leq s\leq t\right)\right)=q_{Y_{t}j}\left(\log\left(S_{t}\right)\right)dt,
\]
where the functions $\left(q_{ij}\right)_{1\leq i\neq j\leq d}$ are
non negative and bounded. The process $(S_{t},Y_{t})_{t\geq0}$ is
then a calibrated regime switching local volatility (RSLV) model. 

The paper is organized as follows. In Section \ref{sec:newClass},
we state the main results that give existence to SDE (\ref{eq:ToySDE}).
Section \ref{sec:proofsnewclass} is dedicated to the proofs of the
results in Section \ref{sec:newClass}. In Section \ref{sec:CalibrationofRSLVmain},
we state the existence of calibrated RSLV models and we prove that
result in Section \ref{sec:calibrationofRSLVproofs}. Beforehand,
we introduce some additional notation.

\subsection*{Notation}
\begin{itemize}
\item For an interval $I\subset\mathbb{R}$, we denote by $L^{2}(I)$ the
space of measurable real valued functions defined on $I$ which are
square integrable for the Lebesgue measure. For $k\geq1$, and $u,v\in\left(L^{2}(I)\right)^{k}$,
we use the notation
\[
(u,v)_{k}=\int_{I}\sum_{i=1}^{k}u_{i}(x)v_{i}(x)dx,\ |u|_{k}=(u,u)_{k}^{\frac{1}{2}}.
\]
and we define $L(I):=\left(L^{2}(I)\right)^{d}$. We also define $L:=L(\mathbb{R})$
and we denote by $L'$ its dual space. 
\item For $m\geq1$, we denote by $H^{m}(I)$ the Sobolev space of real
valued functions on $I$ that are square integrable together with
all their distribution derivatives up to the order $m$. We define
the space $H(I):=\left(H^{1}(I)\right)^{d}$, endowed with the usual
scalar product and norm
\[
\langle u,v\rangle_{d}=\int_{I}\sum_{i=1}^{d}\left(u_{i}(x)v_{i}(x)+\partial_{x}u_{i}(x)\partial_{x}v_{i}(x)\right)dx,\ ||u||_{d}=\langle u,u\rangle_{d}^{\frac{1}{2}},
\]
and we define $H:=H(\mathbb{R})$. We denote by $H'$ its dual space,
and by $\langle\cdot,\cdot\rangle$ the duality product between $H$
and $H'$.
\item For $1\leq p\leq\infty$, we denote by $W^{1,p}(\mathbb{R})$ the
Sobolev space of functions belonging to $L^{p}(\mathbb{R})$, and
that have a first order derivative in the sense of distributions that
also belongs to $L^{p}(\mathbb{R})$. 
\item For $n\geq1$, $\mathcal{O}$ an open subset of $\mathbb{R}^{n}$
and $0\leq k\leq\infty$, we denote by $C^{k}(\mathcal{O})$ the set
of functions $\mathcal{O}\rightarrow\mathbb{R}$ that are continuous
and have continuous derivatives up to the order $k$, we denote by
$C_{c}^{k}(\mathcal{O})$ the set of functions in $C^{k}(\mathcal{O})$
that have compact support on $\mathcal{O}$, and we denote by $C_{b}^{k}(\mathcal{O})$
the set of functions in $C^{k}(\mathcal{O})$ that are uniformly bounded
on $\mathcal{O}$ together with their $p$-th order derivatives, for
$p\leq k$. 
\item For $n\geq1$, we denote by $\mathcal{M}_{n}(\mathbb{R})$ the set
of $n\times n$ matrices with real-valued coefficients. We denote
by $I_{n}\in\mathcal{M}_{n}(\mathbb{R})$ the identity matrix, and
by $J_{n}\in\mathcal{M}_{n}(\mathbb{R})$ the matrix where all the
coefficients are equal to 1.
\item For $n\geq1$, we denote by $\mathcal{S}_{n}^{+}(\mathbb{R})$ (resp.
$\mathcal{S}_{n}^{++}(\mathbb{R})$) the set of $n\times n$ matrices
in $\mathcal{M}_{n}(\mathbb{R})$ which are symmetric and positive
semidefinite (resp. definite). For $S\in\mathcal{S}_{n}^{+}(\mathbb{R})$
(resp. $S\in\mathcal{S}_{n}^{++}(\mathbb{R})$) we denote by $\sqrt{S}$
the unique matrix in $\mathcal{S}_{n}^{+}(\mathbb{R})$ (resp. $\mathcal{S}_{n}^{++}(\mathbb{R})$)
such that $\sqrt{S}\sqrt{S}=S$, and we denote by $l_{min}(S)$ and
$l_{max}(S)$ respectively the smallest eigenvalue and the largest
eigenvalue of $S$.
\item Given $n,p\geq1$ and $\mathcal{A}$ a $n\times p$ matrix, we denote
its transpose by $\mathcal{A}^{*}$ and we define $||A||_{\infty}:=\max\{|\mathcal{A}_{ij}|,1\leq i\leq n,1\leq j\leq p\}$. 
\item For $y\in\mathbb{R}$, we denote its positive part by $y^{+}=\max(0,y)$
and its negative part by $y^{-}=\min(0,y)$. For $k\geq2$ and $x=\left(x_{1},...,x_{k}\right)\in\mathbb{R}^{k}$,
we denote its positive part by $x^{+}:=\left(x_{1}^{+},...,x_{k}^{+}\right)$,
and its negative part by $x^{-}:=\left(x_{1}^{-},...,x_{k}^{-}\right)$.
\item We denote by $\mathcal{P}(\mathbb{R})$ (resp. $\mathcal{P}(\mathbb{R}\times\mathcal{Y})$)
the set of probability measures on $\mathbb{R}$ (resp. $\mathbb{R}\times\mathcal{Y}$).
\item We define $\mathcal{D}:=\left(\mathbb{R}^{+}\right)^{d}\backslash(0,...,0)$. 
\item For notational simplicity, for a function $g$ defined on $\mathbb{R}^{2}$
and $t\in\mathbb{R}$, we may sometimes use the notation $g(t):=g(t,\cdot)$. 
\end{itemize}

\section{\label{sec:newClass}A new class of fake Brownian motions }

In this section, we give the main results concerning the SDE $(1)$.
We introduce a variational formulation to give sense to the PDS (FBM).
Let us assume that $p:=\left(p_{1},...,p_{d}\right)$ is a classical
solution to the PDS (FBM) such that $p$ takes values in $\mathcal{D}$
and for $t\in(0,T]$, $p(t,\cdot)\in H$. For $v:=(v_{1},...,v_{d})\in H$
and $i\in\{1,...,d\}$ we multiply (\ref{eq:toypds2pde}) by $v_{i}$
and integrate over $\mathbb{R}$. Through integration by parts, we
obtain that for $i\in\{1,...,d\}$, the following equality holds in
the classical sense and also in the sense of distributions on $(0,T)$:
\begin{equation}
\int_{\mathbb{R}}v_{i}(x)\partial_{t}p_{i}(t,x)dx=\frac{d}{dt}\int_{\mathbb{R}}v_{i}(x)p_{i}(t,x)dx=-\frac{1}{2}\int_{\mathbb{R}}\partial_{x}v_{i}\partial_{x}\left(\frac{\sum_{l=1}^{d}p_{l}(t,x)}{\sum_{l=1}^{d}\lambda_{l}p_{l}(t,x)}\lambda_{i}p_{i}(t,x)\right)dx.\label{eq:fvarindexi}
\end{equation}
The term $\partial_{x}\left(\frac{\sum_{l=1}^{d}p_{l}}{\sum_{l=1}^{d}\lambda_{l}p_{l}}\lambda_{i}p_{i}\right)$
rewrites
\begin{eqnarray*}
\partial_{x}\left(\frac{\sum_{l=1}^{d}p_{l}}{\sum_{l=1}^{d}\lambda_{l}p_{l}}\lambda_{i}p_{i}\right) & = & \left(1+\frac{\left(\sum_{l\neq i}\lambda_{l}p_{l}\right)\left(\sum_{l\neq i}(\lambda_{i}-\lambda_{l})p_{l}\right)}{\left(\sum_{l=1}^{d}\lambda_{l}p_{l}\right)^{2}}\right)\partial_{x}p_{i}+\sum_{j\neq i}\frac{\lambda_{i}p_{i}\left(\sum_{l\neq j}(\lambda_{l}-\lambda_{j})p_{l}\right)}{\left(\sum_{l=1}^{d}\lambda_{l}p_{l}\right)^{2}}\partial_{x}p_{j}.
\end{eqnarray*}
Let us introduce the function $M:\mathcal{D}\rightarrow\mathcal{M}_{d}(\mathbb{R})$
where for $\rho\in\mathcal{D}$, $M(\rho)$ is the matrix with coefficients:
\begin{eqnarray*}
M_{ii}(\rho) & = & \frac{\left(\sum_{l\neq i}\lambda_{l}\rho_{l}\right)\left(\sum_{l\neq i}(\lambda_{i}-\lambda_{l})\rho_{l}\right)}{\left(\sum_{l=1}^{d}\lambda_{l}\rho_{l}\right)^{2}},\ 1\leq i\leq d,\\
M_{ij}(\rho) & = & \frac{\lambda_{i}\rho_{i}\left(\sum_{l\neq j}(\lambda_{l}-\lambda_{j})\rho_{l}\right)}{\left(\sum_{l=1}^{d}\lambda_{l}\rho_{l}\right)^{2}},\ 1\leq i\neq j\leq d,
\end{eqnarray*}
We also define $A:\mathcal{D}\rightarrow\mathcal{M}_{d}(\mathbb{R})$,
where for $\rho\in\mathcal{D}$, 
\[
A(\rho)=\frac{1}{2}\left(I_{d}+M(\rho)\right).
\]
For $1\leq i\leq d$, we have that $\frac{1}{2}\partial_{x}\left(\frac{\sum_{l=1}^{d}p_{l}}{\sum_{l=1}^{d}\lambda_{l}p_{l}}\lambda_{i}p_{i}\right)=(A(p)\partial_{x}p)_{i}$.
Summing the equality (\ref{eq:fvarindexi}) over the index $1\leq i\leq d$,
the following equality holds in the sense of distributions on $(0,T)$,

\[
\forall v\in H,\ \frac{d}{dt}(v,p(t))_{d}+(\partial_{x}v,A(p(t))\partial_{x}p(t))_{d}=0.
\]
We denote by $p_{0}$ the measure $\mu$ multiplied by the vector
$(\alpha_{1},...,\alpha_{d})$ and introduce below a weak variational
formulation for the PDS (FBM).

\begin{equation}
\text{Find }p=(p_{1},...,p_{d})\ \text{satisfying:}\label{eq:find}
\end{equation}
\begin{equation}
p\in L_{loc}^{2}((0,T];H)\cap L_{loc}^{\infty}((0,T];L),\label{eq:pL2}
\end{equation}
\begin{equation}
p\text{ takes values in \ensuremath{\mathcal{D}}, a.e. on \ensuremath{(0,T)\times\mathbb{R}},}\label{eq:pinD}
\end{equation}

\begin{equation}
\forall v\in H,\ \frac{d}{dt}(v,p)_{d}+(\partial_{x}v,A(p)\partial_{x}p)_{d}=\text{0}\text{ in the sense of distributions on \ensuremath{(0,T)},}\label{eq:fvar}
\end{equation}

\begin{equation}
p(t,\cdot)\underset{t\rightarrow0}{\overset{\text{weakly-*}}{\rightarrow}}p_{0},\label{eq:cond0}
\end{equation}
where the last condition means that: 
\[
\forall v\in H,(v,p(t))_{d}\underset{t\rightarrow0}{\rightarrow}\sum_{i=1}^{d}\alpha_{i}\int_{\mathbb{R}}v_{i}(x)\mu(dx).
\]
We call $V(\mu)$ the problem defined by (\ref{eq:find})-(\ref{eq:cond0}).
If $p$ only satisfies (\ref{eq:pL2}), then the initial condition
(\ref{eq:cond0}) does not make sense. As we will see in the proof
of Theorem \ref{thm:FBMmuProba}, if $p$ satisfies the conditions
(\ref{eq:pL2})-(\ref{eq:fvar}), then $p$ is a.e. equal to a function
continuous from $(0,T]$ into $L$. We will always consider its continuous
representative and therefore (\ref{eq:cond0}) makes sense. The existence
results to $V(\mu)$ and to the SDE (\ref{eq:ToySDE}) are stated
in the following theorems.
\begin{thm}
\label{thm:FBMmuProba}Under Condition $\left(C\right)$, $V(\mu)$
has a solution $p\in C((0,T],L)$ such that for almost every $(t,x)$
in $(0,T]\times\mathbb{R}$, $\sum_{i=1}^{d}p_{i}(t,x)=(\mu*h_{t})(x)$.
\end{thm}

\begin{thm}
\label{thm:FBMweaksolution}Under Condition $\left(C\right)$, SDE
(\ref{eq:ToySDE}) has a weak solution, which has the same marginal
law as $\left(Z+W_{t}\right)_{t\geq0}$, where $Z$ is independent
from $\left(W_{t}\right)_{t\geq0}$ and $Z\sim\mu$.
\end{thm}
The proofs of Theorems \ref{thm:FBMmuProba} and \ref{thm:FBMweaksolution}
are postponed to the following section. To end this section, let us
remark that whenever the function $f$ is not constant on $\mathcal{Y}$,
the solution to the SDE (\ref{eq:ToySDE}) given by Theorem \ref{thm:FBMweaksolution}
is a continuous fake Brownian motion.

\begin{prop}
Under Condition $\left(C\right)$, if $f$ is not constant on $\mathcal{Y}$
then the solutions to SDE (\ref{eq:ToySDE}) with initial condition
$\mu=\delta_{0}$ are continuous fake Brownian motions. \end{prop}
\begin{proof}
Let $X$ be a solution to SDE (\ref{eq:ToySDE}). The term $\frac{f(Y)}{\sqrt{\mathbb{E}\left[f^{2}(Y)|X_{t}\right]}}$
is bounded, so $X$ is a continuous martingale and by \cite[Theorem 4.6]{Gyongy},
for $t\in[0,T]$, $X_{t}$ has the law $\mathcal{N}(0,t)$. Let us
remark that a solution to SDE (\ref{eq:ToySDE}) with the properties
stated by Theorem \ref{thm:FBMweaksolution} satisfies the marginal
constraints without using Gyongy's theorem. We consider the quadratic
variation $\langle X\rangle$ of $X$, which satisfies $d\langle X\rangle_{t}=\frac{f^{2}(Y)}{\mathbb{E}\left[f^{2}(Y)|X_{t}\right]}dt$.
If almost surely, $\forall t\geq0$, $\langle X\rangle_{t}=t$, then
a.s., $dt$-a.e., $\frac{f^{2}(Y)}{\mathbb{E}\left[f^{2}(Y)|X_{t}\right]}=1$
and $f^{2}(Y)$ is $\sigma\left(X_{t}\right)$ measurable. As moreover
$f$ is positive, $X_{t}=W_{t}$, and there exists $t\geq0$ such
that $f^{2}(Y)$ is a measurable function of $W_{t}$. As $Y$ is
independent from $\left(W_{t}\right)_{t\geq0}$, we have that $f^{2}(Y)$
is constant. Then by contraposition, if $f$ is not constant on $\mathcal{Y}$,
we do not have that for $t>0$, $\langle X\rangle_{t}=t$, so $\left(X_{t}\right)_{t\geq0}$
is a continuous fake Brownian motion. 
\end{proof}

\section{\label{sec:proofsnewclass}Proofs of Section 2}

Condition $(C)$, developed in Subsection \ref{sub:ConditionGamma},
enables to establish a priori energy estimates of solutions to $V\left(\mu\right)$,
computed in Subsection \ref{sub:squareintegrabledensity}, where we
give a stronger variational formulation to the PDS (FBM) under the
assumption that the initial distribution $\mu$ has a square integrable
density. Under Condition $(C)$, we prove existence to that formulation.
That result is extended in Subsection \ref{sub:FBMGeneralCase} to
prove Theorem \ref{thm:FBMmuProba}. In Subsection \ref{sub:WeakSolution},
we establish a link between $V(\mu)$ and the variational formulation
in the sense of distributions defined in \cite{Figalli}, and we prove
Theorem \ref{thm:FBMweaksolution}, thanks to \cite[Theorem 2.6]{Figalli}.

\subsection{\label{sub:ConditionGamma}Condition $\left(C\right)$}

Let us introduce the notion of uniform coercivity.
\begin{defn}
Given a domain $D\subset\mathbb{R}^{d}$, a function $G:D\rightarrow\mathcal{M}_{d}(\mathbb{R})$
is uniformly coercive on $D$ with a coefficient $c>0$ if: 
\[
\forall\rho\in D,\forall\xi\in\mathbb{R}^{d},\ \xi^{*}G(\rho)\xi\geq c\xi^{*}\xi.
\]

\end{defn}
First, let us notice that the function $A$ is bounded. 
\begin{lem}
\label{lem:Abounded}Any matrix $B$ in the image $A(\mathcal{D})$
of $\mathcal{D}$ by $A$ satisfies $||B||_{\infty}\leq\frac{1}{2}\left(1+\frac{\lambda_{max}}{\lambda_{min}}\right)$. \end{lem}
\begin{proof}
For $\rho\in\mathcal{D},A(\rho)=\frac{1}{2}(I_{d}+M(\rho))$. It is
sufficient to check that $||M(\rho)||_{\infty}\leq\frac{\lambda_{max}}{\lambda_{min}}$.
Since for $1\leq l\leq d$, $\lambda_{l}>0$, $\rho_{l}\geq0$ and
$\sum_{l=1}^{d}\lambda_{l}\rho_{l}>0$, we have for $1\leq i\neq j\leq d$,
\begin{eqnarray*}
-\frac{\lambda_{max}}{\lambda_{min}}\leq-\lambda_{j}\frac{\lambda_{i}\rho_{i}}{\left(\sum_{l}\lambda_{l}\rho_{l}\right)}\frac{\sum_{l\neq j}\rho_{l}}{\left(\sum_{l}\lambda_{l}\rho_{l}\right)} & \leq M_{ij}(\rho)\leq & \frac{\lambda_{i}\rho_{i}}{\left(\sum_{l}\lambda_{l}\rho_{l}\right)}\frac{\sum_{l\neq j}\lambda_{l}\rho_{l}}{\left(\sum_{l}\lambda_{l}\rho_{l}\right)}\leq1\leq\frac{\lambda_{max}}{\lambda_{min}},\\
-\frac{\lambda_{max}}{\lambda_{min}}\leq-1\leq\frac{\sum_{l\neq i}\lambda_{l}\rho_{l}\sum_{l\neq i}(-\lambda_{l})\rho_{l}}{\left(\sum_{l}\lambda_{l}\rho_{l}\right)^{2}} & \leq M_{ii}(\rho)\leq & \lambda_{i}\frac{\sum_{l\neq i}\lambda_{l}\rho_{l}}{\sum_{l}\lambda_{l}\rho_{l}}\frac{\sum_{l\neq i}\rho_{l}}{\left(\sum_{l}\lambda_{l}\rho_{l}\right)}\leq\frac{\lambda_{max}}{\lambda_{min}}.
\end{eqnarray*}

\end{proof}
The role of Condition $(C)$ is to ensure the existence of a matrix
$\Pi\in\mathcal{S}_{d}^{++}(\mathbb{R})$ such that $\Pi A$ is uniformly
coercive on $\mathcal{D}$. By a slight abuse of notation, we will
say that a matrix $\Gamma$ satisfies $(C)$ if $\Gamma\in\mathcal{S}_{d}^{++}(\mathbb{R})$,
and for $1\leq k\leq d$, the matrix $\Gamma^{(k)}$ defined by (\ref{eq:Gammakij})
is positive definite on $e_{k}^{\perp}$. We consider matrices with
the form $J_{d}+\epsilon\Gamma$, where $\Gamma\in S_{d}^{++}(\mathbb{R})$
and $\epsilon>0$, and we show that if $\Gamma$ satisfies $(C)$,
then for $\epsilon$ small enough, $\left(J_{d}+\epsilon\Gamma\right)A$
achieves the coercivity property. The intuition behind the choice
of $J_{d}$ is the observation that if $\left(p_{1},...,p_{d}\right)$
is a solution to $V(\mu)$ then $\sum_{i=1}^{d}p_{i}$ is a solution
to the heat equation. This also translates into the algebraic property
that $J_{d}M=0$. We define $\mathbf{1}:=(1,...,1)\in\mathbb{R}^{d}$.
\begin{prop}
\label{prop:coercivity1orthogonal}If $\Gamma\in\mathcal{S}_{d}^{++}\left(\mathbb{R}\right)$
satisfies Condition $(C)$, then there exists $z>0$ such that
\begin{equation}
\forall x\in\mathbf{1}^{\perp},\forall\rho\in\mathcal{D},x^{*}\Gamma A(\rho)x\geq zx^{*}x.\label{eq:coercivityperp}
\end{equation}
Moreover, for $\epsilon>0$ small enough, the function $\left(J_{d}+\epsilon\Gamma\right)A$
is uniformly coercive on $\mathcal{D}$.\end{prop}
\begin{proof}
For $\rho\in\mathcal{D}$, we introduce the notation $\overline{\lambda}(\rho):=\sum_{k}\lambda_{k}\frac{\rho_{k}}{\sum_{l}\rho_{l}}$
and remark that $\overline{\lambda}(\mathcal{D})=[\lambda_{min},\lambda_{max}]$.
The matrix $A(\rho)$ rewrites as the convex combination $A(\rho)=\sum_{k=1}^{d}w_{k}(\rho)A_{k}(\rho)$
with the weight $w_{k}(\rho)=\frac{\lambda_{k}\rho_{k}}{\sum_{l}\lambda_{l}\rho_{l}}$
of the matrix $A_{k}(\rho)$ which has non zero coefficients only
on the diagonal and the $k$-th row, and is defined by: 
\begin{eqnarray*}
\left(A_{k}(\rho)\right)_{ij} & = & \frac{1}{2}\left(1_{\{i=j\}}\frac{\lambda_{i}}{\overline{\lambda}(\rho)}+1_{\{i=k\}}\left(1-\frac{\lambda_{j}}{\overline{\lambda}(\rho)}\right)\right),\ 1\le i,j\leq d.
\end{eqnarray*}
We prove that for $1\leq k\leq d$, there exists $z_{k}>0$ such that
\begin{equation}
\forall x\in\mathbf{1}^{\perp},\forall\rho\in\mathcal{D},2x^{*}\overline{\lambda}(\rho)\Gamma A_{k}(\rho)x\geq z_{k}x^{*}x,\label{eq:GammaAk}
\end{equation}
and we can set $z=\underset{1\leq k\leq d}{\min}\frac{z_{k}}{2\lambda_{max}}>0$
to obtain (\ref{eq:coercivityperp}). We have that 
\[
2\overline{\lambda}(\rho)\Gamma A_{k}(\rho)=\left(\lambda_{j}\left(\Gamma_{ij}-\Gamma_{ik}\right)\right)_{1\le i,j\leq d}+\left(\overline{\lambda}(\rho)\left(\Gamma_{ik}\right)\right)_{1\leq i,j\leq d}
\]
For $x\in\mathbf{1}^{\perp}$, we have
\begin{eqnarray}
2x^{*}\overline{\lambda}(\rho)\Gamma A_{k}(\rho)x & = & \sum_{i,j=1}^{d}\lambda_{j}\left(\Gamma_{ij}-\Gamma_{ik}\right)x_{i}x_{j}+\sum_{i,j=1}^{d}\overline{\lambda}(\rho)\left(\Gamma_{ik}\right)x_{i}x_{j}\label{eq:computationGammaGammatilde1}\\
 & = & \sum_{i\neq k,j\neq k}\lambda_{j}\left(\Gamma_{ij}-\Gamma_{ik}-\Gamma_{kj}+\Gamma_{kk}\right)x_{i}x_{j}\\
 & = & \sum_{i\neq k,j\neq k}\frac{\lambda_{i}+\lambda_{j}}{2}\left(\Gamma_{ij}-\Gamma_{ik}-\Gamma_{kj}+\Gamma_{kk}\right)x_{i}x_{j},\label{eq:terme3}
\end{eqnarray}
where between the first and the second equality, we used the fact
that $\sum_{i=1}^{d}x_{i}=0$ and then replaced $x_{k}$ by $-\sum_{i\neq k}x_{i}$.
As $\Gamma^{(k)}$ is positive definite on $e_{k}^{\perp}$, there
exists $\tilde{z}_{k}>0$ such that 
\begin{eqnarray*}
2x^{*}\overline{\lambda}(\rho)\Gamma A_{k}(\rho)x & \geq & \tilde{z}_{k}\sum_{i\neq k}x_{i}^{2}.
\end{eqnarray*}
 By Cauchy-Schwarz inequality, $x_{k}^{2}=\left(\sum_{i\neq k}x_{i}\right)^{2}\leq\left(d-1\right)\sum_{i\neq k}x_{i}^{2}$,
so $\sum_{i\neq k}x_{i}^{2}\geq\frac{1}{d}x^{*}x$ and we can set
$z_{k}=\frac{\tilde{z}_{k}}{d}$ to satisfy (\ref{eq:GammaAk}).

Now, we show that for $\epsilon>0$ small enough, the function $\left(J_{d}+\epsilon\Gamma\right)A$
is uniformly coercive on $\mathcal{D}$. For $\rho\in\mathcal{D}$,
as $A(\rho)=\frac{1}{2}\left(I_{d}+M(\rho)\right)$ and for $1\leq j\leq d,$
as it is easy to check that $\sum_{i=1}^{d}M_{ij}(\rho)=0$, we have
that $J_{d}M=0$ and
\[
\left(J_{d}+\epsilon\Gamma\right)A(\rho)=\frac{1}{2}J_{d}+\epsilon\Gamma A(\rho).
\]
For $x\in\mathbb{R}^{d}$, we decompose $x=u+v$ where $u\in\mathbb{R}\mathbf{1}$
and $v\in\mathbf{1}^{\perp}$. As $J_{d}v=0$, we have 
\begin{eqnarray*}
x^{*}\left(\frac{1}{2}J_{d}+\epsilon\Gamma A(\rho)\right)x & = & u^{*}\left(\frac{1}{2}J_{d}+\epsilon\Gamma A(\rho)\right)u+\epsilon v^{*}\Gamma A(\rho)v+\epsilon u^{*}\Gamma A(\rho)v+\epsilon v^{*}\Gamma A(\rho)u.
\end{eqnarray*}
We will use Young's inequality: 
\[
\forall\eta>0,\forall a,b\in\mathbb{R},\ ab\leq\eta a^{2}+\frac{1}{4\eta}b^{2}.
\]
For $1\leq i,j\leq d$, by Young's inequality and Lemma \ref{lem:Abounded},
$\left(\Gamma A(\rho)\right)_{ij}=\sum_{k=1}^{d}\Gamma_{ik}A(\rho)_{kj}\leq\frac{d}{2}||\Gamma||_{\infty}\left(1+\frac{\lambda_{max}}{\lambda_{min}}\right)$,
so for $a,b\in\mathbb{R}$, and $\eta>0$,
\[
\left(\Gamma A(\rho)\right)_{ij}ab\geq-\frac{d}{2}||\Gamma||_{\infty}\left(1+\frac{\lambda_{max}}{\lambda_{min}}\right)\left(\eta a^{2}+\frac{1}{4\eta}b^{2}\right).
\]
We obtain, for $\tilde{x},\tilde{y}\in\mathbb{R}^{d},$ and $\eta>0$,
\[
\tilde{x}^{*}\Gamma A(\rho)\tilde{y}=\sum_{i,j=1}^{d}\left(\Gamma A(\rho)\right)_{ij}\tilde{x}_{i}\tilde{y}_{j}\geq-\frac{d^{2}}{2}||\Gamma||_{\infty}\left(1+\frac{\lambda_{max}}{\lambda_{min}}\right)\left(\eta\tilde{x}^{*}\tilde{x}+\frac{1}{4\eta}\tilde{y}^{*}\tilde{y}\right).
\]
Then for $\eta>0$, 
\begin{eqnarray*}
v^{*}\Gamma A(\rho)u & \geq & -\frac{d^{2}}{2}||\Gamma||_{\infty}\left(1+\frac{\lambda_{max}}{\lambda_{min}}\right)\left(\eta v^{*}v+\frac{1}{4\eta}u^{*}u\right),\\
u^{*}\Gamma A(\rho)v & \geq & -\frac{d^{2}}{2}||\Gamma||_{\infty}\left(1+\frac{\lambda_{max}}{\lambda_{min}}\right)\left(\eta v^{*}v+\frac{1}{4\eta}u^{*}u\right),
\end{eqnarray*}
and moreover, 
\[
u^{*}\Gamma A(\rho)u\geq-\frac{d^{2}}{2}||\Gamma||_{\infty}\left(1+\frac{\lambda_{max}}{\lambda_{min}}\right)u^{*}u.
\]
Using (\ref{eq:coercivityperp}) and $J_{d}u=du$, we get that for
$\eta>0$, 
\[
x^{*}\left(\frac{1}{2}J_{d}+\epsilon\Gamma A(\rho)\right)x\geq\left(\frac{d}{2}-\epsilon\frac{d^{2}}{2}||\Gamma||_{\infty}\left(1+\frac{\lambda_{max}}{\lambda_{min}}\right)\left(1+\frac{1}{2\eta}\right)\right)u^{*}u+\epsilon\left(z-d^{2}||\Gamma||_{\infty}\left(1+\frac{\lambda_{max}}{\lambda_{min}}\right)\eta\right)v^{*}v.
\]
For $0<\epsilon<\left(d||\Gamma||_{\infty}\left(1+\frac{\lambda_{max}}{\lambda_{min}}\right)\left(1+\frac{d^{2}||\Gamma||_{\infty}\left(1+\frac{\lambda_{max}}{\lambda_{min}}\right)}{2z}\right)\right)^{-1}$,
we check that 
\[
\eta_{1}:=\left(\frac{2}{\epsilon d||\Gamma||_{\infty}\left(1+\frac{\lambda_{max}}{\lambda_{min}}\right)}-2\right)^{-1}<\frac{z}{d^{2}||\Gamma||_{\infty}\left(1+\frac{\lambda_{max}}{\lambda_{min}}\right)}=:\eta_{2},
\]
and with the choice $\eta\in(\eta_{1},\eta_{2})$, we see that the
function $\left(J_{d}+\epsilon\Gamma\right)A$ is uniformly coercive
on $\mathcal{D}$.\end{proof}
\begin{cor}
\label{cor:PiA}Condition $(C)$ is equivalent to the existence of
a matrix $\Pi\in\mathcal{S}_{d}^{++}\left(\mathbb{R}\right)$ such
that the function $\text{\ensuremath{\Pi}}A$ is uniformly coercive
on $\mathcal{D}$ with a coefficient $\kappa\in\left(0,\frac{l_{min}(\Pi)}{2}\right)$.\end{cor}
\begin{proof}
If a matrix $\Gamma\in\mathcal{S}_{d}^{++}\left(\mathbb{R}\right)$
satisfies $(C)$, then by Proposition \ref{prop:coercivity1orthogonal},
for $\epsilon>0$ small enough, the function $(J_{d}+\epsilon\Gamma)A$
is uniformly coercive on $\mathcal{D}$. Moreover, $(J_{d}+\epsilon\Gamma)\in\mathcal{S}_{d}^{++}(\mathbb{R})$
as $J_{d}\in\mathcal{S}_{d}^{+}(\mathbb{R})$.

Conversely, if $\Pi\in\mathcal{S}_{d}^{++}\left(\mathbb{R}\right)$
is such that the function $\text{\ensuremath{\Pi}}A$ is uniformly
coercive on $\mathcal{D}$ with a coefficient $c>0$, with the same
computation as (\ref{eq:computationGammaGammatilde1})-(\ref{eq:terme3}),
we have that for $x\in\mathbf{1}^{\perp}$ and $1\leq k\leq d$,
\[
\sum_{i\neq k,j\neq k}\frac{\lambda_{i}+\lambda_{j}}{2}\left(\Pi_{ij}-\Pi_{ik}-\Pi_{kj}+\Pi_{kk}\right)x_{i}x_{j}=2x^{*}\overline{\lambda}(\rho)\Pi A_{k}(\rho)x\geq2\lambda_{min}c\sum_{i\neq k}x_{i}^{2},
\]
so $\Pi$ satisfies $(C)$. To conclude, it is obvious that if $\Pi A$
is uniformly coercive, then $\Pi A$ is uniformly coercive with a
coefficient $\kappa\leq\frac{l_{min}(\Pi)}{2}$.
\end{proof}
Making Condition $(C)$ explicit does not seem to be an easy task
in general, but we give here simpler criteriae which are all proved
in Appendix \ref{sub:ConditionC}, for particular situations.

\begin{itemize}
\item For $d=2$, Condition $(C)$ is satisfied, for the choice $\Gamma=I_{2}$.
\item For $d=3$, we define 
\[
r_{1}=\frac{\lambda_{3}}{\lambda_{2}}+\frac{\lambda_{2}}{\lambda_{3}}\geq2,\ r_{2}=\frac{\lambda_{3}}{\lambda_{1}}+\frac{\lambda_{1}}{\lambda_{3}}\ge2,\ r_{3}=\frac{\lambda_{1}}{\lambda_{2}}+\frac{\lambda_{2}}{\lambda_{1}}\geq2.
\]
 Condition $(C)$ is satisfied if and only if 
\[
\frac{1}{\sqrt{\left(r_{1}-2\right)\left(r_{2}-2\right)}}+\frac{1}{\sqrt{\left(r_{2}-2\right)\left(r_{3}-2\right)}}+\frac{1}{\sqrt{\left(r_{1}-2\right)\left(r_{3}-2\right)}}>\frac{1}{4},
\]
with the convention that $\frac{1}{0}=+\infty$.
\item For $d\geq4$, we give in Appendix \ref{sub:ConditionC} a numerical
procedure to check if there exists a diagonal matrix that satisfies
Condition $(C)$.
\item For $d\geq4$, if 
\[
\underset{1\leq k\leq d}{\max}\ \sqrt{\sum_{i\neq k}\lambda_{i}\sum_{i\neq k}\frac{1}{\lambda_{i}}}<d+1,
\]
then Condition $(C)$ is satisfied, for the choice $\Gamma=I_{d}$.
\end{itemize}

\subsection{\label{sub:squareintegrabledensity}$\mu$ has a square integrable
density}

We first suppose that the measure $\mu$ has a square integrable density
with respect to the Lebesgue measure. In this case, we denote by $p_{0}$,
the element of $L$ obtained by multiplication of that density by
the vector $(\alpha_{1},...,\alpha_{d})$, and we define the stronger
variational formulation:

\begin{eqnarray}
\text{Find} &  & p=(p_{1},...,p_{d})\ \text{satisfying :}\label{eq:ToyFVfind}
\end{eqnarray}
\begin{equation}
p\in L^{2}([0,T];H)\cap L^{\infty}([0,T];L),\label{eq:ToyFVregularity}
\end{equation}

\begin{equation}
p\text{ takes values in \ensuremath{\mathcal{D}}, a.e. on \ensuremath{[0,T]\times\mathbb{R}}},\label{eq:ToyFVvalueinD}
\end{equation}

\begin{equation}
\forall v\in H,\ \frac{d}{dt}(v,p)_{d}+(\partial_{x}v,A(p)\partial_{x}p)_{d}=\text{0}\text{ in the sense of distributions on \ensuremath{(0,T)},}\label{eq:ToyFVdistribution}
\end{equation}

\begin{equation}
p_{i}(0,\cdot)=p_{0,i}\ ,\forall i\in\{1,...,d\}.\label{eq:ToyFVinit}
\end{equation}
We call $V_{L^{2}}(\mu)$ the problem defined by (\ref{eq:ToyFVfind})-(\ref{eq:ToyFVinit}).
If $p$ only satisfies (\ref{eq:ToyFVregularity}), then the initial
condition (\ref{eq:ToyFVinit}) does not make sense. We will show
that if $p$ satisfies the conditions (\ref{eq:ToyFVregularity})-(\ref{eq:ToyFVdistribution}),
then $p$ is a.e. equal to a function continuous from $[0,T]$ into
$L$. We consider in what follows this continuous representative and
therefore the initial condition (\ref{eq:ToyFVinit}) makes sense.
We now present an existence result to $V_{L^{2}}(\mu)$.
\begin{thm}
\label{thm:ThmL2}Under Condition (C), $V_{L^{2}}(\mu)$ has a solution
$p\in C([0,T],L)$ such that for almost every $(t,x)$ in $[0,T]\times\mathbb{R}$,
$\sum_{i=1}^{d}p_{i}(t,x)=(\mu*h_{t})(x)=\int_{\mathbb{R}}h_{t}(x-y)\mu(y)dy$.
\end{thm}
To prove Theorem \ref{thm:ThmL2}, we apply Galerkin's procedure,
as in \cite[III. 1.3]{temam}. Since $H$ is a separable Hilbert space,
there exists a sequence $(w_{k})_{k\geq1}=(\left(w_{k1},...,w_{kd}\right))_{k\geq1}$
of linearly independent elements which is total in $H$. It is not
at all obvious to preserve the condition that $p$ takes values in
$\mathcal{D}$ a.e. on $[0,T]\times\mathbb{R}$ at the discrete level.
That is why, for $\epsilon>0$, we introduce for $\rho\in\left(\mathbb{R}_{+}\right)^{d}$
the approximation $M_{\epsilon}$ of $M$ defined on $\left(\mathbb{R}_{+}\right)^{d}$
by 
\begin{eqnarray*}
M_{\epsilon,ii}(\rho) & = & \frac{\sum_{l\neq i}\lambda_{l}\rho_{l}\sum_{l\neq i}(\lambda_{i}-\lambda_{l})\rho_{l}}{\epsilon^{2}\vee\left(\sum_{l}\lambda_{l}\rho_{l}\right)^{2}},\ 1\leq i\leq d,\\
M_{\epsilon,ij}(\rho) & = & \frac{\lambda_{i}\rho_{i}\sum_{l\neq j}(\lambda_{l}-\lambda_{j})\rho_{l}}{\epsilon^{2}\vee\left(\sum_{l}\lambda_{l}\rho_{l}\right)^{2}},\ 1\leq i\neq j\leq d.
\end{eqnarray*}
We introduce the approximation $A_{\epsilon}$ of $A$ defined on
$\left(\mathbb{R}_{+}\right)^{d}$ by $A_{\epsilon}=\frac{1}{2}(I_{d}+M_{\epsilon})$,
and the approximate variational formulation $V_{\epsilon}(\mu)$ defined
by

\begin{eqnarray}
\text{Find} &  & p_{\epsilon}=(p_{\epsilon,1},...,p_{\epsilon,d})\ \text{satisfying :}\label{eq:ToyEps1Find}
\end{eqnarray}
\begin{equation}
p_{\epsilon}\in L^{2}([0,T];H)\cap L^{\infty}([0,T];L)\label{eq:ToyEps2regularity}
\end{equation}

\begin{equation}
\forall v\in H,\ \frac{d}{dt}(v,p_{\epsilon})_{d}+(\partial_{x}v,A_{\epsilon}(p_{\epsilon}^{+})\partial_{x}p_{\epsilon})_{d}=\text{0}\text{ in the sense of distributions on \ensuremath{(0,T)},}\label{eq:ToyEps2distribution}
\end{equation}

\begin{equation}
p_{\epsilon,i}(0,\cdot)=p_{0,i},\ \forall i\in\{1,...,d\}\label{eq:ToyEps2init}
\end{equation}
For the same reasons as for the variational formulation $V_{L^{2}}(\mu)$,
any solution of $V_{\epsilon}(\mu)$ has a continuous representative
in $C([0,T],L)$, and therefore (\ref{eq:ToyEps2init}) makes sense. 

We first prove existence to $V_{\epsilon}(\mu)$ by Galerkin's procedure.
Then we will check that a.e. on $[0,T]\times\mathbb{R}$, $p_{\epsilon}$
takes values in $\left(\mathbb{R}_{+}\right)^{d}$ . We will then
obtain existence to $V_{L^{2}}(\mu)$ by taking the limit $\epsilon\rightarrow0$.
In what follows, whenever Condition $(C)$ is satisfied, we denote
by $\Pi$ an element of $\mathcal{S}_{d}^{++}\left(\mathbb{R}\right)$
such that $\Pi A$ is uniformly coercive on $\mathcal{D}$ with a
coefficient $\kappa\in\left(0,\frac{l_{min}(\Pi)}{2}\right)$, and
both exist by Corollary \ref{cor:PiA}. Let us remark that as the
function $v\rightarrow\Pi v$ is a bijection from $H$ to $H$ and
$\Pi$ is symmetric, Equality (\ref{eq:ToyEps2distribution}) is equivalent
to: 
\begin{equation}
\forall v\in H,\ \frac{d}{dt}(v,\Pi p_{\epsilon})_{d}+(\partial_{x}v,\Pi A_{\epsilon}(p_{\epsilon}^{+})\partial_{x}p_{\epsilon})_{d}=\text{0},\text{ in the sense of distributions on \ensuremath{(0,T)}. }\label{eq:formulationPidiscrete}
\end{equation}
This formulation will help us take advantage of the coercivity property
of the function $\Pi A$. For $m\in\mathbb{N}^{*}$, we denote by
$p_{0}^{m}$ the orthogonal projection of $p_{0}$ onto the subspace
of $L$ spanned by $(w_{1},...,w_{m})$. We first solve an approximate
formulation named $V_{\epsilon}^{m}(\mu)$:
\begin{equation}
\text{Find \ensuremath{g_{\epsilon,1}^{m},...,g_{\epsilon,m}^{m}}}\in C^{1}([0,T],\mathbb{R}),\mbox{\text{ such that \text{the function \ensuremath{t\in[0,T]\rightarrow p_{\epsilon}^{m}(t):=\sum_{j=1}^{m}g_{\epsilon,j}^{m}(t)w_{j}} satisfies: }}}\label{eq:ToyEpsM1Find}
\end{equation}

\begin{equation}
\ensuremath{\forall i\in\{1,...,m\}},\ \frac{d}{dt}\left(w_{i},\Pi p_{\epsilon}^{m}(t)\right)_{d}+\left(\partial_{x}w_{i},\Pi A_{\epsilon}\left(\left(p_{\epsilon}^{m}\right)^{+}(t)\right)\partial_{x}p_{\epsilon}^{m}(t)\right)_{d}=0,\label{eq:ToyEpsiMode}
\end{equation}

\begin{equation}
p_{\epsilon}^{m}(0)=p_{0}^{m}.\label{eq:ToyEpsiMinit}
\end{equation}
We denote by $W^{(m)}\in\mathcal{S}_{m}^{++}(\mathbb{R})$ the non
singular Gram matrix of the linearly independent family $\left(\sqrt{\Pi}w_{i}\right)_{1\leq i\leq m}$,
with coefficients $W_{ij}^{(m)}=\left(w_{i},\Pi w_{j}\right)_{d}$
for $1\leq i,j\leq m$. We introduce $g_{\epsilon,0}^{m}:=(g_{01}^{m},...,g_{0m}^{m})\in\mathbb{R}^{m}$,
which is $p_{0}^{m}$ expressed on the basis $(w_{1},...,w_{m})$,
and $g_{\epsilon}^{m}:=\left(\ensuremath{g_{\epsilon,1}^{m},...,g_{\epsilon,m}^{m}}\right)$.
We define the function $K_{\epsilon}^{m}:\mathbb{R}^{m}\rightarrow\mathcal{M}_{m}(\mathbb{R})$
such that for $z\in\mathbb{R}^{m}$, $K_{\epsilon}^{m}(z)$ is the
matrix with coefficients 
\[
K_{\epsilon}^{m}(z)_{ij}=\left(\partial_{x}w_{i},\Pi A_{\epsilon}\left(\left(\sum_{k=1}^{m}z_{k}w_{k}\right)^{+}\right)\partial_{x}w_{j}\right)_{d},
\]
for $1\leq i,j\leq m$ and we define the function $F_{\epsilon}^{m}$
on $\mathbb{R}^{m}$ by: 
\[
F_{\epsilon}^{m}(z)=-\left(W^{(m)}\right)^{-1}K_{\epsilon}^{m}(z)z.
\]
Solving $V_{\epsilon}^{m}(\mu)$ becomes equivalent to solving the
following ODE for $g_{\epsilon}^{m}$: 
\begin{eqnarray}
\left(g_{\epsilon}^{m}\right)'(t) & = & F_{\epsilon}^{m}(g_{\epsilon}^{m})\label{eq:ToyODE1}\\
g_{\epsilon}^{m}(0) & = & g_{\epsilon,0}^{m}.\label{eq:ToyODE2}
\end{eqnarray}
To show existence of a unique solution to $(V_{\epsilon}^{m}(\mu))$,
for $m\geq1$, we check that the function $F_{\epsilon}^{m}$ is locally
Lipschitz and that the function $A_{\epsilon}$ is bounded. The proof
of Lemma \ref{lem:FespLocalLipschitz} below is postponed to Appendix
\ref{sec:AppendixA}.
\begin{lem}
\label{lem:FespLocalLipschitz}For $m\geq1$, the function $z\in\mathbb{R}^{m}\rightarrow F_{\epsilon}^{m}(z)$
is locally Lipschitz.
\end{lem}

\begin{lem}
\label{lem:Bepsbounded}For $\epsilon>0$ and $B\in A_{\epsilon}\left(\left(\mathbb{R}^{+}\right)^{d}\right)$,
$||B||_{\infty}\leq\frac{1}{2}\left(1+\frac{\lambda_{max}}{\lambda_{min}}\right)$. \end{lem}
\begin{proof}
For $\epsilon>0$, $A_{\epsilon}(0)=\frac{1}{2}I_{d}$, so $||A_{\epsilon}(0)||_{\infty}=\frac{1}{2}$.
For $\rho\in\mathcal{D}$ and $\epsilon>0$, it is clear that $||M_{\epsilon}(\rho)||_{\infty}\leq||M(\rho)||_{\infty}\leq\frac{\lambda_{max}}{\lambda_{min}}$,
where the inequality on the r.h.s comes from the proof of Lemma \ref{lem:Abounded},
so $||A_{\epsilon}(\rho)||_{\infty}\leq\frac{1}{2}\left(1+\frac{\lambda_{max}}{\lambda_{min}}\right)$.\end{proof}
\begin{lem}
\label{lem:existencetoODEs}For every $m\geq1$, there exists a unique
solution to $(V_{\epsilon}^{m}(\mu))$.\end{lem}
\begin{proof}
For $m\geq1$, we use Lemma \ref{lem:FespLocalLipschitz} and the
Cauchy-Lipschitz theorem to get existence and uniqueness of a maximal
solution $g_{\epsilon}^{m}$ on the interval $[0,T^{*})$ for a certain
$T^{*}>0$. It is sufficient to show that $T^{*}>T$, to ensure that
$g_{\epsilon}^{m}$ is defined on $[0,T]$. As $A_{\epsilon}$ is
uniformly bounded by $\frac{1}{2}\left(1+\frac{\lambda_{max}}{\lambda_{min}}\right)$
by Lemma \ref{lem:Bepsbounded}, we have that all the coefficients
of $\left(W^{(m)}\right)^{-1}K_{\epsilon}^{m}$ are uniformly bounded
by 
\[
\gamma:=\frac{d}{2}\left(1+\frac{\lambda_{max}}{\lambda_{min}}\right)\left|\left|\left(W^{(m)}\right)^{-1}\right|\right|_{\infty}||\Pi||_{\infty}\underset{1\leq j\leq m}{\max}\left(\sum_{i=1}^{m}\sum_{a,b=1}^{d}\int_{\mathbb{R}}|\partial_{x}w_{ia}||\partial_{x}w_{jb}|dx\right).
\]
For $1\leq i\leq m$, 
\[
\left|\left(g_{\epsilon,i}^{m}\right)'(t)\right|=\left|\sum_{j=1}^{m}\left(\left(W^{(m)}\right)^{-1}K_{\epsilon}^{m}(g_{\epsilon}^{m}(t))\right)_{ij}g_{\epsilon,j}^{m}(t)\right|\leq\gamma\sum_{j=1}^{m}|g_{\epsilon,j}^{m}(t)|.
\]
Summing over the index $i\in\{1,...,d\}$, we have for $t\in[0,T^{*})$
:
\[
\sum_{i=1}^{m}|g_{\epsilon,i}^{m}(t)|\leq\sum_{i=1}^{m}|g_{\epsilon,i}^{m}(0)|+\sum_{i=1}^{m}\int_{0}^{t}|\left(g_{\epsilon,i}^{m}\right)'(t)|dt\leq\sum_{i=1}^{m}|g_{\epsilon,i}^{m}(0)|+m\gamma\int_{0}^{t}\left(\sum_{i=1}^{d}|g_{\epsilon,i}^{m}(t)|\right)dt,
\]
and by Gronwall's lemma, for $0\leq t<T^{*}$, $\left(\sum_{i=1}^{m}|g_{\epsilon,i}^{m}(t)|\right)\leq\left(\sum_{i=1}^{m}|g_{\epsilon,i}^{m}(0)|\right)\exp(m\gamma t)$.
If $T^{*}<\infty$, the function $t\rightarrow\left(\sum_{i=1}^{m}|g_{\epsilon,i}^{m}(t)|\right)$
would explode as $t\underset{}{\rightarrow}T^{*}$. We conclude that
$T^{*}=\infty$, $g_{\epsilon}^{m}$ is defined on $[0,T]$ and $p_{\epsilon}^{m}=\sum_{i=1}^{m}g_{\epsilon,i}^{m}(t)w_{i}$
is the solution to $V_{\epsilon}^{m}(\mu)$.
\end{proof}
Before showing the existence of a converging subsequence of $\left(p_{\epsilon}^{m}\right)_{m\geq1}$
whose limit is a solution of $V_{\epsilon}(\mu)$, we check that $\Pi A_{\epsilon}$
is uniformly coercive on $\mathcal{D}$, uniformly in $\epsilon>0$.
\begin{lem}
\label{lem:AcoerciveimpliesAepscoercive}If $\Pi A$ is uniformly
coercive on $\mathcal{D}$ with coefficient $\kappa\in\left(0,\frac{l_{min}(\Pi)}{2}\right)$,
then for $\epsilon>0$, $\Pi A_{\epsilon}$ is uniformly coercive
on $\left(\mathbb{R}^{+}\right)^{d}$ with coefficient $\kappa$.\end{lem}
\begin{proof}
For $\epsilon>0,A_{\epsilon}(0)=\frac{1}{2}I_{d}$ so $\forall\xi\in\mathbb{R}^{d},\xi^{*}\Pi A_{\epsilon}(0)\xi=\frac{1}{2}\xi^{*}\Pi\xi\geq\frac{l_{min}(\Pi)}{2}\xi^{*}\xi\geq\kappa\xi^{*}\xi$.
For $\rho\in\mathcal{D}$, if $\epsilon\leq\sum_{l}\lambda_{l}\rho_{l}$,
then $A_{\epsilon}(\rho)=A(\rho)$ and by hypothesis $\forall\xi\in\mathbb{R}^{d},\xi^{*}\Pi A_{\epsilon}(\rho)\xi\geq\kappa\xi^{*}\xi$.
If $\epsilon>\sum_{l}\lambda_{l}p_{l}$, then for $1\leq i,j\leq d$,
$M_{\epsilon,ij}(\rho)=\left(\frac{1}{\epsilon}\sum_{l}\lambda_{l}\rho_{l}\right)^{2}M_{ij}(\rho),$
with $\left(\frac{1}{\epsilon}\sum_{l}\lambda_{l}\rho_{l}\right)^{2}<1$.
If $\xi^{*}\Pi M(\rho)\xi\leq0$, then $\xi^{*}\Pi M_{\epsilon}(\rho)\xi\geq\xi^{*}\Pi M(\rho)\xi$
and $\xi^{*}\Pi A_{\epsilon}(\rho)\xi\geq\xi^{*}\Pi A(\rho)\xi\geq\kappa\xi^{*}\xi$.
If $\xi^{*}\Pi M(\rho)\xi>0$, then $\xi^{*}\Pi M_{\epsilon}(\rho)\xi\geq0$
and $\xi^{*}\Pi A_{\epsilon}(\rho)\xi\geq\frac{1}{2}\xi^{*}\Pi\xi\geq\kappa\xi^{*}\xi$,
so $\Pi A_{\epsilon}$ is uniformly coercive on $\left(\mathbb{R}^{+}\right)^{d}$
with coefficient $\kappa$.
\end{proof}
We now state an existence result for $(V_{\epsilon}(\mu))$. 
\begin{prop}
\label{prop:ExistencetoVeps}Under Condition $(C)$, for $\epsilon>0$,
there exists a solution $p_{\epsilon}\in C([0,T],L)$ to $(V_{\epsilon}(\mu))$.\end{prop}
\begin{proof}
We compute standard energy estimates of $p_{\epsilon}^{m}$ for $m\geq1$.
We multiply (\ref{eq:ToyEpsiMode}) by $g_{\epsilon,i}^{m}(t)$ and
add these equations for $i=1,...,m$. This gives 
\begin{equation}
\frac{1}{2}\frac{d}{dt}\left|\sqrt{\Pi}p_{\epsilon}^{m}\right|_{d}^{2}+\left(\partial_{x}p_{\epsilon}^{m}(t),\Pi A_{\epsilon}\left(\left(p_{\epsilon}^{m}\right)^{+}(t)\right)\partial_{x}p_{\epsilon}^{m}(t)\right)_{d}=0.
\end{equation}
By the uniform coercivity of the function $\Pi A_{\epsilon}$, we
obtain that
\[
-\frac{1}{2}\frac{d}{dt}\left|\sqrt{\Pi}p_{\epsilon}^{m}\right|_{d}^{2}=\left(\partial_{x}p_{\epsilon}^{m}(t),\Pi A_{\epsilon}\left(\left(p_{\epsilon}^{m}\right)^{+}(t)\right)\partial_{x}p_{\epsilon}^{m}(t)\right)_{d}\geq\kappa\left|\partial_{x}p_{\epsilon}^{m}(t)\right|_{d}^{2}\geq0.
\]
Therefore, we get the following inequalities:

\begin{equation}
l_{min}(\Pi)\underset{t\in[0,T]}{\sup}|p_{\epsilon}^{m}(t)|_{d}^{2}\leq\underset{t\in[0,T]}{\sup}|\sqrt{\Pi}p_{\epsilon}^{m}(t)|_{d}^{2}\leq|\sqrt{\Pi}p_{\epsilon}^{m}(0)|_{d}^{2}\leq l_{max}(\Pi)|p_{\epsilon}^{m}(0)|_{d}^{2}\leq l_{max}(\Pi)|p_{0}|_{d}^{2},\label{eq:energyestimateLinfty}
\end{equation}

\begin{equation}
\int_{0}^{T}\kappa|\partial_{x}p_{\epsilon}^{m}(t)|_{d}^{2}dt\leq\frac{1}{2}\left(|\sqrt{\Pi}p_{\epsilon}^{m}(0)|_{d}^{2}-|\sqrt{\Pi}p_{\epsilon}^{m}(T)|_{d}^{2}\right)\leq\frac{l_{max}(\Pi)}{2}|p_{0}|_{d}^{2},\label{eq:energyestimatederivative}
\end{equation}
\begin{equation}
\int_{0}^{T}||p_{\epsilon}^{m}(t)||_{d}^{2}dt=\int_{0}^{T}|p_{\epsilon}^{m}(t)|_{d}^{2}dt+\int_{0}^{T}|\partial_{x}p_{\epsilon}^{m}(t)|_{d}^{2}dt\leq\left(T\frac{l_{max}(\Pi)}{l_{min}(\Pi)}+\frac{l_{max}(\Pi)}{2\kappa}\right)|p_{0}|_{d}^{2}.\label{eq:energyestimateH1}
\end{equation}
We see that the sequence $\left(p_{\epsilon}^{m}\right)_{m\geq1}$
remains in a bounded set of $L^{2}([0,T];H)\cap L^{\infty}([0,T];L)$,
so there exists an element $p_{\epsilon}\in L^{2}([0,T];H)\cap L^{\infty}([0,T];L)$
and a subsequence, for notational simplicity also called $\left(p_{\epsilon}^{m}\right)_{m\geq1}$,
that has the following convergence:
\begin{eqnarray*}
p_{\epsilon}^{m} & \underset{m\rightarrow\infty}{\rightarrow} & p_{\epsilon}\text{ {in} }L^{2}([0,T];H)\text{ weakly}\\
p_{\epsilon}^{m} & \underset{m\rightarrow\infty}{\rightarrow} & p_{\epsilon}\text{ {in} }L^{\infty}([0,T];L)\text{ weakly-*}.
\end{eqnarray*}
We now show that there exists a subsequence of $\left(p_{\epsilon}^{m}\right)_{m\geq1}$
that converges a.e. on $(0,T)\times\mathbb{R}$ to $p_{\epsilon}$.
For $q\in H$, let us define the function $G_{\epsilon}q\in H'$,
by 
\[
\langle G_{\epsilon}q,v\rangle=\left(\partial_{x}v,\Pi A_{\epsilon}\left(q^{+}\right)\partial_{x}q\right)_{d}.
\]
for $v\in H$. Then the equality (\ref{eq:ToyEpsiMode}) rewrites:
\begin{equation}
\forall i\in\{1,...,m\},\frac{d}{dt}\left(w_{i},\Pi p_{\epsilon}^{m}\right)_{d}+\langle G_{\epsilon}p_{\epsilon}^{m},w_{i}\rangle=\text{0}.\label{eq:pcontinuous}
\end{equation}
As by Lemma \ref{lem:Bepsbounded}, the matrices in $A_{\epsilon}\left(\left(\mathbb{R}^{+}\right)^{d}\right)$
are uniformly bounded by $\frac{1}{2}\left(1+\frac{\lambda_{max}}{\lambda_{min}}\right)$,
we have for $q\in H$:

\begin{equation}
||G_{\epsilon}q||_{H'}=\underset{v\in H,||v||_{H}\leq1}{\sup}\left(\partial_{x}v,\Pi A_{\epsilon}\left(q^{+}\right)\partial_{x}q\right)_{d}\leq\frac{d^{2}}{2}\left(1+\frac{\lambda_{max}}{\lambda_{min}}\right)||\Pi||_{\infty}||q||_{d}.\label{eq:Gp}
\end{equation}
Since the family $\left(p_{\epsilon}^{m}\right)_{m\geq1}$ is bounded
in $L^{2}([0,T],H)$, the family $\left(G_{\epsilon}p_{\epsilon}^{m}\right)_{m\geq1}$
is bounded in $L^{2}([0,T],H')$ and the computations in \cite[(iii) p. 285]{temam}
give that for any bounded open subset $\mathcal{O}\subset\mathbb{R}$,
modulo the extraction of a subsequence, 
\[
p_{\epsilon|\mathcal{O}}^{m}\rightarrow p_{\epsilon|\mathcal{O}}\text{ {in} }L^{2}([0,T],L(\mathcal{O}))\text{ strongly and a.e. on \ensuremath{[0,T]\times\mathcal{O}}.}
\]
We define for $n\geq1$, $\mathcal{O}_{n}=(-n,n)$, so $\mathcal{O}_{n}\subset\mbox{\ensuremath{\mathcal{O}}}_{n+1}$,
and $\bigcup_{n\geq1}\mathcal{O}_{n}=\mathbb{R}$. By diagonal extraction,
we get from $\left(p_{\epsilon}^{m}\right)_{m\geq1}$ a subsequence,
called $\left(p_{\epsilon}^{\phi(m)}\right)_{m\geq1}$ such that 
\begin{equation}
\forall n\geq1,p_{\epsilon|\mathcal{O}_{n}}^{\phi(m)}\underset{m\rightarrow\infty}{\rightarrow}p_{\epsilon|\mathcal{O}_{n}}\text{ strongly in \ensuremath{L^{2}([0,T];L(\mathcal{O}_{n})),} }\label{eq:strongconvergence}
\end{equation}
 
\begin{equation}
p_{\epsilon}^{\phi(m)}\rightarrow p_{\epsilon}\text{ a.e. on }[0,T]\times\mathbb{R}.\label{eq:almosteverywhereconvergence}
\end{equation}
We show that $p_{\epsilon}$ is a solution to the variational formulation
$V_{\epsilon}(\mu)$. For $1\leq j\leq m$ and $\psi\in C^{1}([0,T],\mathbb{R})$
with $\psi(T)=0$, we have, through integration by parts, the equality:
\[
-\int_{0}^{T}\left(\psi'(t)w_{j},\Pi p_{\epsilon}^{\phi(m)}(t)\right)_{d}dt+\int_{0}^{T}\left(\psi(t)\partial_{x}w_{j},\Pi A_{\epsilon}\left(\left(p_{\epsilon}^{\phi(m)}\right)^{+}(t)\right)\partial_{x}p_{\epsilon}^{\phi(m)}(t)\right)_{d}dt=\left(w_{j},\Pi p_{0}^{\phi(m)}\right)_{d}\psi(0).
\]
The sequence $\left(p_{0}^{\phi(m)}\right)_{m\geq0}$ converges strongly
to $p_{0}$ in $L$ so 
\[
\left(w_{j},\Pi p_{0}^{\phi(m)}\right)_{d}\psi(0)\rightarrow(w_{j},\Pi p_{0})_{d}\psi(0).
\]
The sequence $\left(p_{\epsilon}^{\phi(m)}\right)_{m\geq0}$ converges
weakly to $p_{\epsilon}$ in $L^{2}([0,T],L)$ so 
\[
-\int_{0}^{T}\left(\psi'(t)w_{j},\Pi p_{\epsilon}^{\phi(m)}(t)\right)_{d}dt\rightarrow-\int_{0}^{T}\left(\psi'(t)w_{j},\Pi p_{\epsilon}(t)\right)_{d}dt.
\]
For the remaining term, we have

$\left|\int_{0}^{T}\left(\psi(t)\partial_{x}w_{j},\Pi A_{\epsilon}(p_{\epsilon}^{+}(t))\partial_{x}p_{\epsilon}(t)\right)_{d}dt-\int_{0}^{T}\left(\psi(t)\partial_{x}w_{j},\Pi A_{\epsilon}\left(\left(p_{\epsilon}^{\phi(m)}\right)^{+}(t)\right)\partial_{x}p_{\epsilon}^{\phi(m)}(t)\right)_{d}dt\right|$
\begin{eqnarray}
 & \leq & \left|\int_{0}^{T}\left(\psi(t)\partial_{x}w_{j},\Pi A_{\epsilon}(p_{\epsilon}^{+}(t))\left(\partial_{x}p_{\epsilon}(t)-\partial_{x}p_{\epsilon}^{\phi(m)}(t)\right)\right)_{d}dt\right|\label{eq:ConvergenceApp1}\\
 & + & \left|\int_{0}^{T}\left(\psi(t)\partial_{x}w_{j},\Pi\left(A_{\epsilon}\left(\left(p_{\epsilon}^{\phi(m)}\right)^{+}(t)\right)-A_{\epsilon}(p_{\epsilon}^{+}(t))\right)\partial_{x}p_{\epsilon}^{\phi(m)}(t)\right)_{d}dt\right|.\label{eq:convergenceApp2}
\end{eqnarray}
For the term in (\ref{eq:ConvergenceApp1}), the sequence $\left(\partial_{x}p_{\epsilon}^{\phi(k)}\right)_{k\geq0}$
converges weakly to $\partial_{x}p_{\epsilon}$ in $L^{2}([0,T];L)$.
In addition, by Lemma \ref{lem:Bepsbounded}, $A_{\epsilon}$ is bounded,
so the function $t\rightarrow\psi(t)A_{\epsilon}^{*}(p_{\epsilon}^{+}(t))\Pi\partial_{x}w_{j}$
belongs to $L^{2}([0,T];L)$ and we have the convergence 
\[
\left|\int_{0}^{T}\left(\psi(t)\partial_{x}w_{j},\Pi A_{\epsilon}(p_{\epsilon}^{+}(t))\left(\partial_{x}p_{\epsilon}(t)-\partial_{x}p_{\epsilon}^{\phi(m)}(t)\right)\right)_{d}dt\right|\underset{m\rightarrow\infty}{\rightarrow}0.
\]
For the term in (\ref{eq:convergenceApp2}), using the Cauchy-Schwarz
inequality we have that

$\left|\int_{0}^{T}\left(\psi(t)\left(A_{\epsilon}^{*}\left(\left(p_{\epsilon}^{\phi(m)}\right)^{+}(t)\right)-A_{\epsilon}^{*}\left(p_{\epsilon}^{+}(t)\right)\right)\Pi\partial_{x}w_{j},\partial_{x}p_{\epsilon}^{\phi(m)}(t)\right)_{d}dt\right|^{2}$
\begin{eqnarray*}
 & \leq & \left(\int_{0}^{T}\left|\psi(t)\left(A_{\epsilon}^{*}\left(\left(p_{\epsilon}^{\phi(m)}\right)^{+}(t)\right)-A_{\epsilon}^{*}(p_{\epsilon}^{+}(t))\right)\Pi\partial_{x}w_{j}\right|_{d}^{2}dt\right)\left(\int_{0}^{T}|\partial_{x}p_{\epsilon}^{\phi(k)}(t)|_{d}^{2}dt\right)\\
 & \leq & \frac{l_{max}(\Pi)}{2\kappa}|p_{0}|_{d}^{2}\int_{0}^{T}\left|\psi(t)\left(A_{\epsilon}^{*}\left(\left(p_{\epsilon}^{\phi(m)}\right)^{+}(t)\right)-A_{\epsilon}^{*}(p_{\epsilon}^{+}(t))\right)\Pi\partial_{x}w_{j}\right|_{d}^{2}dt,
\end{eqnarray*}
where we used the energy estimate (\ref{eq:energyestimatederivative})
in the last inequality. The function $A_{\epsilon}$ is continuous
on $\left(\mathbb{R}^{+}\right)^{d}$ as shown in the proof of Lemma
\ref{lem:FespLocalLipschitz}, and $p_{\epsilon}^{\phi(m)}\rightarrow p_{\epsilon}$
a.e. on $[0,T]\times\mathbb{R}$, so $A_{\epsilon}\left(\left(p_{\epsilon}^{\phi(m)}\right)^{+}\right)\rightarrow A_{\epsilon}\left(p_{\epsilon}^{+}\right)$
a.e. on $[0,T]\times\mathbb{R}$. By Lemma \ref{lem:Bepsbounded},
$A_{\epsilon}$ is uniformly bounded so we have through dominated
convergence, 
\[
\int_{0}^{T}\left|\psi(t)\left(A_{\epsilon}^{*}\left(\left(p_{\epsilon}^{\phi(m)}\right)^{+}(t)\right)-A_{\epsilon}^{*}(p_{\epsilon}^{+}(t))\right)\Pi\partial_{x}w_{j}\right|_{d}^{2}dt\underset{m\rightarrow\infty}{\rightarrow}0.
\]
Thus we have shown that
\begin{equation}
\int_{0}^{T}\left(\psi(t)\partial_{x}w_{j},\Pi A_{\epsilon}\left(\left(p_{\epsilon}^{\phi(m)}\right)^{+}(t)\right)\partial_{x}p_{\epsilon}^{\phi(m)}(t)\right)_{d}dt\underset{m\rightarrow\text{\ensuremath{\infty}}}{\rightarrow}\int_{0}^{T}\left(\psi(t)\partial_{x}w_{j},\Pi A_{\epsilon}\left(p_{\epsilon}^{+}(t)\right)\partial_{x}p_{\epsilon}(t)\right)_{d}dt.\label{eq:ToyEpsFVconvergence}
\end{equation}
When we gather the convergence results and use the fact that the sequence
$\left(w_{j}\right)_{j\geq1}$ is total in $H$, we obtain that 
\[
\forall v\in H,-\int_{0}^{T}(\psi'(t)v,\Pi p_{\epsilon}(t))_{d}dt+\int_{0}^{T}(\psi(t)\partial_{x}v,\Pi A_{\epsilon}(p_{\epsilon}^{+}(t))\partial_{x}p_{\epsilon}(t))_{d}dt=(v,\Pi p_{0})_{d}\psi(0),
\]
which rewrites:
\begin{equation}
\forall v\in H,-\int_{0}^{T}(\psi'(t)v,p_{\epsilon}(t))_{d}dt+\int_{0}^{T}(\psi(t)\partial_{x}v,A_{\epsilon}(p_{\epsilon}^{+}(t))\partial_{x}p_{\epsilon}(t))_{d}dt=(v,p_{0})_{d}\psi(0),\label{eq:ToyEpsFVequality1}
\end{equation}
If moreover $\psi$ belongs to $C_{c}^{\infty}((0,T))$, we obtain
that $p_{\epsilon}$ satisfies (\ref{eq:ToyEps2distribution}) in
the distributional sense on $(0,T)$. For $v\in H$, the function
$t\rightarrow(v,p_{\epsilon}(t))_{d}$ belongs to $H^{1}((0,T))$,
as the functions $t\in(0,T)\rightarrow(v,p_{\epsilon}(t))_{d}$ and
$t\in(0,T)\rightarrow(\partial_{x}v,A_{\epsilon}(p_{\epsilon}^{+}(t))\partial_{x}p_{\epsilon}(t))_{d}$
both belong to $L^{2}((0,T))$. Thanks to \cite[Corollary 8.10]{Brezis},
the following integration by parts formula also holds:

\begin{equation}
\forall v\in H,-\int_{0}^{T}\psi'(t)(v,p_{\epsilon}(t))_{d}dt+\int_{0}^{T}\psi(t)(\partial_{x}v,A_{\epsilon}(p_{\epsilon}^{+}(t))\partial_{x}p_{\epsilon}(t))_{d}dt=(v,p_{\epsilon}(0))_{d}\psi(0).\label{eq:ToyEpsFVequality2}
\end{equation}
If we choose $\psi(0)\neq0$, then by comparing (\ref{eq:ToyEpsFVequality2})
with (\ref{eq:ToyEpsFVequality1}), we have that: 
\begin{equation}
\forall v\in H,(v,p_{\epsilon}(0)-p_{0})_{d}=0,\label{eq:initverified}
\end{equation}
and this concludes the proof for the existence of a solution to $(V_{\epsilon}(\mu))$.
We now show that $p_{\epsilon}\in C([0,T],L)$. The function $p_{\epsilon}$
satisfies:
\[
\forall v\in H,\frac{d}{dt}\left(v,p_{\epsilon}\right)_{d}+\left(v,A_{\epsilon}(p_{\epsilon}^{+})\partial_{x}p_{\epsilon}\right)_{d}=0,
\]
in the distributional sense on $(0,T)$, with 
\begin{eqnarray}
p_{\epsilon} & \in & L^{2}([0,T],H),\label{eq:UdansL2H}\\
A_{\epsilon}(p_{\epsilon}^{+})\partial_{x}p_{\epsilon} & \in & L^{2}([0,T],L).\label{eq:UprimedansL2Hprime}
\end{eqnarray}
Then by \cite[III. Lemma 1.2]{temam}, $p_{\epsilon}$ is a.e. equal
to a function belonging to $C([0,T],L)$.
\end{proof}

\begin{prop}
\label{prop:solutionarenonnegative}Under Condition $(C)$, for $\epsilon>0$,
the solutions to $V_{\epsilon}(\mu)$ are non negative.\end{prop}
\begin{proof}
Let $p_{\epsilon}=(p_{\epsilon,1},...,p_{\epsilon,d})$ be a solution
to the variational problem $V_{\epsilon}(\mu)$. For $x\in\mathbb{R},$
let $x^{-}=\min(x,0)$. We take $p_{\epsilon}^{-}=(p_{\epsilon,1}^{-},...,p_{\epsilon,d}^{-})$
as a test function in (\ref{eq:ToyEps2distribution}). Thanks to (\ref{eq:UdansL2H})-(\ref{eq:UprimedansL2Hprime}),
we obtain by \cite[Lemma 1.2 p. 261]{temam}:

\begin{equation}
\frac{1}{2}\frac{d}{dt}(p_{\epsilon}^{-},p_{\epsilon})_{d}+(\partial_{x}p_{\epsilon}^{-},A_{\epsilon}(p_{\epsilon}^{+})\partial_{x}p_{\epsilon})_{d}=0\label{eq:nonnegative}
\end{equation}
in the sense of distributions. For $f\in H^{1}(\mathbb{R}),\partial_{x}f^{-}=1_{\{f<0\}}\partial_{x}f$,
so we have $\forall i\in\{1,...,d\}$, $p_{\epsilon,i}^{+}\partial_{x}p_{\epsilon,i}^{-}=0$,
$\partial_{x}p_{\epsilon,i}^{-}\partial_{x}p_{\epsilon,i}=(\partial_{x}p_{\epsilon,i}^{-})^{2}$.
As a consequence, 

\[
\partial_{x}p_{\epsilon,i}^{-}\sum_{j\neq i}M_{\epsilon,ij}(p_{\epsilon}^{+})\partial_{x}p_{\epsilon,j}=\partial_{x}p_{\epsilon,i}^{-}\frac{\lambda_{i}p_{\epsilon,i}^{+}}{\epsilon\lor\left(\sum_{l}\lambda_{l}p_{\epsilon,l}^{+}\right)}\left(\sum_{j\neq i}\frac{\sum_{l\neq j}(\lambda_{l}-\lambda_{j})p_{\epsilon,l}^{+}}{\epsilon\lor\left(\sum_{l}\lambda_{l}p_{\epsilon,l}^{+}\right)}\partial_{x}p_{\epsilon,j}\right)=0,
\]
 
\[
(\partial_{x}p_{\epsilon}^{-},A_{\epsilon}(p_{\epsilon}^{+})\partial_{x}p_{\epsilon})_{d}=\int_{\mathbb{R}}\sum_{i,j=1}^{d}\partial_{x}p_{\epsilon,i}^{-}A_{\epsilon,ij}(p_{\epsilon}^{+})\partial_{x}p_{\epsilon,j}dx=\int_{\mathbb{R}}\sum_{i=1}^{d}\partial_{x}p_{\epsilon,i}^{-}A_{\epsilon,ii}(p_{\epsilon}^{+})\partial_{x}p_{\epsilon,i}dx.
\]
Equality (\ref{eq:nonnegative}) simplifies into
\[
\frac{1}{2}\frac{d}{dt}\int_{\mathbb{R}}\sum_{i=1}^{d}\left(p_{\epsilon,i}^{-}\right)^{2}dx+\int_{\mathbb{R}}\sum_{i=1}^{d}A_{\epsilon,ii}(p_{\epsilon}^{+})\left(\partial_{x}p_{\epsilon,i}^{-}\right)^{2}dx=0.
\]
Let us show that for $1\leq i\leq d$, 
\begin{equation}
A_{\epsilon,ii}(p_{\epsilon}^{+})\left(\partial_{x}p_{\epsilon,i}^{-}\right)^{2}\geq\frac{\lambda_{min}}{2\lambda_{max}}\left(\partial_{x}p_{\epsilon,i}^{-}\right)^{2}\geq0.\label{eq:inegalitepositive}
\end{equation}
For $1\leq i\leq d$, 
\[
2A_{\epsilon,ii}(p_{\epsilon}^{+})=1+\frac{\left(\sum_{l\neq i}\lambda_{l}p_{\epsilon,l}^{+}\right)\left(\lambda_{i}\sum_{l\neq i}p_{\epsilon,l}^{+}\right)}{\epsilon^{2}\vee\left(\sum_{l=1}^{d}\lambda_{l}p_{\epsilon,l}^{+}\right)^{2}}-\frac{\left(\sum_{l\neq i}\lambda_{l}p_{\epsilon,l}^{+}\right)^{2}}{\epsilon^{2}\vee\left(\sum_{l=1}^{d}\lambda_{l}p_{\epsilon,l}^{+}\right)^{2}}.
\]
We distinguish the cases $p_{\epsilon,i}^{+}>0$ and $p_{\epsilon,i}^{+}=0$.
In the first case, $\partial_{x}p_{\epsilon,i}^{-}=1_{\{p_{\epsilon,i}<0\}}\partial_{x}p_{\epsilon,i}=0$,
so (\ref{eq:inegalitepositive}) is true. In the second case, let
us define $z:=\left(\sum_{l\neq i}\lambda_{l}p_{\epsilon,l}^{+}\right)=\left(\sum_{l=1}^{d}\lambda_{l}p_{\epsilon,l}^{+}\right)$,
and notice that $\left(\lambda_{i}\sum_{l\neq i}p_{\epsilon,l}^{+}\right)\geq\frac{\lambda_{min}}{\lambda_{max}}z$.
Thus we obtain:
\[
2A_{\epsilon,ii}(p_{\epsilon}^{+})\geq1+\left(\frac{\lambda_{min}}{\lambda_{max}}-1\right)\frac{z^{2}}{\epsilon^{2}\vee z^{2}}\geq\frac{\lambda_{min}}{\lambda_{max}},
\]
as $0\leq\frac{z^{2}}{\epsilon^{2}\vee z^{2}}\leq1$, and (\ref{eq:inegalitepositive})
is also true. 

Consequently, $\int_{\mathbb{R}}\sum_{i=1}^{d}\left(p_{\epsilon,i}^{-}\right)^{2}(t)$
is non increasing in time. As $\int_{\mathbb{R}}\sum_{i=1}^{d}\left(p_{\epsilon,i}^{-}\right)^{2}(0)dx=\int_{\mathbb{R}}\sum_{i=1}^{d}\left(p_{0,i}^{-}\right)^{2}dx=0$,
we conclude that $p_{\epsilon}$ has every component non-negative.
Let us remark that to obtain the non negativity of the solutions to
$V_{\epsilon}(\mu)$, it is sufficient and easier to prove directly
that $A_{\epsilon,ii}\geq0$ for $1\leq i\leq d$, but Inequality
(\ref{eq:inegalitepositive}) will be reused for the proof of Proposition
\ref{lem:VFinnonnegative} below.
\end{proof}
Before showing the existence of a solution to $V_{L^{2}}(\mathbb{R})$,
we compute the energy estimates on $p_{\epsilon}$.
\begin{lem}
\label{lem:EnergyEstimates}Let $\epsilon>0$, and let $p_{\epsilon}$
be a solution to $V_{\epsilon}(\mu)$. The following energy estimates
hold.
\begin{equation}
\underset{t\in[0,T]}{\sup}|p_{\epsilon}(t)|_{d}^{2}\leq\frac{l_{max}(\Pi)}{l_{min}(\Pi)}|p_{0}|_{d}^{2},\label{eq:energyestimateLinfty-1}
\end{equation}

\begin{equation}
\int_{0}^{T}|\partial_{x}p_{\epsilon}(t)|_{d}^{2}dt\leq\frac{l_{max}(\Pi)}{2\kappa}|p_{0}|_{d}^{2},\label{eq:energyestimatederivative-1}
\end{equation}
\begin{equation}
\int_{0}^{T}||p_{\epsilon}||_{d}^{2}dt=\int_{0}^{T}|p_{\epsilon}|_{d}^{2}dt+\int_{0}^{T}|\partial_{x}p_{\epsilon}|_{d}^{2}dt\leq\left(T\frac{l_{max}(\Pi)}{l_{min}\left(\Pi\right)}+\frac{l_{max}(\Pi)}{2\kappa}\right)|p_{0}|_{d}^{2}.\label{eq:energyestimateH1-1}
\end{equation}

\end{lem}

\begin{proof}
We obtain the energy estimates by taking $p_{\epsilon}$ as the test
function in (\ref{eq:ToyEps2distribution}), and the computations
are the same as for (\ref{eq:energyestimateLinfty})-(\ref{eq:energyestimateH1}).
\end{proof}

\begin{prop}
\label{prop:existencetoVL2}Under Condition $(C)$, $V_{L^{2}}(\mu)$
has a solution $p$, which satisfies the same energy estimates as
the solutions to $V_{\epsilon}(\mu)$, and that are given by (\ref{eq:energyestimateLinfty-1})-(\ref{eq:energyestimateH1-1}).
Moreover, a.e. on $(0,T)\times\mathbb{R},\sum_{i=1}^{d}p_{i}(t,x)=\mu*h_{t}(x)$.\end{prop}
\begin{proof}
For $\epsilon>0$, let us now denote by $p_{\epsilon}$ a solution
to $V_{\epsilon}(\mu)$, which exists under Condition $(C)$ according
to Proposition \ref{prop:ExistencetoVeps}. The family $\left(p_{\epsilon}\right)_{\epsilon>0}$
is bounded in $L^{2}([0,T];H)\cap L^{\infty}([0,T];L)$, as the energy
estimates (\ref{eq:energyestimateLinfty-1})-(\ref{eq:energyestimateH1-1})
have bounds that do not depend on $\epsilon$. There exists a limit
function $p\in L^{2}([0,T];H)\cap L^{\infty}([0,T];L)$ and a subsequence
$\left(\epsilon_{k}\right)_{k\geq1}$ decreasing to $0$, such that

\begin{eqnarray*}
p_{\epsilon_{k}} & \rightarrow & p\text{ {in} }L^{2}([0,T];H)\text{ weakly},\\
p_{\epsilon_{k}} & \rightarrow & p\text{ {in} }L^{\infty}([0,T];L)\text{ weakly-*},
\end{eqnarray*}
and $p$ also satisfies Inequalities (\ref{eq:energyestimateLinfty-1})-(\ref{eq:energyestimateH1-1}),
where we replace $p_{\epsilon}$ by $p$. Moreover, in the sense of
\cite[III, Lemma 1.1]{temam}, we check with arguments similar to
(\ref{eq:Gp}) that the family $\left(\frac{dp_{\epsilon_{k}}}{dt}\right)_{k\geq1}$
is bounded in $L^{2}([0,T],H')$ and by \cite[III, Theorem 2.3]{temam},
for any bounded open subset $\mathcal{O}\subset\mathbb{R}$, there
is a subsequence of $\left(p_{\epsilon_{k}}\right)_{k\geq1}$ converging
to $p$ in $L^{2}([0,T],L(\mathcal{O}))$ strongly and a.e. on $\ensuremath{[0,T]\times\mathcal{O}}$.
Then by diagonal extraction, similarly to the proof of Proposition
\ref{prop:ExistencetoVeps}, we can also assume that: 
\[
p_{\epsilon_{k}}\rightarrow p\text{ a.e. on \ensuremath{[0,T]\times\mathbb{R}}},
\]
and that $p$ is non negative as $p_{\epsilon}$ is non negative for
$\epsilon>0$. For any function $\psi\in C^{1}([0,T],\mathbb{R})$,
with $\psi(T)=0$, and $k\geq1$, we have 
\[
\forall v\in H,-\int_{0}^{T}(\psi'(t)v,p_{\epsilon_{\phi(k)}}(t))_{d}dt+\int_{0}^{T}(\psi(t)\partial_{x}v,A_{\epsilon_{\phi(k)}}(p_{\epsilon_{\phi(k)}}(t))\partial_{x}p_{\epsilon_{\phi(k)}}(t))_{d}dt=(v,p_{0})_{d}\psi(0).
\]
The sequence $\left(p_{\epsilon_{\phi(k)}}\right)_{k\geq0}$ converges
weakly to $p$ in $L^{2}([0,T];H)$ so 
\[
-\int_{0}^{T}(\psi'(t)v,p_{\epsilon_{\phi(k)}}(t))_{d}dt\underset{k\rightarrow\infty}{\rightarrow}-\int_{0}^{T}(\psi'(t)v,p(t))_{d}dt.
\]
We show that $p$ takes values in $\mathcal{D}$ a.e. on $[0,T]\times\mathbb{R}$.
As for $\rho\in\left(\mathbb{R}_{+}\right)^{d}$, and $1\leq j\leq d$,
$\sum_{i=1}^{d}M_{\epsilon,ij}(\rho)=0$, we have 
\[
\sum_{i=1}^{d}\left(A_{\epsilon}(p_{\epsilon}^{+})\partial_{x}p_{\epsilon}\right)_{i}=\frac{1}{2}\partial_{x}\left(\sum_{i=1}^{d}p_{\epsilon,i}\right)+\frac{1}{2}\sum_{j=1}^{d}\left(\sum_{i=1}^{d}M_{\epsilon,ij}(p_{\epsilon}^{+})\right)\partial_{x}p_{\epsilon,j}=\frac{1}{2}\partial_{x}\left(\sum_{i=1}^{d}p_{\epsilon,i}\right).
\]
Thus, in the sense of distributions on $(0,T)$, for $\tilde{v}\in H^{1}(\mathbb{R})$,
and $v:=(\tilde{v},...,\tilde{v})\in H$, 
\begin{eqnarray*}
0=\frac{d}{dt}(v,p_{\epsilon})_{d}+(\partial_{x}v,A_{\epsilon}(p_{\epsilon}^{+})\partial_{x}p_{\epsilon})_{d} & = & \frac{d}{dt}\int_{\mathbb{R}}\tilde{v}\left(\sum_{i=1}^{d}p_{\epsilon,i}\right)dx+\frac{1}{2}\int_{\mathbb{R}}\partial_{x}\tilde{v}\partial_{x}\left(\sum_{i=1}^{d}p_{\epsilon,i}\right)dx.
\end{eqnarray*}
We also have for $\epsilon>0$, $\sum_{i=1}^{d}p_{\epsilon,i}(0,x)dx=\mu(dx)$.
Then the function $\sum_{i=1}^{d}p_{\epsilon,i}$ is a solution to
the formulation $H_{L^{2}}(\mu)$ defined by : 
\begin{equation}
z\in L^{2}([0,T];H^{1}(\mathbb{R}))\cap L^{\infty}([0,T];L^{2}(\mathbb{R}))\label{eq:pL2-1-1}
\end{equation}

\begin{equation}
\forall w\in H^{1}(\mathbb{R}),\ \frac{d}{dt}\int_{\mathbb{R}}wzdx+\frac{1}{2}\int_{\mathbb{R}}\left(\partial_{x}w\right)\left(\partial_{x}z\right)dx=\text{0}\text{ in the sense of distributions on \ensuremath{(0,T)},}\label{eq:fvar-1-1}
\end{equation}

\begin{equation}
z(0)=\mu.\label{eq:cond0-1-1}
\end{equation}
The problem $H_{L^{2}}(\mu)$ is a variational formulation to the
heat PDE
\begin{eqnarray*}
\partial_{t}z & = & \frac{1}{2}\partial_{xx}z\\
z(0) & = & \mu,
\end{eqnarray*}
and it is well known that the solution of $H_{L^{2}}(\mu)$ is unique
and expressed as the convolution of $\mu$ with the heat kernel $h_{t}(x)=\frac{1}{\sqrt{2\pi t}}\exp\left(-\frac{x^{2}}{2t}\right).$
Consequently, for a subsequence $\left(p_{\epsilon_{k}}\right)_{k\geq1}$
such that $p_{\epsilon_{k}}\underset{k\rightarrow\infty}{\rightarrow}p$
a.e. on $(0,T]\times\mathbb{R}$, the sequence $\left(z_{\epsilon_{k}}\right)_{k\geq1}$,
defined by $z_{\epsilon_{k}}=\sum_{i=1}^{d}p_{\epsilon_{k},i}$ satisfies
\[
\text{for almost every }(t,x)\in(0,T]\times\mathbb{R},\forall k\geq1,\ z_{\epsilon_{k}}(t,x)=\mu*h_{t}(x)>0,
\]
and the value of $z_{\epsilon_{k}}(t,x)$ is independent from $\epsilon_{k}$.
If we define $z_{0}=\sum_{i=1}^{d}p_{i}$, we have that $z_{0}(t,x)=\mu*h_{t}(x)>0$
a.e. on $(0,T]\times\mathbb{R}$. As the limit $p$ has every component
non-negative, we can thus conclude that $p$ takes values in $\mathcal{D}$,
a.e. on $(0,T]\times\mathbb{R}$. Then for almost every $(t,x)\in[0,T]\times\mathbb{R}$,
there exists $k_{0}(t,x)$ such that $\forall k\geq k_{0}(t,x)$,
$\sum\lambda_{i}p_{\epsilon_{\phi(k)},i}(t,x)\geq\frac{1}{2}\sum_{i}\lambda_{i}p_{i}(t,x)>\epsilon_{\phi(k)}$.
For $k\geq k_{0}(t,x)$, $A_{\epsilon_{\phi(k)}}(p_{\epsilon_{\phi(k)}}(t,x))=A(p_{\epsilon_{\phi(k)}}(t,x))$,
and as $A$ is continuous on $\mathcal{D}$, we have 
\begin{equation}
A_{\epsilon_{\phi(k)}}(p_{\epsilon_{\phi(k)}}(t,x))\underset{k\rightarrow\infty}{\rightarrow}A(p(t,x))\ a.e..\label{eq:AepsPepsCVAE}
\end{equation}
By Lemma \ref{lem:Bepsbounded}, the family $\left(A_{\epsilon_{\phi(k)}}\right)_{k\geq1}$
has bounds uniform in $k$. Hence 
\[
\int_{0}^{T}(\psi(t)\partial_{x}v,A_{\epsilon_{\phi(k)}}(p_{\epsilon_{\phi(k)}}(t))\partial_{x}p_{\epsilon_{\phi(k)}}(t))_{d}dt\underset{k\rightarrow\infty}{\rightarrow}\int_{0}^{T}(\psi(t)\partial_{x}v,A(p(t))\partial_{x}p(t))_{d}dt,
\]
in the same manner as in (\ref{eq:convergenceApp2}). Consequently,
we have that
\begin{equation}
\forall v\in H,-\int_{0}^{T}(\psi'(t)v,p(t))_{d}dt+\int_{0}^{T}(\psi(t)\partial_{x}v,A(p(t))\partial_{x}p(t))_{d}dt=(v,p_{0})_{d}\psi(0).\label{eq:lastequalityL2}
\end{equation}
Similarly to the proof of Proposition \ref{prop:ExistencetoVeps},
we prove (\ref{eq:ToyFVdistribution}) and obtain $p\in C([0,T],L)$.
Finally we repeat the same arguments as in (\ref{eq:initverified})
to conclude that $p$ is a solution to $V_{L^{2}}(\mu)$.
\end{proof}

\subsection{\label{sub:FBMGeneralCase}Proof of Theorem \ref{thm:FBMmuProba}}

We first introduce a lemma that will be used later in the proof.
\begin{lem}
\label{lem:L1dxp}Let $\gamma\geq0$ and let $\phi$ be a non-negative
function, s.t. $\forall t\in(0,T],\int_{t}^{T}\phi^{2}(s)ds\leq\frac{\gamma}{\sqrt{t}}$. 

Then for $t\in(0,T]$ $\int_{0}^{t}\phi(s)ds\leq\frac{\sqrt{\gamma}t^{\frac{1}{4}}}{2^{\frac{1}{4}}-1}$.\end{lem}
\begin{proof}
With monotone convergence, then using the Cauchy-Schwarz inequality,
we obtain, for $t\in(0,T]$: 
\begin{eqnarray*}
\int_{0}^{t}\phi(s)ds & = & \sum_{k=0}^{\infty}\int_{t2^{-(k+1)}}^{t2^{-k}}\phi(s)ds\\
 & \leq & \sum_{k=0}^{\infty}\sqrt{t}2^{-(\frac{k+1}{2})}\sqrt{\int_{t2^{-(k+1)}}^{t2^{-k}}\phi^{2}(s)ds}\leq\sum_{k=0}^{\infty}\sqrt{t}2^{-(\frac{k+1}{2})}\sqrt{\int_{t2^{-(k+1)}}^{T}\phi^{2}(s)ds}\\
 & \leq & \sqrt{t}\sum_{k=0}^{\infty}2^{-(\frac{k+1}{2})}\sqrt{\frac{\gamma}{2^{-(\frac{k+1}{2})}\sqrt{t}}}=\frac{\sqrt{\gamma}t^{\frac{1}{4}}}{2^{\frac{1}{4}}-1}.
\end{eqnarray*}

\end{proof}
Now let $\left(\sigma_{k}\right)_{k\geq0}$ be a sequence decreasing
to $0$ as $k\rightarrow\infty$. In order to rely on Theorem \ref{thm:ThmL2},
we approximate $\mu$ by a sequence of measures $\mu_{\sigma_{k}}$,
weakly converging to $\mu$ and that have densities in $L^{2}(\mathbb{R})$.
To do so, we apply a convolution product on $\mu$ by setting $\mu_{\sigma_{k}}:=\mu*h_{\sigma_{k}^{2}}$.
The measure $\mu_{\sigma_{k}}$ is absolutely continuous with respect
to the Lebesgue measure, with the density

\[
\mu_{\sigma_{k}}(x):=\frac{1}{\sqrt{2\pi\sigma_{k}^{2}}}\int_{\mathbb{R}}e^{-\frac{(x-y)^{2}}{2\sigma_{k}^{2}}}d\mu(y).
\]
With Jensen's inequality and Fubini's theorem, we check that $\mu_{\sigma_{k}}\in L^{2}(\mathbb{R})$.
Indeed,
\begin{eqnarray*}
\int_{\mathbb{R}}\mu_{\sigma_{k}}^{2}(x)dx & = & \int_{\mathbb{R}}\left(\frac{1}{\sqrt{2\pi\sigma_{k}^{2}}}\int_{\mathbb{R}}e^{-\frac{(x-y)^{2}}{2\sigma_{k}^{2}}}d\mu(y)\right)^{2}dx\leq\frac{1}{2\pi\sigma_{k}^{2}}\int_{\mathbb{R}}\left(\int_{\mathbb{R}}e^{-\frac{(x-y)^{2}}{\sigma_{k}^{2}}}dx\right)d\mu(y)=\frac{1}{2\sqrt{\pi}\sigma_{k}}.
\end{eqnarray*}
Let us define $p_{0,\sigma_{k}}:=(\alpha_{1}\mu_{\sigma_{k}},...,\alpha_{d}\mu_{\sigma_{k}})\in L$
and let us denote by $p_{\sigma_{k}}=(p_{\sigma_{k},1},...,p_{\sigma_{k},d})$
a solution to the variational formulation $V_{L^{2}}(\mu_{\sigma_{k}})$,
which exists as an application of Theorem \ref{thm:ThmL2}, and that
satisfies equality (\ref{eq:lastequalityL2}) with initial condition
$p_{0,\sigma_{k}}$ as in the proof of Proposition \ref{prop:existencetoVL2}.
We compute the same energy estimates as previously. As $p_{\sigma_{k},i}$
is non-negative for $1\leq i\leq d$ and $\sum_{i=1}^{d}p_{\sigma_{k},i}=\mu_{\sigma_{k}}*h_{t}$
a.e. on $(0,T)\times\mathbb{R}$, 
\[
\text{for a.e. \ensuremath{t} on }(0,T),\ |p_{\sigma_{k}}(t)|_{d}^{2}\leq\left|\left|\sum_{i=1}^{d}p_{\sigma_{k},i}(t)\right|\right|_{L^{2}}^{2}=\left|\left|\mu*h_{\sigma_{k}^{2}}*h_{t}\right|\right|_{L^{2}}^{2}\leq\frac{1}{2\sqrt{\pi t}},
\]
\begin{eqnarray*}
\int_{0}^{T}|p_{\sigma_{k}}(t)|_{d}^{2}dt & \leq & \int_{0}^{T}\frac{1}{2\sqrt{\pi t}}dt=\sqrt{\frac{T}{\pi}},
\end{eqnarray*}
Using $\Pi p_{\sigma_{k}}$ as a test function in (\ref{eq:ToyFVdistribution})
and the fact that $\Pi A$ is uniformly coercive on $\mathcal{D}$
with coefficient $\kappa$, 
\begin{equation}
\forall t\in(0,T),\ \int_{t}^{T}|\partial_{x}p_{\sigma_{k}}(s)|_{d}^{2}ds\leq\frac{1}{2\kappa}|\sqrt{\Pi}p_{\sigma_{k}}(t)|_{d}^{2}\leq\frac{l_{max}(\Pi)}{4\kappa\sqrt{\pi t}}.\label{eq:forlemma14}
\end{equation}
Let us remark that the estimates $\int_{0}^{T}|p_{\sigma_{k}}(t)|_{d}^{2}dt$
and $\int_{t}^{T}|\partial_{x}p_{\sigma_{k}}(t)|_{d}^{2}ds,$ for
$t\in(0,T)$ have bounds that are independent from the choice of $\sigma_{k}$.
Then for a sequence $\left(s_{n}\right)_{n\geq1}$ with values in
$(0,T)$ and decreasing to $0$ as $n\rightarrow\infty$, there exists
a function $p$ defined a.e. on $(0,T]\times\mathbb{R}$ such that
for each $n\geq1$, there exists a converging subsequence called again
$p_{\sigma_{k}}$ with the following convergences:

\begin{eqnarray*}
p_{\sigma_{k}} & \rightarrow & p\text{ {in} }L^{2}((0,T];L)\text{ weakly},\\
p_{\sigma_{k}} & \rightarrow & p\text{ {in} }L^{2}([s_{n},T];H)\text{ weakly},\\
p_{\sigma_{k}} & \rightarrow & p\text{ {in} }L^{\infty}([s_{n},T];L)\text{ weakly-*}.
\end{eqnarray*}
Similarly to the proof of Proposition \ref{prop:existencetoVL2},
we can also suppose, modulo the extraction of a subsection that 
\[
p_{\sigma_{k}}\rightarrow p\text{ a.e. on \ensuremath{[s_{n},T]\times\mathbb{R}}.}
\]
By diagonal extraction, we obtain a subsequence, called again $p_{\sigma_{k}}$
such that

\begin{eqnarray*}
p_{\sigma_{k}} & \rightarrow & p\text{ {in} }L^{2}((0,T];L)\text{ weakly},\\
p_{\sigma_{k}} & \rightarrow & p\text{ {in} }L_{loc}^{2}((0,T];H)\text{ weakly},\\
p_{\sigma_{k}} & \rightarrow & p\text{ {in} }L_{loc}^{\infty}((0,T];L)\text{ weakly-*},\\
p_{\sigma_{k}} & \rightarrow & p\text{ {in} }(0,T]\times\mathbb{R}.
\end{eqnarray*}
Let us remark that $p\geq0$ as $p_{\sigma_{k}}\geq0$ for $k\geq1$.
Moreover, for a.e. $(t,x)\in(0,T]\times\mathbb{R}$, $\mu_{\sigma_{k}}*h_{t}(x)=\sum_{i=1}^{d}p_{\sigma_{k},i}(t,x)\underset{k\rightarrow\infty}{\rightarrow}\sum_{i=1}^{d}p(t,x)=\mu_{\sigma_{k}}*h_{t}(x)$,
so $p$ takes values in $\mathcal{D}$ a.e. on $(0,T]\times\mathbb{R}$.
Let us prove that $p$ is a solution to $V(\mu)$. For $k\geq0$,
as $p_{\sigma_{k}}$ satisfies (\ref{eq:lastequalityL2}) with initial
condition $p_{0,\sigma_{k}}$, for $\psi\in\mathcal{C}^{1}([0,T],\mathbb{R})$
such that $\psi(T)=0,v\in H^{1}(\mathbb{R})$, we have that 

\begin{eqnarray*}
-\int_{0}^{T}(\psi'(s)v,p_{\sigma_{k}}(s))_{d}ds+\int_{0}^{T}(\psi(s)\partial_{x}v,A(p_{\sigma_{k}})\partial_{x}p_{\sigma_{k}})_{d}ds & = & (v,p_{0,\sigma_{k}})_{d}\psi(0).
\end{eqnarray*}
We study the limit as $k\rightarrow\infty$. First,
\begin{eqnarray*}
\int_{0}^{T}\left(\psi'(s)v,p_{\sigma_{k}}(s)\right)_{d}ds & \rightarrow & \int_{0}^{T}\left(\psi'(s)v,p(s)\right)_{d}ds,
\end{eqnarray*}
as $p_{\sigma_{k}}\rightarrow p\text{ {in} }L^{2}((0,T];L)\text{ weakly}$.
Since $p_{0,\sigma_{k}}\rightarrow p_{0}$ weakly and $v\in H$ is
continuous and bounded, we have the following convergence 
\[
(v,p_{0,\sigma_{k}})_{d}\psi(0)\rightarrow\psi(0)\sum_{i=1}^{d}\alpha_{i}\int v_{i}(x)\mu(dx).
\]
To show that 
\begin{equation}
\int_{0}^{T}\left(\psi(s)\partial_{x}v,A(p_{\sigma_{k}})\partial_{x}p_{\sigma_{k}}\right)_{d}ds\rightarrow\int_{0}^{T}\left(\psi(s)\partial_{x}v,A(p)\partial_{x}p\right)_{d}ds,\label{eq:cvAsigmaA}
\end{equation}
we check that the family $\left(\partial_{x}p_{\sigma_{k}}\right)_{k\geq1}$
and $\partial_{x}p$ belong to a bounded subset of $L^{1}((0,T],L)$.
More precisely, for $q\in\left(\partial_{x}p_{\sigma_{k}}\right)_{k\geq1}\cup\{\partial_{x}p\}$,
using (\ref{eq:forlemma14}) we have that 
\[
\forall t\in(0,T),\ \int_{t}^{T}|\partial_{x}q(s)|_{d}^{2}ds\leq\frac{l_{max}(\Pi)}{4\kappa\sqrt{\pi t}},
\]
and by Lemma \ref{lem:L1dxp}, we have that 
\begin{equation}
\forall t\in(0,T),\ \int_{0}^{t}|\partial_{x}q(s)|_{d}ds\leq\sqrt{\frac{l_{max}(\Pi)}{4\kappa\sqrt{\pi}}}\frac{t^{\frac{1}{4}}}{2^{\frac{1}{4}}-1}.\label{eq:q}
\end{equation}
We then write 
\begin{eqnarray*}
\left|\int_{0}^{T}(\psi(s)\partial_{x}v,A(p_{\sigma_{k}})\partial_{x}p_{\sigma_{k}}-A(p)\partial_{x}p)_{d}ds\right| & \leq & \left|\int_{t}^{T}(\psi(s)\partial_{x}v,A(p_{\sigma_{k}})\partial_{x}p_{\sigma_{k}}-A(p)\partial_{x}p)_{d}ds\right|\\
 & \ + & \left|\int_{0}^{t}(\psi(s)\partial_{x}v,A(p_{\sigma_{k}})\partial_{x}p_{\sigma_{k}})_{d}ds\right|+\left|\int_{0}^{t}(\psi(s)\partial_{x}v,A(p)\partial_{x}p)_{d}ds\right|.
\end{eqnarray*}
For fixed $t>0$, the first term of the r.h.s. goes to $0$ as $k\rightarrow\infty$
with the same reasoning used to obtain (\ref{eq:ToyEpsFVconvergence}).
Moreover,
\begin{eqnarray}
\left|\int_{0}^{t}(\psi(s)\partial_{x}v,A(p_{\sigma_{k}})\partial_{x}p_{\sigma_{k}})_{d}ds\right| & \leq & \frac{d}{2}\left(1+\frac{\lambda_{max}}{\lambda_{min}}\right)||\psi||_{\infty}|\partial_{x}v|_{d}\int_{0}^{t}|\partial_{x}p_{\sigma_{k}}|_{d}ds,\label{eq:CVL2aProba1}\\
\left|\int_{0}^{t}(\psi(s)\partial_{x}v,A(p)\partial_{x}p)_{d}ds\right| & \leq & \frac{d}{2}\left(1+\frac{\lambda_{max}}{\lambda_{min}}\right)||\psi||_{\infty}|\partial_{x}v|_{d}\int_{0}^{t}|\partial_{x}p|_{d}ds,\label{eq:CVL2aProba2}
\end{eqnarray}
by the Cauchy-Schwarz inequality and Lemma \ref{lem:Bepsbounded}.
Using (\ref{eq:q}), the l.h.s. of Inequalities (\ref{eq:CVL2aProba1})
and (\ref{eq:CVL2aProba2}) can be made arbitrarily small, uniformly
in $k$, choosing $t$ small enough, so we obtain (\ref{eq:cvAsigmaA}).
Gathering the convergence results we get 
\begin{equation}
-\int_{0}^{T}(\psi'(s)v,p(s))_{d}ds+\int_{0}^{T}(\psi(s)\partial_{x}v,A(p)\partial_{x}p)_{d}ds=\psi(0)\sum_{i=1}^{d}\alpha_{i}\int v_{i}(x)\mu(dx).\label{eq:egaliteconditioninitiale}
\end{equation}
and this is enough to obtain (\ref{eq:fvar}). As it is easy to check
that the function $s\rightarrow(v,p(s))$ belongs to $H^{1}((t,T))$
for $t\in(0,T)$, similarly to (\ref{eq:ToyEpsFVequality2}), we also
have the following integration by parts formula: 
\begin{equation}
\forall t\in(0,T),-\int_{t}^{T}(\psi'(s)v,p(s))_{d}ds+\int_{t}^{T}(\psi(s)\partial_{x}v,A(p)\partial_{x}p)_{d}ds=(v,p(t))_{d}\psi(t).\label{eq:egalitepetitT}
\end{equation}
The integrals on the l.h.s converge to $-\int_{0}^{T}(\psi'(s)v,p(s))_{d}ds+\int_{0}^{T}(\psi(s)\partial_{x}v,A(p)\partial_{x}p)_{d}ds$
as $t\rightarrow0$, as it is easy to check by Cauchy-Schwarz inequality
that the functions $s\rightarrow(\psi'(s)v,p(s))_{d}$ and $s\rightarrow(\psi(s)\partial_{x}v,A(p)\partial_{x}p)_{d}$
belong to $L^{1}((0,T),\mathbb{R})$. If we choose $\psi(0)\neq0$,
we obtain that 
\begin{equation}
\underset{t\rightarrow0}{\lim}(p(t),v)_{d}=\sum_{i=1}^{d}\alpha_{i}\int v_{i}(x)\mu(dx),\label{eq:ToyVFmeslimit}
\end{equation}
by comparing (\ref{eq:egaliteconditioninitiale}) with (\ref{eq:egalitepetitT})
and this gives existence to $V(\mu)$. Moreover, we obtain, with the
same arguments as the end of Proposition \ref{prop:ExistencetoVeps}
$p\in C((0,T],L)$, and this concludes the proof.

\subsection{\label{sub:WeakSolution}Proof of Theorem \ref{thm:FBMweaksolution}}

The existence of a solution to the variational formulation $V(\mu)$
will give the existence of a weak solution to the original SDE (\ref{eq:ToySDE}).
This essentially comes from the equivalence between existence to a
Fokker-Planck equation and existence to the corresponding martingale
problem, established by Figalli in \cite{Figalli}. To use that result,
we first give a lemma that makes $V(\mu)$ compatible with the variational
formulation described in \cite{Figalli}.
\begin{lem}
\label{lem:VFtoFigalli}Let $p$ be a solution to $V(\mu)$. Then,
for $1\leq i\leq d$ and a.e. $s\in(0,T)$, $\left(\frac{\sum_{k}p_{k}}{\sum_{k}\lambda_{k}p_{k}}\lambda_{i}p_{i}\right)(s,\cdot)\in H^{1}(\mathbb{R})$
and $\frac{1}{2}\partial_{x}\left(\frac{\sum_{k}p_{k}}{\sum_{k}\lambda_{k}p_{k}}\lambda_{i}p_{i}\right)(s,\cdot)=\left(A(p)\partial_{x}p\right)_{i}(s,\cdot)$.\end{lem}
\begin{proof}
It is sufficient to show that for $1\leq i,j\leq d$ and a.e. $s\in(0,T)$,
$\frac{p_{i}p_{j}}{\sum_{k}\lambda_{k}p_{k}}(s,\cdot)\in H^{1}(\mathbb{R})$
and 
\[
\partial_{x}\left(\frac{p_{i}p_{j}}{\sum_{k}\lambda_{k}p_{k}}\right)(s,\cdot)=\left(\frac{p_{i}\partial_{x}p_{j}+p_{j}\partial_{x}p_{i}}{\sum_{k}\lambda_{k}p_{k}}-\frac{p_{i}p_{j}}{\left(\sum_{k}\lambda_{k}p_{k}\right)^{2}}\sum_{k}\lambda_{k}\partial_{x}p_{k}\right)(s,\cdot),
\]
as we conclude by linearity. For a.e. $s\in(0,T)$, $\frac{p_{i}p_{j}}{\sum_{k}\lambda_{k}p_{k}}(s,\cdot)$
and $\left(\frac{p_{i}\partial_{x}p_{j}+p_{j}\partial_{x}p_{i}}{\sum_{k}\lambda_{k}p_{k}}-\frac{p_{i}p_{j}}{\left(\sum_{k}\lambda_{k}p_{k}\right)^{2}}\sum_{k}\lambda_{k}\partial_{x}p_{k}\right)(s,\cdot)$
belong to $L^{2}(\mathbb{R})$, as $\forall i\in\{1,...,d\},\frac{p_{i}}{\sum_{k}\lambda_{k}p_{k}}\in[0;\frac{1}{\lambda_{min}}]$,
a.e. on $(0,T)\times\mathbb{R}$. It is sufficient to show that for
$K\subset\mathbb{R}$ compact, $\phi\in C_{c}^{\infty}(\mathbb{R})$
with support included in $K$, and a.e. $s\in(0,T)$, 
\[
\int_{K}\partial_{x}\phi(x)\left(\frac{p_{i}p_{j}}{\sum_{k}\lambda_{k}p_{k}}\right)(s,x)dx=-\int_{K}\phi(x)\left(\frac{p_{i}\partial_{x}p_{j}+p_{j}\partial_{x}p_{i}}{\sum_{k}\lambda_{k}p_{k}}-\frac{p_{i}p_{j}}{\left(\sum_{k}\lambda_{k}p_{k}\right)^{2}}\sum_{k}\lambda_{k}\partial_{x}p_{k}\right)(s,x)dx.
\]
Let $\left(\rho_{n}\right)_{n\geq1}$ be a regularizing sequence,
where $\rho_{n}\in C_{c}^{\infty}(\mathbb{R}),$ with a support included
in $(-\frac{1}{n},\frac{1}{n})$, $\int_{\mathbb{R}}\rho_{n}=1,$
and $\rho_{n}\geq0$ for $n\geq1$. We define $p_{n,i}(s,\cdot):=\rho_{n}*p_{i}(s,\cdot)$.
Then for a.e. $s\in(0,T)$, $(p_{n,i}(s,\cdot))_{n\geq0},i\in\{1,...,d\}$
are sequences of non negative functions in $C^{\infty}(\mathbb{R})\cap H^{1}(\mathbb{R})$
such that for $1\leq i\leq d$, the following strong convergences
hold, by \cite[Theorem 4.22]{Brezis}: 
\begin{eqnarray*}
p_{n,i}(s,\cdot) & \underset{n\rightarrow\infty}{\rightarrow} & p_{i}(s,\cdot)\text{ in \ensuremath{L^{2}(\mathbb{R})}},\\
\partial_{x}p_{n,i}(s,\cdot) & \underset{n\rightarrow\infty}{\rightarrow} & \partial_{x}p_{i}(s,\cdot)\text{ in \ensuremath{L^{2}(\mathbb{R})}}.
\end{eqnarray*}
We first check that for $n\geq1$, $x\in\mathbb{R}$ and a.e. $s\in(0,T)$,
$\sum_{k=1}^{d}\lambda_{k}p_{n,k}(s,x)>0$. Indeed, as $\sum_{i}p_{i}$
is a solution to the heat equation, $\sum_{i=1}^{d}p_{i}(s,x)=\mu*h_{s}(x)$
a.e. on $(0,T)\times\mathbb{R}$, we have that for $x\in\mathbb{R}$
and a.e. $s\in(0,T)$, $\sum_{k=1}^{d}\lambda_{k}p_{n,k}(s,x)\geq\lambda_{min}\sum_{k=1}^{d}p_{n,k}(s,x)=\lambda_{min}\rho_{n}*\mu*h_{s}(x)>0$.
Then for a.e. $s\in(0,T)$, we have the equality: 
\begin{equation}
\int_{K}\partial_{x}\phi(x)\left(\frac{p_{n,i}p_{n,j}}{\sum_{k}\lambda_{k}p_{n,k}}\right)(s,x)dx=-\int_{K}\phi(x)\left(\frac{p_{n,i}\partial_{x}p_{n,j}+p_{n,j}\partial_{x}p_{n,i}}{\sum_{k}\lambda_{k}p_{n,k}}-\frac{p_{n,i}p_{n,j}}{\left(\sum_{k}\lambda_{k}p_{n,k}\right)^{2}}\sum_{k}\lambda_{k}\partial_{x}p_{n,k}\right)(s,x)dx.\label{eq:deriv}
\end{equation}
For a.e. $s\in(0,T)$, modulo the extraction of a subsequence (which
can depend on $s$), we can assume that for $1\leq l\leq d$ and a.e.
$x\in K$, $p_{n,l}(s,x)\underset{n\rightarrow\infty}{\rightarrow}p_{l}(s,x)$.
As $\sum_{l=1}^{d}p_{l}(s,x)=\mu*h_{s}(x)>0$ for a.e. $x\in\mathbb{R}$,
we have $\frac{p_{n,l}}{\sum_{k}\lambda_{k}p_{n,k}}(s,x)\rightarrow\frac{p_{l}}{\sum_{k}\lambda_{k}p_{k}}(s,x)$
for a.e. $x\in K$. Then,

\begin{eqnarray*}
\left|\left|\left(\frac{p_{n,i}p_{n,j}}{\sum_{k}\lambda_{k}p_{n,k}}-\frac{p_{i}p_{j}}{\sum_{k}\lambda_{k}p_{k}}\right)(s,\cdot)\right|\right|_{L^{2}(K)} & \leq & \left|\left|\left(\frac{p_{n,i}p_{n,j}}{\sum_{k}\lambda_{k}p_{n,k}}-\frac{p_{n,i}p_{j}}{\sum_{k}\lambda_{k}p_{n,k}}\right)(s,\cdot)\right|\right|_{L^{2}(K)}\\
 & \ + & \left|\left|\left(\frac{p_{n,i}p_{j}}{\sum_{k}\lambda_{k}p_{n,k}}-\frac{p_{i}p_{j}}{\sum_{k}\lambda_{k}p_{k}}\right)(s,\cdot)\right|\right|_{L^{2}(K)}\\
 & \leq & \frac{1}{\lambda_{min}}||\left(p_{n,j}-p_{j}\right)(s,\cdot)||_{L^{2}(K)}\\
 & \ + & \left|\left|\left(p_{j}\left(\frac{p_{n,i}}{\sum_{k}\lambda_{k}p_{n,k}}-\frac{p_{i}}{\sum_{k}\lambda_{k}p_{k}}\right)\right)(s,\cdot)\right|\right|_{L^{2}(K)}.
\end{eqnarray*}

The first term of the r.h.s converges to $0$ as $n\rightarrow\infty$,
as $p_{n,i}(s,\cdot)\underset{n\rightarrow\infty}{\rightarrow}p_{i}(s,\cdot)\text{ in \ensuremath{L^{2}(\mathbb{R})}}$,
and the second term also converges to $0$ by dominated convergence
as $\frac{p_{n,i}}{\sum_{k}\lambda_{k}p_{n,k}}(s,\cdot)\rightarrow\frac{p_{i}}{\sum_{k}\lambda_{k}p_{k}}(s,\cdot)$
a.e. on $K$, and $\forall i\in\{1,...,d\},\forall n\geq1,\frac{p_{n,i}}{\sum_{k}\lambda_{k}p_{n,k}}(s,x)\in[0;\frac{1}{\lambda_{min}}].$
This ensures that 
\[
\int_{K}\partial_{x}\phi(x)\left(\frac{p_{n,i}p_{n,j}}{\sum_{k}\lambda_{k}p_{n,k}}\right)(s,x)dx\underset{n\rightarrow\infty}{\rightarrow}\int_{K}\partial_{x}\phi(x)\left(\frac{p_{i}p_{j}}{\sum_{k}\lambda_{k}p_{k}}\right)(s,x)dx,
\]
for a.e. $s\in(0,T)$. With similar arguments, we let $n\rightarrow\infty$
in the r.h.s. of (\ref{eq:deriv}), we have the convergence of the
r.h.s. term to 
\[
-\int_{K}\phi(x)\left(\frac{p_{i}\partial_{x}p_{j}+p_{j}\partial_{x}p_{i}}{\sum_{k}\lambda_{k}p_{k}}-\frac{p_{i}p_{j}}{\left(\sum_{k}\lambda_{k}p_{k}\right)^{2}}\sum_{k}\lambda_{k}\partial_{x}p_{k}\right)(s,x)dx,
\]
for a.e. $s\in(0,T)$ and this concludes the proof.
\end{proof}
We can now prove Theorem \ref{thm:FBMweaksolution}. By Lemma \ref{lem:VFtoFigalli}
the solution of the variational formulation $V(\mu)$ satisfies :

\[
\forall i\in\{1,...,d\},\forall\phi\in H^{1}(\mathbb{R}),\frac{d}{dt}\int_{x}\phi p_{i}dx+\frac{1}{2}\int_{x}\left(\partial_{x}\phi\right)\partial_{x}\left(\frac{\sum_{k}p_{k}}{\sum_{k}\lambda_{k}p_{k}}\lambda_{i}p_{i}\right)dx=0,
\]
in the sense of distributions. Then through an integration by parts,

\[
\forall i\in\{1,...,d\},\forall\phi\in C_{c}^{\infty}(\mathbb{R}),\frac{d}{dt}\int_{x}\phi p_{i}dx-\frac{1}{2}\int_{x}\left(\partial_{xx}^{2}\phi\right)\left(\frac{\sum_{k}p_{k}}{\sum_{k}\lambda_{k}p_{k}}\lambda_{i}p_{i}\right)dx=0.
\]
For $y_{i}\in\mathcal{Y}$, \cite[Theorem 2.6]{Figalli} gives the
existence of a probability measure $\mathbb{P}^{y_{i}}$ on the space
$C([0,T],\mathbb{R})$ with canonical process $(X_{t})_{0\leq t\leq T}$
satisfying :

\begin{eqnarray*}
X_{0} & \sim & \mu\text{ under \ensuremath{\mathbb{P}^{y_{i}}},}\\
\forall\phi\in C_{c}^{\infty}(\mathbb{R}), &  & M_{t}^{\phi,y_{i}}:=\phi(X_{t})-\phi(X_{0})-\frac{1}{2}\int_{0}^{t}\partial_{xx}^{2}\phi(X_{s})f^{2}(y_{i})\frac{\sum_{k}p_{k}(s,X_{s})}{\sum_{k}f^{2}(y_{k})p_{k}(s,X_{s})}ds\text{ is a martingale under \ensuremath{\mathbb{P}^{y_{i}}}},
\end{eqnarray*}
and for $t>0$, $X_{t}$ has the density $\frac{p_{i}(t,\cdot)}{\alpha_{i}}$
under $\mathbb{P}^{y_{i}}$. We then form the measure $\mathbb{Q}=\sum_{i=1}^{d}\alpha_{i}\mathbb{P}^{y_{i}}(dX)\otimes\delta_{y_{i}}(dY)$
and show that it solves the following martingale problem: 

\begin{eqnarray*}
X_{0} & \sim & \mu\text{ under \ensuremath{\mathbb{Q}},}\\
\forall\phi\in C_{c}^{\infty}(\mathbb{R}), &  & M_{t}^{\phi,Y}:=\phi(X_{t})-\phi(X_{0})-\frac{1}{2}\int_{0}^{t}\partial_{xx}^{2}\phi(X_{s})f^{2}(Y)\frac{\sum_{i=1}^{d}p_{i}(s,X_{s})}{\sum_{i=1}^{d}f^{2}(y_{i})p_{i}(s,X_{s})}ds\text{ is a martingale under \ensuremath{\mathbb{Q}}.}
\end{eqnarray*}
For $\phi\in C_{c}^{\infty}(\mathbb{R})$ and $t\geq0$, $M_{t}^{\phi,Y}$
is bounded as $\left|\frac{f^{2}(Y)\sum_{i=1}^{d}p_{i}(s,X_{s})}{\sum_{i=1}^{d}f^{2}(y_{i})p_{i}(s,X_{s})}\right|\leq\frac{\lambda_{max}}{\lambda_{min}}$.
For $s\geq0$, we define $\mathcal{F}_{s}=\sigma(\{X_{u},u\leq s\})$.
To obtain our result it is enough to check that for $0\leq s\leq t$,
\[
\mathbb{E}^{\mathbb{Q}}\left[M_{t}^{\phi,Y}-M_{s}^{\phi,Y}|\mathcal{F}_{s},Y\right]=0.
\]
For $\phi\in C_{c}^{\infty}(\mathbb{R})$ and $g$ measurable and
bounded on $\mathbb{R}^{p}\times\mathcal{Y}$, $p\geq1$, and $0\leq s_{1}\leq...\leq s_{p}\leq s$,
\begin{eqnarray*}
\mathbb{E}^{\mathbb{Q}}\left[(M_{t}^{\phi,Y}-M_{s}^{\phi,Y})g(X_{s_{1}},...,X_{s_{p}},Y)\right] & = & \sum_{i=1}^{d}\alpha_{i}\mathbb{E}^{\mathbb{P}^{y_{i}}}\left[(M_{t}^{\phi,y_{i}}-M_{s}^{\phi,y_{i}})g(X_{s_{1}},...,X_{s_{p}},y_{i})\right]=0,
\end{eqnarray*}
as $M_{t}^{\phi,y_{i}}$ is a $\mathbb{P}^{y_{i}}$-martingale for
$1\leq i\leq d$. So $\mathbb{E}^{\mathbb{Q}}\left[M_{t}^{\phi,Y}|\mathcal{F}_{s},Y\right]=M_{s}^{\phi,Y}$,
and $M_{t}^{\phi,Y}$ is a $\mathbb{Q}$-martingale. For $0\leq s\leq T$,
we compute the conditional expectation $\mathbb{E}^{\mathbb{Q}}\left[f^{2}(Y)|X_{s}\right]$.
Given a bounded measurable function $g$ on $\mathbb{R}$,
\begin{eqnarray*}
\mathbb{E}^{\mathbb{Q}}\left[f^{2}(Y)g(X_{s})\right] & = & \sum_{i=1}^{d}\alpha_{i}\mathbb{E}^{\mathbb{P}^{y_{i}}}\left[f^{2}(y_{i})g(X_{s})\right]=\sum_{i=1}^{d}\alpha_{i}\int_{\mathbb{R}}f^{2}(y_{i})g(x)\frac{p_{i}(s,x)}{\alpha_{i}}dx\\
 & = & \int_{\mathbb{R}}\sum_{i=1}^{d}f^{2}(y_{i})g(x)p_{i}(s,x)dx=\int_{\mathbb{R}}\left(\frac{\sum_{i=1}^{d}f^{2}(y_{i})p_{i}(s,x)}{\sum_{i=1}^{d}p_{i}(s,x)}\right)g(x)\sum_{i=1}^{d}p_{i}(s,x)dx,\\
\mathbb{E}^{\mathbb{Q}}\left[g(X_{s})\right] & = & \sum_{i=1}^{d}\alpha_{i}\mathbb{E}^{\mathbb{P}^{y_{i}}}\left[g(X_{s})\right]=\int_{\mathbb{R}}g(x)\sum_{i=1}^{d}p_{i}(s,x)dx.
\end{eqnarray*}
Thus under $\mathbb{Q}$, $X_{s}$ has the density $\sum_{i=1}^{d}p_{i}(s,\cdot)$,
which is equal to $\mu*h_{s}(\cdot)$ by Theorem \ref{thm:FBMmuProba},
so $X_{s}$ has the same law as $Z+W_{s}$, where $Z\sim\mu$ and
$Z$ is independent from $\left(W_{t}\right)_{t\geq0}$. Moreover,
we have the equality: 
\[
\mathbb{E}^{\mathbb{Q}}\left[f^{2}(Y)|X_{s}\right]=\frac{\sum_{i=1}^{d}f^{2}(y_{i})p_{i}(s,X_{s})}{\sum_{i=1}^{d}p_{i}(s,X_{s})}\ \text{a.s..}
\]
Therefore $\mathbb{Q}$ is a solution to the martingale problem: 
\begin{eqnarray*}
X_{0} & \sim & \mu\text{ under \ensuremath{\mathbb{Q}},}\\
\forall\phi\in C_{c}^{\infty}(\mathbb{R}), &  & \phi(X_{t})-\phi(X_{0})-\frac{1}{2}\int_{0}^{t}\partial_{xx}^{2}\phi(X_{s})\frac{f^{2}(Y)}{\mathbb{E}^{\mathbb{Q}}\left[f^{2}(Y)|X_{s}\right]}ds\text{ is a martingale under \ensuremath{\mathbb{Q}}.}
\end{eqnarray*}
and this ensures the existence of a weak solution to the SDE (\ref{eq:ToySDE}).

\section{\label{sec:CalibrationofRSLVmain}Calibrated RSLV models}

We extend the results obtained in the previous sections to the case
when the asset price $S$ follows the dynamics: 
\begin{eqnarray*}
dS_{t} & = & rS_{t}dt+\frac{f(Y_{t})}{\sqrt{\mathbb{E}\left[f^{2}(Y_{t})|S_{t}\right]}}\sigma_{Dup}(t,S_{t})S_{t}dW_{t},\\
\left(\log\left(S_{0}\right),Y_{0}\right) & \sim & \mu,
\end{eqnarray*}
where $Y_{0}$ is a random variables with values in $\mathcal{Y}$,
$\mu$ is a probability measure on $\mathbb{R}\times\mathcal{Y}$,
$Y$ is a process evolving in $\mathcal{Y}$, with 
\[
\mathbb{P}\left(Y_{t+dt}=j|\sigma\left(\left(S_{s},Y_{s}\right),0\leq s\leq t\right)\right)=q_{Y_{t}j}(\log(S_{t}))dt
\]
for $j\neq Y_{t}$, and for $1\leq i\neq j\leq d$, the function $q_{ij}:\mathbb{R}\rightarrow\mathbb{R}$
is non negative. Moreover, for $1\leq i\leq d$, we define $q_{ii}:=-\sum_{j\neq i}q_{ij}$.
We assume that $\left(\log\left(S_{0}\right),Y_{0}\right)$ and $(W_{t})_{t\geq0}$
are independent. In addition, we assume that there exists $\overline{q}>0$
such that $||q_{ij}||_{\infty}\leq\overline{q}$ for $1\leq i,j\leq d$
, and that the risk free rate $r$ is constant. We define $\tilde{\sigma}_{Dup}(t,x):=\sigma_{Dup}(t,e^{x})$.
The asset log-price $X_{t}=\log S_{t}$ follows the dynamics:
\begin{eqnarray}
dX_{t} & = & \left(r-\frac{1}{2}\frac{f^{2}(Y_{t})}{\mathbb{E}\left[f^{2}(Y_{t})|X_{t}\right]}\tilde{\sigma}_{Dup}^{2}(t,X_{t})\right)dt+\frac{f(Y_{t})}{\sqrt{\mathbb{E}\left[f^{2}(Y_{t})|X_{t}\right]}}\tilde{\sigma}_{Dup}(t,X_{t})dW_{t},\label{eq:sdefinance1}\\
\left(X_{0},Y_{0}\right) & \sim & \mu.\label{eq:sdefinance2}
\end{eqnarray}
Formally, if we apply Gyongy's theorem, any solution to the SDE (\ref{eq:sdefinance1})
should have the same time marginals as the solution to the Dupire
SDE for the asset's log-price:
\begin{eqnarray}
dX_{t}^{D} & = & \left(r-\frac{1}{2}\tilde{\sigma}_{Dup}^{2}(t,X_{t}^{D})\right)dt+\tilde{\sigma}_{Dup}(t,X_{t}^{D})dW_{t},\label{eq:SDEFinDupireLogPrice1}\\
X_{0}^{D} & \sim & \mu_{X_{0}}\label{eq:SDEFinDupireLogPrice2}
\end{eqnarray}
where $\mu_{X_{0}}$ is the law of $X_{0}$. The Fokker-Planck PDS
associated with the SDE (\ref{eq:sdefinance1})-(\ref{eq:sdefinance2})
writes, for $1\leq i\leq d$:
\begin{eqnarray}
\partial_{t}p_{i} & = & -\partial_{x}\left(\left[r-\frac{1}{2}\frac{\lambda_{i}\sum_{i=1}^{d}p_{k}}{\sum_{i=1}^{d}\lambda_{k}p_{k}}\tilde{\sigma}_{Dup}^{2}\right]p_{i}\right)+\frac{1}{2}\partial_{xx}^{2}\left(\frac{\lambda_{i}\sum_{i=1}^{d}p_{k}}{\sum_{i=1}^{d}\lambda_{k}p_{k}}\tilde{\sigma}_{Dup}^{2}p_{i}\right)+\sum_{j=1}^{d}q_{ji}p_{j}\label{eq:fokkerplanckFin}\\
p_{i}(0,\cdot) & = & \alpha_{i}\mu_{i},\label{eq:fokkerplanckFin2}
\end{eqnarray}
where for $1\leq i\leq d$, $\mu_{i}$ is the conditional law of $X_{0}$
given $\{Y_{0}=i\}$, and as before, $\alpha_{i}=\mathbb{P}\left(Y_{0}=i\right)\geq0$.
We denote by $\Lambda$ the diagonal $d\times d$ matrix with coefficients
$(\lambda_{i})_{1\leq i\leq d}$. We define, for $\rho\in\mathcal{D}$,
$R(\rho):=\frac{\sum_{i=1}^{d}\rho_{i}}{\sum_{i=1}^{d}\lambda_{i}\rho_{i}}$,
and for $x\in\mathbb{R},$ $Q(x)=(q_{ij}(x))_{1\leq i,j\leq d}$.
Moreover, we assume that the European call prices given by the market
have sufficient regularity so that the following assumption holds. 
\begin{assumption*}[B]
 The function $\tilde{\sigma}_{Dup}$ belongs to the space $L^{\infty}([0,T],W^{1,\infty}(\mathbb{R}))$,
and there exists a constant $\underbar{\ensuremath{\sigma}}>0$ such
that a.e. on $[0,T]\times\mathbb{R}$, $\underbar{\ensuremath{\sigma}}\leq\tilde{\sigma}_{Dup}$. 
\end{assumption*}
We will denote by $\overline{\sigma}>0$ some constant such that a.e.
on $[0,T]\times\mathbb{R}$, $\tilde{\sigma}_{Dup}\leq\overline{\sigma}$.
For the PDS (\ref{eq:fokkerplanckFin})-(\ref{eq:fokkerplanckFin2}),
we introduce an associated variational formulation, called $V_{Fin}(\mu)$:

\begin{eqnarray*}
\text{Find} &  & p=(p_{1},...,p_{d})\ \text{satisfying:}
\end{eqnarray*}
\[
p\in L_{loc}^{2}((0,T];H)\cap L_{loc}^{\infty}((0,T];L),
\]
\[
\text{\ensuremath{p} takes values in \ensuremath{\mathcal{D}}, a.e. on \ensuremath{(0,T)\times\mathbb{R}},}
\]
\[
\forall v\in H,\ \frac{d}{dt}(v,p)_{d}-r\left(\partial_{x}v,p\right)_{d}+\left(\partial_{x}v,\frac{1}{2}R(p)\tilde{\sigma}_{Dup}\left(\tilde{\sigma}_{Dup}+2\partial_{x}\tilde{\sigma}_{Dup}\right)\Lambda p\right)_{d}+(\partial_{x}v,\tilde{\sigma}_{Dup}^{2}A(p)\partial_{x}p)_{d}=(Qv,p)_{d}
\]
\[
\text{in the sense of distributions on \ensuremath{(0,T)}, and}
\]

\[
p(t,\cdot)\underset{t\rightarrow0^{+}}{\overset{\text{weakly-*}}{\rightarrow}}p_{0}:=(\alpha_{1}\mu_{1},...,\alpha_{d}\mu_{d}).
\]
Let us remark that if we sum the PDS (\ref{eq:fokkerplanckFin}) over
the index $i$, as $\sum_{i=1}^{d}q_{ji}=0$ for $1\leq j\leq d$,
then $\sum_{i=1}^{d}p_{i}$ satisfies the Fokker-Planck equation associated
to the SDE (\ref{eq:SDEFinDupireLogPrice2}):
\begin{eqnarray}
\partial_{t}\sum_{i=1}^{d}p_{i} & = & -\partial_{x}\left(\left[r-\frac{1}{2}\tilde{\sigma}_{Dup}^{2}\right]\sum_{i=1}^{d}p_{i}\right)+\frac{1}{2}\partial_{xx}^{2}\left(\tilde{\sigma}_{Dup}^{2}\sum_{i=1}^{d}p_{i}\right)\label{eq:DupireFP}\\
\sum_{i=1}^{d}p_{i}(0,\cdot) & = & \mu_{X_{0}}.
\end{eqnarray}
In the same way, if $p$ is a solution to $V_{Fin}(\mu)$, then $u:=\sum_{i=1}^{d}p_{i}$
solves $LV(\mu_{X_{0}})$, where for $\nu\in\mathcal{P}(\mathbb{R})$,
$LV(\nu)$ is defined by:
\[
u\in L_{loc}^{2}((0,T];H^{1}(\mathbb{R}))\cap L_{loc}^{\infty}((0,T];L^{2}(\mathbb{R})),\ u\geq0,
\]
\[
\forall v\in H^{1}(\mathbb{R}),\ \frac{d}{dt}(v,u)_{1}-r\left(\partial_{x}v,u\right)_{1}+\left(\partial_{x}v,\frac{1}{2}\tilde{\sigma}_{Dup}\left(\tilde{\sigma}_{Dup}+2\partial_{x}\tilde{\sigma}_{Dup}\right)u\right)_{1}+\frac{1}{2}\left(\partial_{x}v,\tilde{\sigma}_{Dup}^{2}\partial_{x}u\right)_{1}=0,
\]
\[
\text{in the sense of distributions on \ensuremath{(0,T)}, and }
\]
\[
u(t,\cdot)\underset{t\rightarrow0^{+}}{\overset{\text{weakly-*}}{\rightarrow}}\nu.
\]

\begin{lem}
\label{lem:LVpositivecontinuous}Under Assumption (B), for $\nu\in\mathcal{P}(\mathbb{R})$,
the solutions of $LV(\nu)$ are continuous and positive on $(0,T]\times\mathbb{R}$.
\end{lem}
The proof of Lemma \ref{lem:LVpositivecontinuous} is postponed to
Appendix \ref{sec:Additional-proofs-of-Section4}. We make here an
additional assumption on the regularity of the function $\tilde{\sigma}_{Dup}$.
\begin{assumption*}[H]
 The function $\tilde{\sigma}_{Dup}$ is continuous and there exist
two constants $H_{0}>0$ and $\chi\in(0,1]$ such that 
\[
\forall s,t\in[0,T],\forall x\in\mathbb{R},\ |\tilde{\sigma}_{Dup}(s,x)-\tilde{\sigma}_{Dup}(t,x)|\leq H_{0}|t-s|^{\chi}.
\]

\end{assumption*}
If $\tilde{\sigma}_{Dup}\in L^{\infty}([0,T],W^{1,\infty}(\mathbb{R}))$
and $\tilde{\sigma}_{Dup}$ is continuous, then $\tilde{\sigma}_{Dup}$
and $\tilde{\sigma}_{Dup}^{2}$ have the Lipschitz property in the
space variable, uniformly in time so that existence and trajectorial
uniqueness hold for (\ref{eq:SDEFinDupireLogPrice1})-(\ref{eq:SDEFinDupireLogPrice2}).
Moreover, Assumptions $(H)$ and $(B)$ imply that there exists $\tilde{H}_{0}>0$
such that: 
\[
\forall s,t\in[0,T],\forall x,y\in\mathbb{R},\ |\tilde{\sigma}_{Dup}(s,x)-\tilde{\sigma}_{Dup}(t,y)|\leq\tilde{H}_{0}\left(|t-s|^{\chi}+|x-y|\right).
\]
Moreover, for $\nu\in\mathcal{P}(\mathbb{R})$, they are sufficient
to obtain uniqueness to $LV(\nu)$ and Aronson-like upper-bounds on
the solution of $LV(\nu)$. The proof of Proposition \ref{prop:uniquenessandaronson}
that follows is also postponed to Appendix \ref{sec:Additional-proofs-of-Section4}. 
\begin{prop}
\label{prop:uniquenessandaronson}Under Assumptions $(B)$ and $(H)$,
there exists a unique solution $u$ to $LV(\mu_{X_{0}})$ and the
time marginals of the solution $\left(X_{t}^{D}\right)_{t\in(0,T]}$
to the SDE (\ref{eq:SDEFinDupireLogPrice1})-(\ref{eq:SDEFinDupireLogPrice2})
are given by $\left(u(t,\cdot)\right)_{t\in(0,T]}$. Moreover, there
exists a finite constant $\zeta$, independent from $\mu_{X_{0}}$
and such that $u$ satisfies $||u(t)||_{L^{2}}^{2}\leq\frac{\zeta}{\sqrt{t}}$
for a.e. $t\in(0,T]$. 
\end{prop}
We give here the main results on the calibrated RSLV model, that we
prove in Section \ref{sec:calibrationofRSLVproofs}.
\begin{thm}
\label{thm:VFin}Under Condition $(C)$, Assumptions $(B)$ and $(H)$,
there exists a solution $p\in C((0,T],L)$ to $V_{Fin}(\mu)$ such
that $\sum_{i=1}^{d}p_{i}$ is the unique solution to $LV(\mu_{X_{0}})$.
\end{thm}

\begin{thm}
\label{thm:weaksolutionFin}Under Condition $(C)$, Assumptions $(B)$
and $(H)$, there exists a weak solution to the SDE (\ref{eq:sdefinance1})-(\ref{eq:sdefinance2}).
which has the same time marginals as those of the solution to the
SDE (\ref{eq:SDEFinDupireLogPrice1})-(\ref{eq:SDEFinDupireLogPrice2}).
\end{thm}

\section{\label{sec:calibrationofRSLVproofs}Proofs of Section 4}

\subsection{Proof of Theorem \ref{thm:VFin}}

Similarly to the proof of Theorem \ref{thm:FBMmuProba}, in Subsection
5.1.1, under the hypothesis that for $1\leq i\leq d$, $\mu_{i}$
has a square integrable density with respect to the Lebesgue measure,
we prove existence to a variational formulation slightly stronger
than $V_{Fin}(\mu)$. Then, in Subsection 5.1.2, when $\mu$ is a
general probability measure on $\mathbb{R}\times\mathcal{Y}$, we
mollify $\mu_{i}$ for $1\leq i\leq d$, in order to use the results
of Subsection 5.1.1 and obtain a solution to $V_{Fin}(\mu).$

\subsubsection{Case when $\mu$ has square integrable densities}

In this section, the measures $\left(\mu_{i}\right)_{1\leq i\leq d}$
are assumed to have square integrable densities which are also denoted
by $\mu_{i}$ for notational simplicity. We define $p_{0}:=\left(\alpha_{1}\mu_{1},...,\alpha_{d}\mu_{d}\right)\in L$.
We define the variational formulation $V_{Fin,L^{2}}(\mu)$:
\begin{eqnarray}
\text{Find} &  & p=(p_{1},...,p_{d})\ \text{satisfying :}\label{eq:findFinL2}
\end{eqnarray}
\begin{equation}
p\in L^{2}([0,T];H)\cap L^{\infty}([0,T];L)\text{ and takes values in \ensuremath{\mathcal{D}} a.e. on \ensuremath{(0,T)\times\mathbb{R}},}\label{eq:pL2FinL2}
\end{equation}

\begin{equation}
\forall v\in H,\ \frac{d}{dt}(v,p)_{d}-r\left(\partial_{x}v,p\right)_{d}+\left(\partial_{x}v,\frac{1}{2}R(p)\tilde{\sigma}_{Dup}\left(\tilde{\sigma}_{Dup}+2\partial_{x}\tilde{\sigma}_{Dup}\right)\Lambda p\right)_{d}+(\partial_{x}v,\tilde{\sigma}_{Dup}^{2}A(p)\partial_{x}p)_{d}=(Qv,p)_{d}\label{eq:fvarFinL2}
\end{equation}
\[
\text{ in the sense of distributions on \ensuremath{(0,T)},}
\]

\begin{equation}
p(0,\cdot)=p_{0}.\label{eq:cond0FinL2}
\end{equation}
To show existence to $V_{Fin,L^{2}}(\mu)$, we use Galerkin's procedure,
as in the proof of Theorem \ref{thm:FBMmuProba}. It is not obvious
that $p$ takes values in $\mathcal{D}$ a.e. on $(0,T)\times\mathbb{R}$
at the discrete level, that is why for $\epsilon>0$, we define, for
$\rho\in\left(\mathbb{R}_{+}\right)^{d}$, $R_{\epsilon}(\rho)=\frac{\sum_{i=1}^{d}\rho_{i}}{\epsilon\vee\left(\sum_{i=1}^{d}\lambda_{i}\rho_{i}\right)}$
and we introduce the variational formulation $V_{Fin,\epsilon}(\mu)$:
\begin{eqnarray}
\text{Find} &  & p_{\epsilon}=(p_{\epsilon,1},...,p_{\epsilon,d})\ \text{satisfying :}\label{eq:findFinL2eps}
\end{eqnarray}
\begin{equation}
p_{\epsilon}\in L^{2}([0,T];H)\cap L^{\infty}([0,T];L)\label{eq:pL2FinL2eps}
\end{equation}

\begin{equation}
\forall v\in H,\ \frac{d}{dt}(v,p_{\epsilon})_{d}-r\left(\partial_{x}v,p_{\epsilon}\right)_{d}+\left(\partial_{x}v,\frac{1}{2}R_{\epsilon}(p_{\epsilon}^{+})\tilde{\sigma}_{Dup}\left(\tilde{\sigma}_{Dup}+\partial_{x}\tilde{\sigma}_{Dup}\right)\Lambda p_{\epsilon}\right)_{d}+\left(\partial_{x}v,\tilde{\sigma}_{Dup}^{2}A_{\epsilon}(p_{\epsilon}^{+})\partial_{x}p_{\epsilon}\right)_{d}=(Qv,p_{\epsilon})_{d}\label{eq:fvarFinL2eps}
\end{equation}
\[
\text{ in the sense of distributions on \ensuremath{(0,T)},}
\]

\begin{equation}
p_{\epsilon}(0,\cdot)=p_{0}.\label{eq:cond0FinL2eps}
\end{equation}
Let us remark that if $p$ (resp. $p_{\epsilon}$) satisfies (\ref{eq:findFinL2})-(\ref{eq:fvarFinL2})
(resp. (\ref{eq:findFinL2eps})-(\ref{eq:fvarFinL2eps})), those conditions
imply that $\frac{dp}{dt}$ (resp. $\frac{dp_{\epsilon}}{dt}$) belongs
to $L^{2}([0,T],H')$ in the sense of \cite[III, Lemma 1.1]{temam},
so by \cite[III, Lemma 1.2]{temam}, $p$ (resp. $p_{\epsilon}$)
is equal a.e. on $[0,T]$ to a function of $C([0,T],L)$, so that
the initial condition (\ref{eq:cond0FinL2}) (resp. \ref{eq:cond0FinL2eps})
makes sense. 

To take advantage of the fact that under Condition $(C)$, there exists
$\Pi\in S_{d}^{++}(\mathbb{R})$ and $\kappa>0$ such that $\Pi A$
and $\Pi A_{\epsilon}$, for $\epsilon>0$, are uniformly coercive
on $\mathcal{D}$ with the coefficient $\kappa$, we introduce, for
$\epsilon>0$ and $m\geq1$, the approximate variational formulation
$V_{Fin,\epsilon}^{m}(\mu)$: 
\[
\text{Find \ensuremath{g_{\epsilon,1}^{m},...,g_{\epsilon,m}^{m}}}\in C^{0}([0,T],\mathbb{R}),\text{ such that:}
\]
\[
\text{the function }t\in[0,T]\rightarrow p_{\epsilon}^{m}(t)=\sum_{j=1}^{m}g_{\epsilon,j}^{m}(t)w_{j}\text{ satisfies, for \ensuremath{1\leq i\leq m,}}
\]
\begin{eqnarray}
\left(Q\Pi w_{i}+r\Pi\partial_{x}w_{i},p_{\epsilon}^{m}(t)\right)_{d} & = & \frac{d}{dt}\left(w_{i},\Pi p_{\epsilon}^{m}(t)\right)_{d}+\left(\partial_{x}w_{i},\tilde{\sigma}_{Dup}^{2}\Pi A_{\epsilon}\left(\left(p_{\epsilon}^{m}\right)^{+}(t)\right)\partial_{x}p_{\epsilon}^{m}(t)\right)_{d}\nonumber \\
 & \ + & \frac{1}{2}\left(\partial_{x}w_{i},R_{\epsilon}\left(\left(p_{\epsilon}^{m}\right)^{+}(t)\right)\tilde{\sigma}_{Dup}\left(\tilde{\sigma}_{Dup}+2\partial_{x}\tilde{\sigma}_{Dup}\right)\Pi\Lambda p_{\epsilon}^{m}(t)\right)_{d}\label{eq:VFinEspVar}
\end{eqnarray}
\[
p_{\epsilon}^{m}(0)=p_{0}^{m},
\]
where $p_{0}^{m}$ is the orthogonal projection of $p_{0}$ in $L$
on the space spanned by $(w_{j})_{1\le j\leq m}$. For $z\in\mathbb{R}^{m}$
and $t\geq0$, we define $K_{\epsilon,1}^{m}(t,z)$ the matrix, where
for $1\leq i,j\leq m$: 
\[
K_{\epsilon,1,ij}^{m}(t,z)=\frac{1}{2}\left(\partial_{x}w_{i},R_{\epsilon}\left(\left(\sum_{k=1}^{m}z_{k}w_{k}\right)^{+}\right)\tilde{\sigma}_{Dup}(t)\left(\tilde{\sigma}_{Dup}(t)+2\partial_{x}\tilde{\sigma}_{Dup}(t)\right)\Pi\Lambda w_{j}\right)_{d},
\]
$K_{\epsilon,2}^{m}(t,z)$ the matrix where for $1\leq i,j\leq m$,
\[
K_{\epsilon,2,ij}^{m}(t,z)=\left(\partial_{x}w_{i},\tilde{\sigma}_{Dup}^{2}(t)\Pi A_{\epsilon}\left(\left(\sum_{k=1}^{m}z_{k}w_{k}\right)^{+}\right)\partial_{x}w_{j}\right)_{d},
\]
and $K_{\epsilon,3}^{m}$ the constant matrix where for $1\leq i,j\leq m$:
\[
K_{\epsilon,3,ij}^{m}=(Q\Pi w_{i}+r\Pi\partial_{x}w_{i},w_{j})_{d}.
\]
We then define $F_{\epsilon,k}^{m}(t,z):=\left(W^{(m)}\right)^{-1}K_{\epsilon,k}^{m}(t,z)z$,
for $k=1,2,$ and $F_{\epsilon,3}^{m}(z):=\left(W^{(m)}\right)^{-1}K_{\epsilon,3}^{m}z$,
Finally we define for $z\in\mathbb{R}^{m}$, $G_{\epsilon}^{m}(t,z):=-F_{\epsilon,1}^{m}(t,z)-F_{\epsilon,2}^{m}(t,z)+F_{\epsilon,3}^{m}(z)$.
The ODE for $g_{\epsilon}^{m}$ rewrites:
\begin{eqnarray*}
\left(g_{\epsilon}^{m}\right)'(t) & = & G_{\epsilon}^{m}(t,g_{\epsilon}^{m}(t))\\
g_{\epsilon}^{m}(0) & = & g_{\epsilon,0}^{m},
\end{eqnarray*}
where the vector $g_{\epsilon,0}^{m}$ is the expression of $p_{0}^{m}$
on the basis $\left(w_{i}\right)_{1\leq i\leq m}$. We prove existence
and uniqueness to $V_{Fin,\epsilon}^{m}(\mu)$. We clearly have the
following lemma. 
\begin{lem}
\label{lem:Rbounded}The functions $R$ and $R_{\epsilon}$, for $\epsilon>0$,
are uniformly bounded. More precisely, $||R_{\epsilon}||_{L^{\infty}\left(\left(\mathbb{R}_{+}\right)^{d}\right)}\leq\frac{1}{\lambda_{min}}$,
$||R||_{L^{\infty}\left(\mathcal{D}\right)}\leq\frac{1}{\lambda_{min}}$. 
\end{lem}

\begin{lem}
\label{lem:Kbounded}Under Assumption (B), for $m\geq1$, the functions
$K_{\epsilon,1}^{m}$ and $K_{\epsilon,2}^{m}$ are uniformly bounded.\end{lem}
\begin{proof}
Using Assumption $(B)$, Lemma \ref{lem:Rbounded} and Lemma \ref{lem:Bepsbounded},
for $t\in[0,T]$, $x\in\mathbb{R}$, $\rho\in\mathcal{D}$, we have

\begin{eqnarray*}
\left|\frac{1}{2}R_{\epsilon}(\rho)\tilde{\sigma}_{Dup}(t,x)\left(\tilde{\sigma}_{Dup}(t,x)+2\partial_{x}\tilde{\sigma}_{Dup}(t,x)\right)\right| & \leq & \frac{1}{2\lambda_{min}}\overline{\sigma}(\overline{\sigma}+2||\partial_{x}\tilde{\sigma}_{Dup}||_{\infty}),\\
||A_{\epsilon}(\rho)\tilde{\sigma}_{Dup}^{2}(t,x)||_{\infty} & \leq & \frac{1}{2}\left(1+\frac{\lambda_{max}}{\lambda_{min}}\right)\overline{\sigma}^{2}.
\end{eqnarray*}
This is sufficient to show that $K_{\epsilon,1}^{m}$ and $K_{\epsilon,2}^{m}$
are uniformly bounded, as the functions $\left(w_{ik}\partial_{x}w_{jl}\right)_{1\leq i,j\leq m,1\leq k,l\leq d}$
and $\left(\partial_{x}w_{ik}\partial_{x}w_{jl}\right)_{1\leq i,j\leq m,1\leq k,l\leq d}$
belong to $L^{1}(\mathbb{R})$. \end{proof}
\begin{lem}
\label{lem:F1F2localLip}Under Assumption $(B)$, for $m\geq1$, the
functions $F_{\epsilon,1}^{m}$ and $F_{\epsilon,2}^{m}$ are locally
Lipschitz in $z$, uniformly in $t\in[0,T]$.
\end{lem}
The proof of Lemma \ref{lem:F1F2localLip} is similar to the proof
of Lemma \ref{lem:FespLocalLipschitz} and is postponed to Appendix
A.

\begin{lem}
\label{lem:VFinEpsM}Under Assumption (B), $V_{Fin,\epsilon}^{m}(\mu)$
has a unique solution.\end{lem}
\begin{proof}
In addition to Lemma \ref{lem:F1F2localLip}, $F_{\epsilon,3}$ is
clearly a Lipschitz function. Therefore the function $G_{\epsilon}$
is locally Lipschitz in $z$ uniformly in $t$. Caratheodory's theorem
(see e.g. \cite[Theorems 5.2, 5.3]{hale2009ordinary}) gives the existence
of a unique maximal absolutely continuous solution $g_{\epsilon}^{m}$
on an interval $[0,T^{*})$, with $T^{*}>\text{0}$. In addition,
as the function $(t,z)\rightarrow-W^{-1}K_{\epsilon,1}(t,z)-W^{-1}K_{\epsilon,2}(t,z)+W^{-1}K_{\epsilon,3}$
is uniformly bounded on $\mathbb{R}_{+}\times\mathbb{R}^{m}$, we
conclude as in the proof of Lemma \ref{lem:existencetoODEs}, using
Gronwall's lemma, that $T^{*}=\infty$. Consequently, $g_{\epsilon}^{m}$
is defined on $[0,T]$, and there exists a unique solution to $V_{Fin,\epsilon}^{m}(\mu)$.
\end{proof}
We now compute energy estimates on the solution $p_{\epsilon}^{m}$
to $V_{Fin,\epsilon}^{m}(\mu)$, for $m\ge1$ and $\epsilon>0$. Taking
$p_{\epsilon}^{m}$ as a test function in (\ref{eq:VFinEspVar}),
we have
\begin{eqnarray*}
\frac{1}{2}\frac{d}{dt}|\sqrt{\Pi}p_{\epsilon}^{m}|_{d}^{2}-(Q\Pi p_{\epsilon}^{m},p_{\epsilon}^{m})_{d} & = & -\left(\partial_{x}p_{\epsilon}^{m},\tilde{\sigma}_{Dup}^{2}\Pi A_{\epsilon}\left(\left(p_{\epsilon}^{m}\right)^{+}\right)\partial_{x}p_{\epsilon}^{m}\right)_{d}\\
 & \ - & \left(\partial_{x}p_{\epsilon}^{m},\left(\frac{1}{2}R_{\epsilon}(\left(p_{\epsilon}^{m}\right)^{+}(t))\tilde{\sigma}_{Dup}\left(\tilde{\sigma}_{Dup}+\partial_{x}\tilde{\sigma}_{Dup}\right)\Pi\Lambda-r\Pi\right)p_{\epsilon}^{m}(t)\right)_{d}.
\end{eqnarray*}
For $\eta>0$, by Young's inequality we have that 
\begin{equation}
\left|\left(\partial_{x}p_{\epsilon}^{m},\left(\frac{1}{2}R_{\epsilon}(\left(p_{\epsilon}^{m}\right)^{+})\tilde{\sigma}_{Dup}\left(\tilde{\sigma}_{Dup}+2\partial_{x}\tilde{\sigma}_{Dup}\right)\Pi\Lambda-r\Pi\right)p_{\epsilon}^{m}\right)_{d}\right|\leq C\left(\eta|\partial_{x}p_{\epsilon}^{m}|_{d}^{2}+\frac{1}{4\eta}|p_{\epsilon}^{m}|_{d}^{2}\right),\label{eq:Young}
\end{equation}
where $C:=d||\Pi||_{\infty}\left(\frac{\lambda_{max}}{2\lambda_{min}}\overline{\sigma}(\overline{\sigma}+2||\partial_{x}\tilde{\sigma}||_{\infty})+r\right)$.
As for $\epsilon>0$, $\Pi A_{\epsilon}$ is uniformly coercive with
coefficient $\kappa$, and $\tilde{\sigma}_{Dup}$ is bounded from
below by $\underline{\sigma}>0$, so
\begin{eqnarray*}
\left(\partial_{x}p_{\epsilon}^{m},\tilde{\sigma}_{Dup}^{2}\Pi A_{\epsilon}\left(\left(p_{\epsilon}^{m}\right)^{+}\right)\partial_{x}p_{\epsilon}^{m}\right)_{d} & \geq & \kappa\underline{\sigma}^{2}|\partial_{x}p_{\epsilon}^{m}|_{d}^{2}.
\end{eqnarray*}
We then choose $\eta=\frac{\kappa}{2C}\underbar{\ensuremath{\sigma}}^{2}$,
so that $C\eta=\frac{\kappa}{2}\underline{\sigma}^{2}$. Moreover,
for $b:=d^{2}\overline{q}||\Pi||_{\infty}$ we have that 
\[
\forall\xi\in\mathbb{R}^{d},|\xi^{*}Q\Pi\xi|\leq b\xi^{*}\xi.
\]
Gathering the previous inequalities we have that 
\begin{eqnarray}
\frac{1}{2}\frac{d}{dt}|\sqrt{\Pi}p_{\epsilon}^{m}|_{d}^{2}-(Q\Pi p_{\epsilon}^{m},p_{\epsilon}^{m})_{d} & \leq & -\kappa\underline{\sigma}^{2}|\partial_{x}p_{\epsilon}^{m}|_{d}^{2}+\frac{\kappa}{2}\underbar{\ensuremath{\sigma}}^{2}|\partial_{x}p_{\epsilon}^{m}|_{d}^{2}+\frac{C^{2}}{2\kappa\underline{\sigma}^{2}}|p_{\epsilon}^{m}|_{d}^{2}.\label{eq:FinenergyODE}
\end{eqnarray}
Consequently, we have

\[
\frac{1}{2}\frac{d}{dt}|\sqrt{\Pi}p_{\epsilon}^{m}|_{d}^{2}-\left(b+\frac{C^{2}}{2\kappa\underline{\sigma}^{2}}\right)|p_{\epsilon}^{m}|_{d}^{2}\leq-\frac{\kappa}{2}\underline{\sigma}^{2}|\partial_{x}p_{\epsilon}^{m}|_{d}^{2}\leq0.
\]
As $|p_{\epsilon}^{m}|_{d}^{2}\leq\frac{1}{l_{min}(\Pi)}|\sqrt{\Pi}p_{\epsilon}^{m}|_{d}^{2}$,
we also have
\begin{equation}
\frac{1}{2}\frac{d}{dt}|\sqrt{\Pi}p_{\epsilon}^{m}|_{d}^{2}-\frac{1}{l_{min}(\Pi)}\left(b+\frac{C^{2}}{2\kappa\underline{\sigma}^{2}}\right)|\sqrt{\Pi}p_{\epsilon}^{m}|_{d}^{2}\leq-\frac{\kappa}{2}\underline{\sigma}^{2}|\partial_{x}p_{\epsilon}^{m}|_{d}^{2}\leq0.\label{eq:FinanceEstimODE}
\end{equation}
Integrating the inequality (\ref{eq:FinanceEstimODE}), and using
the fact that 
\[
l_{min}(\Pi)|p_{\epsilon}^{m}|_{d}^{2}\leq|\sqrt{\Pi}p_{\epsilon}^{m}|_{d}^{2}\leq l_{max}(\Pi)|p_{\epsilon}^{m}|_{d}^{2},
\]
we obtain the following lemma.
\begin{lem}
\label{lem:EstimatesFin}The following energy estimates hold.
\begin{eqnarray}
\underset{t\in[0,T]}{\sup}|p_{\epsilon}^{m}(t)|_{d}^{2} & \leq & \frac{l_{max}(\Pi)}{l_{min}(\Pi)}e^{\frac{2}{l_{min}(\Pi)}\left(b+\frac{C^{2}}{2\kappa\underline{\sigma}^{2}}\right)T}|p_{0}|_{d}^{2},\label{eq:energyestFinM}\\
\int_{0}^{T}|\partial_{x}p_{\epsilon}^{m}(t)|_{d}^{2}dt & \leq & \frac{l_{max}(\Pi)}{\kappa\underbar{\ensuremath{\sigma}}^{2}}e^{\frac{2}{l_{min}(\Pi)}\left(b+\frac{C^{2}}{2\kappa\underline{\sigma}^{2}}\right)T}|p_{0}|_{d}^{2},\label{eq:energyest2FinM}\\
\int_{0}^{T}||p_{\epsilon}^{m}||_{d}^{2}dt & \leq & \left(T\frac{l_{max}(\Pi)}{l_{min}(\Pi)}+\frac{l_{max}(\Pi)}{\kappa\underbar{\ensuremath{\sigma}}^{2}}\right)e^{\frac{2}{l_{min}(\Pi)}\left(b+\frac{C^{2}}{2\kappa\underline{\sigma}^{2}}\right)T}|p_{0}|_{d}^{2}.\label{eq:energyest3FinM}
\end{eqnarray}

\end{lem}
Now we can prove existence to $V_{Fin,\epsilon}(\mu)$.
\begin{prop}
There exists a solution $p_{\epsilon}$ to $(V_{Fin,\epsilon}(\mu))$.\end{prop}
\begin{proof}
Given $\epsilon>0$, the family $\left(p_{\epsilon}^{m}\right)_{m\geq0}$
has standard energy estimates with bounds independent from $m$. We
also check that (\ref{eq:VFinEspVar}) rewrites for $1\leq j\leq m$
\[
\frac{d}{dt}\left(w_{j},\Pi p_{\epsilon}^{m}\right)_{d}+\langle w_{j},Gp_{\epsilon}^{m}\rangle=0
\]
where for $q\in H$, $G_{\epsilon}q\in H'$, and 
\begin{eqnarray*}
\forall v\in H,\langle v,G_{\epsilon}q\rangle & = & -r(\partial_{x}v,\Pi q)_{d}+\left(\partial_{x}v,\left(\frac{1}{2}R_{\epsilon}\left(q^{+}\right)\tilde{\sigma}_{Dup}\left(\tilde{\sigma}_{Dup}+2\partial_{x}\tilde{\sigma}_{Dup}\right)\right)\Pi\Lambda p\right)_{d}\\
 & \ + & \left(\partial_{x}v,\tilde{\sigma}_{Dup}^{2}\Pi A_{\epsilon}\left(q^{+}\right)\partial_{x}q\right)_{d}-(Q\Pi v,q)_{d}.
\end{eqnarray*}
We see that for almost every $t\in[0,T]$, 
\begin{equation}
||G_{\epsilon}p(t)||_{H'}\leq d||\Pi||_{\infty}\left(r+\frac{\lambda_{max}}{2\lambda_{min}}\overline{\sigma}(\overline{\sigma}+2||\partial_{x}\tilde{\sigma}_{Dup}||_{\infty})+\frac{d}{2}\left(1+\frac{\lambda_{max}}{\lambda_{min}}\right)\overline{\sigma}^{2}+d\overline{q}\right)||p(t)||_{H}.\label{eq:GpFin}
\end{equation}
By Lemma \ref{lem:EstimatesFin} the family $\left(p_{\epsilon}^{m}\right)_{m\geq1}$
is bounded in $L^{2}([0,T],H)$ and through the equality (\ref{eq:GpFin})
the family $\left(G_{\epsilon}p_{\epsilon}^{m}\right)_{m\geq1}$ is
bounded in $L^{2}([0,T],H')$. What follows is a repetition of arguments
used in the proof of Proposition \ref{prop:ExistencetoVeps}. We extract
a subsequence of $(p_{\epsilon}^{m})_{m\geq1}$, also denoted by $(p_{\epsilon}^{m})_{m\geq1}$
, and a function $p_{\epsilon}\geq0$ such that as $m\rightarrow\infty$,
\begin{eqnarray*}
p_{\epsilon}^{m} & \rightarrow & p_{\epsilon}\text{ {in} }L^{2}([0,T];H)\text{ weakly},\\
p_{\epsilon}^{m} & \rightarrow & p_{\epsilon}\text{ {in} }L^{\infty}([0,T];L)\text{ weakly-*,}\\
p_{\epsilon}^{m} & \rightarrow & p_{\epsilon}\text{ a.e. on \ensuremath{(0,T]\times\mathbb{R}}}.
\end{eqnarray*}
For $j\geq1$, $\psi\in C^{1}(\mathbb{R})$, with $\psi(T)=0$, and
$m\geq j$, we have: 
\begin{eqnarray*}
\int_{0}^{T}\psi(t)(Q\Pi w_{j}+r\Pi\partial_{x}w_{j},p_{\epsilon}^{m}(t))_{d}dt & = & -\int_{0}^{T}\psi'(t)(w_{j},\Pi p_{\epsilon}^{m}(t))_{d}dt-\psi(0)(w_{j},\Pi p_{\epsilon}^{m}(0))_{d}\\
 & \ + & \int_{0}^{T}\psi(t)\left(\partial_{x}w_{j},\tilde{\sigma}_{Dup}^{2}(t)\Pi A_{\epsilon}\left(\left(p_{\epsilon}^{m}\right)^{+}(t)\right)\partial_{x}p_{\epsilon}^{m}(t)\right)_{d}dt\\
 & \ + & \frac{1}{2}\int_{0}^{T}\psi(t)\left(\partial_{x}w_{j},R_{\epsilon}\left(\left(p_{\epsilon}^{m}\right)^{+}(t)\right)\tilde{\sigma}_{Dup}\left(\tilde{\sigma}_{Dup}+\partial_{x}\tilde{\sigma}_{Dup}\right)\Pi\Lambda p_{\epsilon}^{m}(t)\right)_{d}dt.
\end{eqnarray*}
The following convergences hold, using the same arguments as in the
proof of Proposition \ref{prop:ExistencetoVeps}.

\begin{eqnarray*}
\psi(0)(w_{j},\Pi p_{\epsilon}^{m}(0))_{d} & \underset{m\rightarrow\infty}{\rightarrow} & \psi(0)(w_{j},\Pi p_{0})_{d},\\
-\int_{0}^{T}\psi'(t)(w_{j},\Pi p_{\epsilon}^{m}(t))_{d}dt & \underset{m\rightarrow\infty}{\rightarrow} & -\int_{0}^{T}\psi'(t)(w_{j},\Pi p_{\epsilon}(t))_{d}dt,\\
\int_{0}^{T}\psi(t)(\partial_{x}w_{j},\tilde{\sigma}_{Dup}^{2}(t)\Pi A_{\epsilon}\left(\left(p_{\epsilon}^{m}\right)^{+}(t)\right)\partial_{x}p_{\epsilon}^{m}(t))_{d}dt & \underset{m\rightarrow\infty}{\rightarrow} & \int_{0}^{T}\psi(t)(\partial_{x}w_{j},\tilde{\sigma}_{Dup}^{2}(t)\Pi A_{\epsilon}\left(\left(p_{\epsilon}\right)^{+}(t)\right)\partial_{x}p_{\epsilon}(t))_{d}dt,
\end{eqnarray*}
We only explain how to deal with the additional terms. As $p_{\epsilon}^{m}\rightarrow p_{\epsilon}$
weakly in $L^{2}([0,T];L)$, we have that:
\begin{eqnarray*}
\int_{0}^{T}\psi(t)(Q\Pi w_{j}+r\Pi\partial_{x}w_{j},p_{\epsilon}^{m}(t))_{d}dt & \underset{m\rightarrow\infty}{\rightarrow} & \int_{0}^{T}\psi(t)(Q\Pi w_{j}+r\Pi\partial_{x}w_{j},p_{\epsilon}(t))_{d}dt,
\end{eqnarray*}
As the function $R_{\epsilon}$ is continuous and bounded on $\left(\mathbb{R}_{+}\right)^{d}$,
and as $p_{\epsilon}^{m}\underset{m\rightarrow\infty}{\rightarrow}p_{\epsilon}\text{ a.e. on \ensuremath{(0,T]\times\mathbb{R}}},$
we have that $R_{\epsilon}\left(\left(p_{\epsilon}^{m}\right)^{+}\right)\underset{m\rightarrow\infty}{\rightarrow}R_{\epsilon}\left(\left(p_{\epsilon}\right)^{+}\right)$
a.e. on $[0,T]\times\mathbb{R}$. As the bounds (\ref{eq:energyestFinM})-(\ref{eq:energyest3FinM})
are independent from $m$, we show, with the same arguments as for
(\ref{eq:ToyEpsFVconvergence}), that:

\begin{eqnarray*}
 &  & \int_{0}^{T}\psi(t)\left(\partial_{x}w_{j},R_{\epsilon}\left(\left(p_{\epsilon}^{m}\right)^{+}(t)\right)\tilde{\sigma}_{Dup}\left(\tilde{\sigma}_{Dup}+2\partial_{x}\tilde{\sigma}_{Dup}\right)\Pi\Lambda p_{\epsilon}^{m}(t)\right)_{d}dt\\
 &  & \underset{m\rightarrow\infty}{\rightarrow}\int_{0}^{T}\psi(t)\left(\partial_{x}w_{j},R_{\epsilon}\left(\left(p_{\epsilon}\right)^{+}(t)\right)\tilde{\sigma}_{Dup}\left(\tilde{\sigma}_{Dup}+2\partial_{x}\tilde{\sigma}_{Dup}\right)\Pi\Lambda p_{\epsilon}(t)\right)_{d}dt.
\end{eqnarray*}
As the family $\left(w_{j}\right)_{j\geq1}$ is total in $H$, and
the function $v\rightarrow\Pi v$ is a bijection from $H$ to $H$,
we conclude that:

\begin{eqnarray*}
\forall v\in H,\int_{0}^{T}\psi(t)(Qv+r\partial_{x}v,p_{\epsilon}(t))_{d}dt & = & -\int_{0}^{T}\psi'(t)(v,p_{\epsilon}(t))_{d}dt-\psi(0)(v,p_{0})\\
 & \ + & \int_{0}^{T}\psi(t)(\partial_{x}v,A_{\epsilon}\left(\left(p_{\epsilon}\right)^{+}(t)\right)\tilde{\sigma}_{Dup}^{2}(t)\partial_{x}p_{\epsilon}(t))_{d}dt\\
 & \ + & \frac{1}{2}\int_{0}^{T}\psi(t)\left(\partial_{x}v,R_{\epsilon}\left(\left(p_{\epsilon}\right)^{+}(t)\right)\tilde{\sigma}_{Dup}\left(\tilde{\sigma}_{Dup}+2\partial_{x}\tilde{\sigma}_{Dup}\right)\Lambda p_{\epsilon}(t)\right)_{d}dt.
\end{eqnarray*}
Thus we have checked that $p_{\epsilon}$ satisfies (\ref{eq:fvarFinL2eps}).
With the same arguments as the end of the proof of Proposition \ref{prop:ExistencetoVeps},
we can show that for $v\in H$, the function $t\rightarrow(v,p_{\epsilon}(t))_{d}$
belongs to $H^{1}(0,T)$, and that $p_{\epsilon}$ satisfies the initial
condition (\ref{eq:cond0FinL2eps}), so that $p_{\epsilon}$ is a
solution to $V_{Fin,\epsilon}(\mu)$.
\end{proof}
We now show that the solutions of $V_{Fin,\epsilon}(\mu)$ are non
negative.
\begin{prop}
\label{lem:VFinnonnegative}The solutions of $V_{Fin,\epsilon}(\mu)$
are non negative. \end{prop}
\begin{proof}
Let us take $p_{\epsilon}^{-}$ as a test function in (\ref{eq:fvarFinL2eps}).
We obtain
\[
\frac{1}{2}\frac{d}{dt}|p_{\epsilon}^{-}|_{d}^{2}+\left(\partial_{x}p_{\epsilon}^{-},\left(\frac{1}{2}R_{\epsilon}(p_{\epsilon}^{+})\tilde{\sigma}_{Dup}\left(\tilde{\sigma}_{Dup}+2\partial_{x}\tilde{\sigma}_{Dup}\right)\Lambda-rI_{d}\right)p_{\epsilon}\right)_{d}+\left(\partial_{x}p_{\epsilon}^{-},\tilde{\sigma}_{Dup}^{2}A_{\epsilon}(p_{\epsilon}^{+})\partial_{x}p_{\epsilon}\right)_{d}=(Qp_{\epsilon}^{-},p_{\epsilon})_{d}
\]
In the proof of Proposition \ref{prop:solutionarenonnegative}, we
have shown that $A_{\epsilon,ii}(p_{\epsilon}^{+})\left(\partial_{x}p_{\epsilon,i}^{-}\right)^{2}\geq\frac{\lambda_{min}}{2\lambda_{max}}\left(\partial_{x}p_{\epsilon,i}^{-}\right)^{2}$,
so 
\[
\left(\partial_{x}p_{\epsilon}^{-},\tilde{\sigma}_{Dup}^{2}A_{\epsilon}\left(p_{\epsilon}^{+}\right)\partial_{x}p_{\epsilon}\right)_{d}\geq\underline{\sigma}^{2}\frac{\lambda_{min}}{2\lambda_{max}}|\partial_{x}p_{\epsilon}^{-}|_{d}^{2}.
\]
By the Young inequality, for $\eta>0$,
\[
\left(\partial_{x}p_{\epsilon}^{-},\left(\frac{1}{2}R_{\epsilon}(p_{\epsilon}^{+})\tilde{\sigma}_{Dup}\left(\tilde{\sigma}_{Dup}+2\partial_{x}\tilde{\sigma}_{Dup}\right)\Lambda-rI_{d}\right)p_{\epsilon}\right)_{d}\geq-K\left(\eta|\partial_{x}p_{\epsilon}^{-}|_{d}^{2}+\frac{1}{4\eta}|p_{\epsilon}^{-}|_{d}^{2}\right),
\]
where $K=\frac{1}{2}\frac{\lambda_{max}}{\lambda_{min}}\overline{\sigma}\left(\overline{\sigma}+2||\partial_{x}\tilde{\sigma}_{Dup}||_{\infty}\right)+r$.
We set $\eta=\underline{\sigma}^{2}\frac{\lambda_{min}}{4K\lambda_{max}}$,
so that
\begin{equation}
\frac{1}{2}\frac{d}{dt}|p_{\epsilon}^{-}|_{d}^{2}-(Qp_{\epsilon}^{-},p_{\epsilon}^{-})_{d}-\frac{K^{2}}{\underline{\sigma}^{2}}\frac{\lambda_{max}}{\lambda_{min}}|p_{\epsilon}^{-}|_{d}^{2}\leq(Qp_{\epsilon}^{-},p_{\epsilon}^{+})_{d}-\underline{\sigma}^{2}\frac{\lambda_{min}}{4\lambda_{max}}|\partial_{x}p_{\epsilon}^{-}|_{d}^{2}.\label{eq:inegalite5.7}
\end{equation}
We also check that $(Qp_{\epsilon}^{-},p_{\epsilon}^{+})_{d}\leq0$.
Indeed, as for $i\in\{1,...,d\}$, $p_{\epsilon,i}^{-}p_{\epsilon,i}^{+}=0$
and $q_{ij}\geq0$ for $j\neq i$, 
\[
(Qp_{\epsilon}^{-},p_{\epsilon}^{+})_{d}=\int_{\mathbb{R}}\sum_{i=1}^{d}p_{\epsilon,i}^{+}\left(\sum_{j=1}^{d}q_{ij}p_{\epsilon,j}^{-}\right)dx=\int_{\mathbb{R}}\sum_{i\neq j}q_{ij}p_{\epsilon,j}^{-}p_{\epsilon,i}^{+}dx\leq0,
\]
so the r.h.s of (\ref{eq:inegalite5.7}) is nonpositive. Moreover,
$(Qp_{\epsilon}^{-},p_{\epsilon}^{-})_{d}\leq d\overline{q}|p_{\epsilon}^{-}|_{d}^{2}$,
so we have the inequality
\[
\frac{1}{2}\frac{d}{dt}|p_{\epsilon}^{-}|_{d}^{2}-\left(d\overline{q}+\frac{K^{2}}{\underline{\sigma}^{2}}\frac{\lambda_{max}}{\lambda_{min}}\right)|p_{\epsilon}^{-}|_{d}^{2}\leq0
\]
We thus obtain that the function $t\rightarrow\exp\left(-2\left(d\overline{q}+\frac{K^{2}}{\underline{\sigma}^{2}}\frac{\lambda_{max}}{\lambda_{min}}\right)t\right)|p_{\epsilon}^{-}(t)|_{d}^{2}$
is non increasing. As $p_{\epsilon}^{-}(0)=0$, we can conclude that
$p_{\epsilon}^{-}\equiv0$.
\end{proof}
We check that if $p$ is the limit of a sequence $\left(p_{\epsilon_{k}}\right)_{k\geq1}$,
with $p_{\epsilon_{k}}$ being a solution of $V_{Fin,\epsilon_{k}}(\mu)$
for $k\geq1$, then $p$ takes values in $\mathcal{D}$. 
\begin{lem}
\label{lem:5.8}Let $\left(p_{\epsilon_{k}}\right)_{k\geq1}$ be a
sequence such that for $k\geq1$, $p_{\epsilon_{k}}$ is a solution
to $V_{Fin,\epsilon_{k}}(\mu)$, and $\epsilon_{k}\underset{k\rightarrow\infty}{\rightarrow}0$.
If the sequence $\left(p_{\epsilon_{k}}\right)_{k\geq1}$ has the
limit $p$, in the sense: 
\begin{eqnarray*}
p_{\epsilon_{k}} & \rightarrow & p\text{ {in} }L^{2}([0,T];H)\text{ weakly},\\
p_{\epsilon_{k}} & \rightarrow & p\text{ {in} }L^{\infty}([0,T];L)\text{ weakly-*},\\
p_{\epsilon_{k}} & \rightarrow & p\text{ a.e. on \ensuremath{(0,T]\times\mathbb{R}}},
\end{eqnarray*}

then under Assumptions $(B)$ and $(H)$, $\sum_{i=1}^{d}p_{i}$ is
the unique solution to $LV(\mu_{X_{0}})$ and $p$ takes values in
$\mathcal{D}$.\end{lem}
\begin{proof}
For $k\geq1$, we define $u_{\epsilon_{k}}:=\sum_{i=1}^{d}p_{\epsilon_{k},i}$,
and $u:=\sum_{i=1}^{d}p_{i}$. Then $u_{\epsilon_{k}}$ satisfies,
for $\psi\in C_{c}^{\infty}((0,T),\mathbb{R})$ and $v\in H^{1}(\mathbb{R})$,
\begin{eqnarray*}
0 & = & -\int_{0}^{T}\psi'(t)(v,u_{\epsilon_{k}})_{1}dt-\int_{0}^{T}\psi(t)r(\partial_{x}v,u_{\epsilon_{k}})_{1}dt+\int_{0}^{T}\psi(t)\frac{1}{2}\left(\partial_{x}v,\tilde{\sigma}_{Dup}^{2}\partial_{x}u_{\epsilon_{k}}\right)_{1}dt\\
 & \ + & \int_{0}^{T}\psi(t)\left(\partial_{x}v,\frac{1}{2}\tilde{\sigma}_{Dup}(\tilde{\sigma}_{Dup}+2\left(\partial_{x}\tilde{\sigma}_{Dup}\right))\frac{\sum_{i=1}^{d}\lambda_{i}p_{\epsilon_{k},i}}{\epsilon_{k}\vee\left(\sum_{i=1}^{d}\lambda_{i}p_{\epsilon_{k},i}\right)}u_{\epsilon_{k}}\right)_{1}dt,
\end{eqnarray*}
as for $i\in\{1,...,d\}$, $\sum_{j=1}^{d}q_{ij}=0$ and $\sum_{j=1}^{d}A_{\epsilon,ij}=\frac{1}{2}$.
As $p_{\epsilon_{k}}\rightarrow p\text{ {in} }L^{2}([0,T];H)\text{ weakly}$,
the terms on the r.h.s of the first line converge to

\[
-\int_{0}^{T}\psi'(t)(v,u)_{1}dt-\int_{0}^{T}\psi(t)r(\partial_{x}v,u)_{1}dt+\int_{0}^{T}\psi(t)\frac{1}{2}\left(\partial_{x}v,\tilde{\sigma}_{Dup}^{2}\partial_{x}u\right)_{1}dt
\]
It is sufficient to show that the term on the r.h.s of the second
line converges to 
\[
\int_{0}^{T}\psi(t)\left(\partial_{x}v,\frac{1}{2}\tilde{\sigma}_{Dup}(\tilde{\sigma}_{Dup}+2\left(\partial_{x}\tilde{\sigma}_{Dup}\right))u\right)_{1}dt.
\]
For a.e. $(t,x)\in[0,T]\times\mathbb{R}$, $u_{\epsilon_{k}}(t,x)\rightarrow u(t,x)$.
Let us remark that by Proposition \ref{lem:VFinnonnegative}, for
$k\geq1$, $p_{\epsilon_{k}}\geq0$, so $p$ and $u$ are nonnegative.
If $u(t,x)>0$, then $\frac{\sum_{i=1}^{d}\lambda_{i}p_{\epsilon_{k},i}}{\epsilon_{k}\vee\left(\sum_{i=1}^{d}\lambda_{i}p_{\epsilon_{k},i}\right)}(t,x)\rightarrow1$.
If $u(t,x)=0$, $\frac{\sum_{i=1}^{d}\lambda_{i}p_{\epsilon_{k},i}}{\epsilon_{k}\vee\left(\sum_{i=1}^{d}\lambda_{i}p_{\epsilon_{k},i}\right)}u_{\epsilon_{k}}(t,x)\rightarrow0$
as $\left|\frac{\sum_{i=1}^{d}\lambda_{i}p_{\epsilon_{k},i}}{\epsilon_{k}\vee\left(\sum_{i=1}^{d}\lambda_{i}p_{\epsilon_{k},i}\right)}\right|\leq1$.
Let us use the following decomposition : $u_{\epsilon_{k}}=1_{\{u>0\}}u_{\epsilon_{k}}+1_{\{u=0\}}u_{\epsilon_{k}}$.
We first study the limit, as $k\rightarrow\infty$, of

\[
I_{1}(k):=\int_{0}^{T}\psi(t)\left(\frac{1}{2}\tilde{\sigma}_{Dup}(\tilde{\sigma}_{Dup}+2\left(\partial_{x}\tilde{\sigma}_{Dup}\right))\frac{\sum_{i=1}^{d}\lambda_{i}p_{\epsilon_{k},i}}{\epsilon_{k}\vee\left(\sum_{i=1}^{d}\lambda_{i}p_{\epsilon_{k},i}\right)}1_{\{u>0\}}\partial_{x}v,u_{\epsilon_{k}}\right)_{1}dt
\]
The function $u_{\epsilon_{k}}$ converges weakly to $u$ in $L^{2}([0,T],L^{2}(\mathbb{R}))$.
Moreover the function 
\[
\frac{1}{2}\tilde{\sigma}_{Dup}(\tilde{\sigma}_{Dup}+2\left(\partial_{x}\tilde{\sigma}_{Dup}\right))\frac{\sum_{i=1}^{d}\lambda_{i}p_{\epsilon_{k},i}}{\epsilon_{k}\vee\left(\sum_{i=1}^{d}\lambda_{i}p_{\epsilon_{k},i}\right)}1_{\{u>0\}}\partial_{x}v
\]
 converges strongly to $\frac{1}{2}\tilde{\sigma}_{Dup}(\tilde{\sigma}_{Dup}+2\left(\partial_{x}\tilde{\sigma}_{Dup}\right))1_{\{u>0\}}\partial_{x}v$,
in $L^{2}([0,T],L^{2}(\mathbb{R}))$. Indeed, the convergence is a.e.
and 
\[
\left|\frac{1}{2}\tilde{\sigma}_{Dup}(\tilde{\sigma}_{Dup}+2\left(\partial_{x}\tilde{\sigma}_{Dup}\right))\frac{\sum_{i=1}^{d}\lambda_{i}p_{\epsilon_{k},i}}{\epsilon_{k}\vee\left(\sum_{i=1}^{d}\lambda_{i}p_{\epsilon_{k},i}\right)}1_{\{u>0\}}\partial_{x}v\right|\leq\frac{1}{2}\overline{\sigma}(\overline{\sigma}+2||\partial_{x}\tilde{\sigma}||_{\infty})|\partial_{x}v|\in L^{2}([0,T],L^{2}(\mathbb{R}))
\]
 and we conclude by dominated convergence. Therefore, $I_{1}(k)$
converges, as $k\rightarrow\infty$, to 
\[
\int_{0}^{T}\psi(t)\left(\frac{1}{2}\tilde{\sigma}_{Dup}(\tilde{\sigma}_{Dup}+2\left(\partial_{x}\tilde{\sigma}_{Dup}\right))\partial_{x}v,1_{\{u>0\}}u\right)_{1}dt=\int_{0}^{T}\psi(t)\left(\frac{1}{2}\tilde{\sigma}_{Dup}(\tilde{\sigma}_{Dup}+2\left(\partial_{x}\tilde{\sigma}_{Dup}\right))\partial_{x}v,u\right)_{1}dt,
\]
because $u\geq0$. We then study the term 
\[
I_{2}(k):=\int_{0}^{T}\psi(t)\left(\frac{1}{2}\tilde{\sigma}_{Dup}(\tilde{\sigma}_{Dup}+2\left(\partial_{x}\tilde{\sigma}_{Dup}\right))\frac{\sum_{i=1}^{d}\lambda_{i}p_{\epsilon_{k},i}}{\epsilon_{k}\vee\left(\sum_{i=1}^{d}\lambda_{i}p_{\epsilon_{k},i}\right)}1_{\{u=0\}}\partial_{x}v,u_{\epsilon_{k}}\right)_{1}dt.
\]
We notice that $I_{2}(k)\leq\frac{1}{2}\overline{\sigma}(\overline{\sigma}+2||\partial_{x}\tilde{\sigma}||_{\infty})\int_{0}^{T}|\psi(t)|\left(|\partial_{x}v|1_{\{u=0\}},u_{\epsilon_{k}}\right)_{1}$,
and the r.h.s. converges to $0$ as $k\rightarrow\infty$, as $u_{\epsilon_{k}}\rightarrow u$
weakly in $L^{2}([0,T],L)$. We prove that the initial condition is
satisfied with the arguments at the end of the proof of Proposition
\ref{prop:ExistencetoVeps}. So we conclude that $u$ is, by Proposition
\ref{prop:uniquenessandaronson}, the unique solution to $LV(\mu_{X_{0}})$
and by Lemma \ref{lem:LVpositivecontinuous}, $u>0$ a.e. on $(0,T)\times\mathbb{R}$.
Finally, as $p$ is nonnegative, $p$ takes values in $\mathcal{D}$
a.e. on $[0,T]\times\mathbb{R}$. \end{proof}
\begin{prop}
\label{prop:VFinL2existence}Under Assumption (B), (H) and Condition
(C), there exists a solution $p\in C([0,T],L)$ to $V_{Fin,L^{2}}(\mu)$
such that $\sum_{i=1}^{d}p_{i}$ is the unique solution to $LV(\mu_{X_{0}})$.\end{prop}
\begin{proof}
It is easy to check that the family $\left(p_{\epsilon}\right)_{\epsilon>0}$
satisfies uniform in $\epsilon$ energy estimates. Using (\ref{eq:GpFin}),
we also obtain that the family $\left(\frac{dp_{\epsilon}}{dt}\right)_{\epsilon>0}$
is bounded in $L^{2}\left([0,T],H'\right)$. Similarly to the proof
of Proposition \ref{prop:existencetoVL2}, there exists a subsequence
$\left(p_{\epsilon_{k}}\right)_{k\geq1}$ converging to a function
$p\geq0$, with the convergences as $\epsilon_{k}\rightarrow0$. 
\begin{eqnarray*}
p_{\epsilon_{k}} & \rightarrow & p\text{ {in} }L^{2}([0,T];H)\text{ weakly},\\
p_{\epsilon_{k}} & \rightarrow & p\text{ {in} }L^{\infty}([0,T];L)\text{ weakly-*},\\
p_{\epsilon_{k}} & \rightarrow & p\text{ a.e. on \ensuremath{[0,T]\times\mathbb{R}}},
\end{eqnarray*}
We check that for $v\in H$ and $\psi\in C^{1}([0,T],\mathbb{R})$
with $\psi(T)=0$, 
\[
\int_{0}^{T}\left(\psi(t)\partial_{x}v,\frac{1}{2}R_{\epsilon_{k}}(p_{\epsilon_{k}})\tilde{\sigma}_{Dup}\left(\tilde{\sigma}_{Dup}+\partial_{x}\tilde{\sigma}_{Dup}\right)\Lambda p_{\epsilon_{k}}\right)_{d}dt\rightarrow\int_{0}^{T}\left(\psi(t)\partial_{x}v,\frac{1}{2}R(p)\tilde{\sigma}_{Dup}\left(\tilde{\sigma}_{Dup}+\partial_{x}\tilde{\sigma}_{Dup}\right)\Lambda p\right)_{d}dt.
\]
The convergence is justified as $p$ takes values in $\mathcal{D}$
a.e. by Lemma \ref{lem:5.8}, $R_{\epsilon_{k}}(p_{\epsilon_{k}})\underset{k\rightarrow\infty}{\rightarrow}R(p)$
a.e. on $[0,T]\times\mathbb{R}$ (with the same argument that justifies
$A_{\epsilon}(p_{\epsilon})\rightarrow A(p)$ a.e. in (\ref{eq:AepsPepsCVAE})),
the functions $R_{\epsilon}$ and $R$ have uniform bounds by Lemma
\ref{lem:Rbounded}, and $p_{\epsilon_{k}}\rightarrow p$ weakly in
$L^{2}([0,T],H)$. Similarly we also obtain the convergence:
\[
\int_{0}^{T}(\psi(t)\partial_{x}v,A_{\epsilon_{k}}(p_{\epsilon_{k}})\tilde{\sigma}_{Dup}^{2}\partial_{x}p_{\epsilon_{k}})_{d}dt\rightarrow\int_{0}^{T}(\psi(t)\partial_{x}v,A(p)\tilde{\sigma}_{Dup}^{2}\partial_{x}p)_{d}dt.
\]
As $p_{\epsilon_{k}}\rightarrow p$ weakly in $L^{2}([0,T],L)$, we
finally obtain the equality:
\begin{eqnarray*}
\int_{0}^{T}\psi(t)(Qv+r\partial_{x}v,p(t))_{d}dt & = & -\int_{0}^{T}(\psi'(t)v,p)_{d}dt-(v,p_{0})_{d}\psi(0)\\
 & \ + & \int_{0}^{T}(\psi(t)\partial_{x}v,A(p)\tilde{\sigma}_{Dup}^{2}\partial_{x}p)_{d}dt\\
 & \ + & \int_{0}^{T}\left(\psi(t)\partial_{x}v,\frac{1}{2}R(p)\tilde{\sigma}_{Dup}\left(\tilde{\sigma}_{Dup}+\partial_{x}\tilde{\sigma}_{Dup}\right)\Lambda p\right)_{d}dt,
\end{eqnarray*}
so that (\ref{eq:fvarFinL2}) is verified and we conclude the proof
with the same arguments as in the end of the proof of Proposition
\ref{prop:ExistencetoVeps} to obtain existence to $V_{Fin,L^{2}}(\mu)$.
By Lemma \ref{lem:5.8}, $\sum_{i=1}^{d}p_{i}$ solves $LV(\mu_{X_{0}})$,
which has a unique solution by Proposition \ref{prop:uniquenessandaronson}.
Finally, if $p$ is a solution to $V_{Fin,L^{2}}(\mu)$, then $p$
is a.e. equal to a function that belongs to $C([0,T],L)$ and this
concludes the proof.
\end{proof}

\subsubsection{Case when $\mu$ is a general probability measure on $\mathbb{R}\times\mathcal{Y}$}

We now prove Theorem \ref{thm:VFin}. When $\mu$ is a general probability
measure on $\mathbb{R}\times\mathcal{Y}$, for $\sigma>0$, we mollify
each $\mu_{i},1\leq i\leq d$ into measures $\mu_{\sigma,i}:=\mu_{i}*h_{\sigma}$
with square integrable densities, also denoted by $\mu_{\sigma,i}$,
and we define $\mu_{\sigma}$ as the measure under which the conditional
law of $X_{0}$ given $\{Y_{0}=i\}$ is given by $\mu_{\sigma,i}$
and $\mathbb{P}\left(Y_{0}=i\right)=\alpha_{i}$. Let $\left(\sigma_{k}\right)_{k\geq0}$
be a decreasing sequence converging to $0$ as $k\rightarrow\infty$.
Let us denote by $p_{\sigma_{k}}$ a solution to $V_{Fin,L^{2}}\left(\mu_{\sigma_{k}}\right)$,
which exists by Proposition \ref{prop:VFinL2existence}. Let us remark
that for $k\geq0$, $\sum_{i=1}^{d}p_{\sigma_{k},i}$ satisfies $LV\left(\sum_{i=1}^{d}\alpha_{i}\mu_{\sigma_{k},i}\right)$.
We compute energy estimates, using Proposition \ref{prop:uniquenessandaronson}:
\begin{eqnarray}
|p_{\sigma_{k}}(t)|_{d}^{2} & \leq & \left|\left|\sum_{i=1}^{d}p_{\sigma_{k},i}(t)\right|\right|_{L^{2}}^{2}\leq\frac{\zeta}{\sqrt{t}},\nonumber \\
\int_{0}^{T}|p_{\sigma_{k}}(t)|_{d}^{2}dt & \leq & 2\zeta\sqrt{T},\nonumber \\
\int_{t}^{T}|\partial_{x}p_{\sigma_{k}}(s)|_{d}^{2}ds & \leq & \frac{l_{max}(\Pi)}{\kappa\underbar{\ensuremath{\sigma}}}\frac{\zeta}{\sqrt{t}}e^{\frac{2}{l_{min}(\Pi)}\left(b+\frac{C^{2}}{2\kappa\underline{\sigma}^{2}}\right)T},\label{eq:equalityChacon}
\end{eqnarray}
where we used (\ref{eq:energyest2FinM}) for the last inequality.
Repeating the arguments of the proof in Subsection \ref{sub:FBMGeneralCase},
we obtain a subsequence again called $p_{\sigma_{k}}$, such that:

\begin{eqnarray*}
p_{\sigma_{k}} & \rightarrow & p\text{ {in} }L^{2}((0,T];L)\text{ weakly},\\
p_{\sigma_{k}} & \rightarrow & p\text{ {in} }L_{loc}^{2}((0,T];H)\text{ weakly},\\
p_{\sigma_{k}} & \rightarrow & p\text{ {in} }L_{loc}^{\infty}((0,T];L)\text{ weakly-*}\\
p_{\sigma_{k}} & \rightarrow & p\text{ a.e. on \ensuremath{(0,T]\times\mathbb{R}}}.
\end{eqnarray*}
We show that $p$ takes values in $\mathcal{D}$ a.e. on $(0,T]\times\mathbb{R}$.
To do so, we show that $u:=\sum_{i=1}^{d}p_{i}$ is a solution to
$(LV)$ with initial condition $\mu_{X_{0}}$. For $\psi\in C^{1}([0,T],\mathbb{R})$,
and $v\in H^{1}(\mathbb{R})$, as $u_{\sigma_{k}}:=\sum_{i=1}^{d}p_{\sigma_{k},i}\rightarrow u$
in $L^{2}((0,T],L^{2}(\mathbb{R}))$ weakly,

\begin{eqnarray*}
-\int_{0}^{T}\psi'(t)(v,u_{\sigma_{k}})_{1}dt & \rightarrow & -\int_{0}^{T}\psi'(t)(v,u)_{1}dt,\\
\int_{0}^{T}\psi(t)\left(\partial_{x}v,\left(-r+\frac{1}{2}\tilde{\sigma}_{Dup}^{2}+\tilde{\sigma}_{Dup}\left(\partial_{x}\tilde{\sigma}_{Dup}\right)\right)u_{\sigma_{k}}\right)_{1}dt & \rightarrow & \int_{0}^{T}\psi(t)\left(\partial_{x}v,\left(-r+\frac{1}{2}\tilde{\sigma}_{Dup}^{2}+\tilde{\sigma}_{Dup}\left(\partial_{x}\tilde{\sigma}_{Dup}\right)\right)u\right)_{1}dt.
\end{eqnarray*}
Let us define $\varUpsilon:=\frac{1}{2^{1/4}-1}\left(\frac{l_{max}(\Pi)\zeta}{\kappa\underbar{\ensuremath{\sigma}}}e^{\frac{2}{l_{min}(\Pi)}\left(b+\frac{C^{2}}{2\kappa\underline{\sigma}^{2}}\right)T}\right)^{1/2}$.
By Lemma \ref{lem:L1dxp}, we have that for $\forall t\in(0,T)$ and
$k\ge1$, $\int_{0}^{t}|\partial_{x}p_{\sigma_{k}}(s)|_{d}ds\leq\varUpsilon t^{1/4}$
and $\int_{0}^{t}|\partial_{x}p(s)|_{d}ds\leq\varUpsilon t^{1/4}$.
Then, $\forall t\in(0,T)$, 
\[
\int_{0}^{t}\left|\partial_{x}u_{\sigma_{k}}\right|_{d}dt=\int_{0}^{t}\left|\sum_{i=1}^{d}\partial_{x}p_{\sigma_{k},i}(s)\right|_{1}ds\leq\sqrt{d}\int_{0}^{t}|\partial_{x}p_{\sigma_{k}}(s)|_{d}ds\leq\sqrt{d}\varUpsilon t^{1/4},
\]
and $\int_{0}^{t}\left|\partial_{x}u\right|_{d}dt\leq\sqrt{d}\varUpsilon t^{1/4}.$
The same arguments as those leading to (\ref{eq:cvAsigmaA}) enable
to prove that: 
\[
\int_{0}^{T}\psi(t)\left(\partial_{x}v,\tilde{\sigma}_{Dup}^{2}\partial_{x}u_{\sigma_{k}}\right)dt\rightarrow\int_{0}^{T}\psi(t)\left(\partial_{x}v,\tilde{\sigma}_{Dup}^{2}\partial_{x}u\right)dt.
\]
The initial condition is treated as before. This is sufficient to
prove that $u$ solves $LV(\mu_{X_{0}})$ and assert that $p$ takes
values in $\mathcal{D}$ by Lemma \ref{lem:LVpositivecontinuous}.
For $\psi\in C^{1}([0,T],\mathbb{R})$, such that $\psi(T)=0$ and
$v\in H$, as $p_{\sigma_{k}}\rightarrow p$ weakly in $L^{2}((0,T],L),$

\begin{eqnarray*}
\int_{0}^{T}(\psi'(s)v,p_{\sigma_{k}}(s))_{d}ds & \underset{k\rightarrow\infty}{\rightarrow} & \int_{0}^{T}(\psi'(s)v,p(s))_{d}ds\\
\int_{0}^{T}\psi(t)(Qv,p_{\sigma_{k}})_{d}dt & \underset{k\rightarrow\infty}{\rightarrow} & \int_{0}^{T}\psi(t)(Qv,p)_{d}dt
\end{eqnarray*}
We check that:

\begin{eqnarray*}
\int_{0}^{T}\psi(t)\left(\partial_{x}v,R(p_{\sigma_{k}}(t))\tilde{\sigma}_{Dup}\left(\tilde{\sigma}_{Dup}+\partial_{x}\tilde{\sigma}_{Dup}\right)\Lambda p_{\sigma_{k}}(t)\right)_{d}dt & \underset{k\rightarrow\infty}{\rightarrow} & \frac{1}{2}\int_{0}^{T}\psi(t)\left(\partial_{x}v,R(p(t))\tilde{\sigma}_{Dup}\left(\tilde{\sigma}_{Dup}+\partial_{x}\tilde{\sigma}_{Dup}\right)\Lambda p(t)\right)_{d},
\end{eqnarray*}
which is the case as the sequence $\left(p_{\sigma_{k}}\right)_{k\geq0}$
converges to $p$ weakly in $L^{2}((0,T],L)$, $R$ is bounded and
continuous on $\mathcal{D}$, and $R(p_{\sigma_{k}})\rightarrow R(p)$
a.e. on $(0,T]\times\mathbb{R}$. We also check that:
\[
\int_{0}^{T}(\psi(s)\partial_{x}v,\tilde{\sigma}_{Dup}^{2}A(p_{\sigma_{k}})\partial_{x}p_{\sigma_{k}})_{d}ds\underset{k\rightarrow\infty}{\rightarrow}\int_{0}^{T}(\psi(s)\partial_{x}v,\tilde{\sigma}_{Dup}^{2}A(p)\partial_{x}p)_{d}ds,
\]
using the same argument as for (\ref{eq:cvAsigmaA}). Arguments similar
to (\ref{eq:ToyVFmeslimit}) enable to check the initial condition
and thus assert existence to $V_{Fin}(\mu).$ As any solution to $V_{Fin}(\mu)$
has a representative in $C((0,T],L)$, this concludes the proof of
Theorem \ref{thm:VFin}.

\subsection{Existence of a weak solution to the SDE (\ref{eq:sdefinance1})}

To prove Theorem \ref{thm:weaksolutionFin}, we use Theorem \ref{thm:FigalliGeneralization}
below, which is a generalization of \cite[Theorem 2.6]{Figalli} to
make a link between the existence to a Fokker-Planck system and the
existence to the corresponding martingale problem. Theorem \ref{thm:FigalliGeneralization}
is proved in Appendix \ref{sec:FigGen}. There exist several generalizations
of \cite[Theorem 2.6]{Figalli}, among whom we can mention \cite{FournierHauray},
where the coefficients of the generator are no longer bounded but
have linear growth, and \cite{zhang}, where the author deals with
a partial integro differential equation with a Lévy generator.

\subsubsection{Generalization of \cite[Theorem 2.6]{Figalli}}

Given functions $\left(b_{i}\right)_{1\leq i\leq d},\left(a_{i}\right)_{1\leq i\leq d},\left(q_{ij}\right)_{1\leq i,j\leq d}$
defined on $[0,T]\times\mathbb{R}$ and a finite measure $\mu_{0}$
on $\mathbb{R}\times\mathcal{Y}$, we study the following PDS, where
for $1\leq i\leq d$: 
\begin{eqnarray}
\partial_{t}\mu_{i}+\partial_{x}(b_{i}\mu_{i})-\frac{1}{2}\partial_{xx}^{2}(a_{i}\mu_{i})-\sum_{j=1}^{d}q_{ji}\mu_{j} & = & 0\text{ in \ensuremath{(0,T)\times\mathbb{R}}}\label{eq:PDSJump1}\\
\mu_{i}(0) & = & \mu_{0}(\cdot,\{i\}).\label{eq:PDSJump2}
\end{eqnarray}

\begin{defn}
A family of vectors of Borel measures $\left(\mu_{1}(t,\cdot),...,\mu_{d}(t,\cdot)\right)_{t\in(0,T]}$
is a solution to the PDS (\ref{eq:PDSJump1})-(\ref{eq:PDSJump2})
if for any function $\phi$ defined on $\mathbb{\mathcal{S}}$ such
that $\forall i\in\{1,...,d\},\phi(\cdot,i)\in C_{c}^{\infty}(\mathbb{R})$,
\begin{eqnarray}
\frac{d}{dt}\int_{\mathbb{R}}\sum_{i=1}^{d}\phi(x,i)\mu_{i}(t,dx) & = & \sum_{i=1}^{d}\int_{\mathbb{R}}\left(b_{i}(t,x)\partial_{x}\phi(x,i)+\frac{1}{2}a_{i}(t,x)\partial_{xx}^{2}\phi(x,i)\right)\mu_{i}(t,dx)\nonumber \\
 & \ + & \sum_{i=1}^{d}\sum_{j=1}^{d}\left(\int_{\mathbb{R}}q_{ji}(t,x)\phi(x,i)\mu_{j}(t,dx)\right),\label{eq:distributionsolution}
\end{eqnarray}
in the distributional sense on $(0,T)$, and for $1\leq i\leq d$
and $\psi\in C_{b}^{2}(\mathbb{R})$, the function $t\rightarrow\int_{\mathbb{R}}\psi(x)\mu_{i}(t,dx)$
is continuous on $(0,T]$ and converges to $\int_{\mathbb{R}}\psi(x)\mu_{0}(dx,\{i\})$
as $t\rightarrow0$. 
\end{defn}
We suppose that all the coefficients $b_{i},a_{i},q_{ij},1\leq i,j\leq d,$
are uniformly bounded on $[0,T]\times\mathbb{R}$, that the coefficients
$\left(a_{i}\right)_{1\leq i\leq d}$ are non negative, that the coefficients
$\left(q_{ij}\right)_{1\leq i,j\leq d}$ are non negative functions
for $i\neq j$, and $q_{ii}=-\sum_{j\neq i}q_{ij}$. We introduce
the SDE 
\begin{equation}
dX_{t}=b_{Y_{t}}(t,X_{t})dt+\sqrt{a_{Y_{t}}(t,X_{t})}dW_{t},\label{eq:SDEJumpFig}
\end{equation}
where $Y_{t}$ is a stochastic process with values in $\mathcal{Y}$,
and that satisfies, for $j\neq Y_{t}$, 
\[
\text{\ensuremath{\mathbb{P}}}\left(Y_{t+dt}=j|\left(X_{s},Y_{s}\right)_{0\leq s\leq t}\right)=q_{Y_{t}j}(t,X_{t})dt.
\]
We define $E=\{(X,Y),X\in C([0,T],\mathbb{R}),Y\text{ càdlàg with values in \ensuremath{\mathcal{Y}}}\}$
endowed with the Skorokhod topology. For a probability measure $m$
on $\mathbb{R}\times\mathcal{Y}$, a probability measure $\nu$ on
$E$ is a martingale solution to the SDE (\ref{eq:SDEJumpFig}) with
initial condition $m$ if under the probability $\nu$, the canonical
process $(X,Y)$ on $E$ satisfies $(X_{0},Y_{0})\sim m$ and for
any function $\phi$ defined on $\mathbb{R}\times\mathcal{Y}$ s.t.
$\ensuremath{\forall i\in\{1,...,d\},\phi(\cdot,i)\in C_{b}^{2}(\mathbb{R})}$,
the process 

\[
\phi(X_{t},Y_{t})-\phi(X_{0},Y_{0})-\int_{0}^{t}\left(\frac{1}{2}a_{Y_{s}}(s,X_{s})\partial_{xx}^{2}\phi(X_{s},Y_{s})+b_{Y_{s}}(s,X_{s})\partial_{x}\phi(X_{s},Y_{s})+\sum_{l=1}^{d}q_{Y_{s}l}(s,X_{s})\phi(X_{s},l)\right)ds
\]
is a $\nu$-martingale. Now we can state the generalization of \cite[Theorem 2.6]{Figalli},
that will be used in the following section.
\begin{thm}
\label{thm:FigalliGeneralization}Let $(\mu_{1}(t,\cdot),...,\mu_{d}(t,\cdot))_{t\in(0,T]}$
be a solution to the PDS (\ref{eq:PDSJump1})-(\ref{eq:PDSJump2})
where the initial condition $\mu_{0}$ is a probability measure on
$\mathbb{R}\times\mathcal{Y}$. We moreover suppose that there exists
$B>0$, s.t. for $1\leq i\leq d$, and $t\in(0,T]$, $\mu_{i}(t,\mathbb{R})\leq B$.
Then the SDE (\ref{eq:SDEJumpFig}) with initial distribution $\mu_{0}$
has a martingale solution $\nu$ which satisfies the following representation
formula: for $1\leq i\leq d$ and $\psi\in C_{c}^{\infty}(\mathbb{R})$,
\[
\int_{\mathbb{R}}\psi(x)\mu_{i}(t,dx)=\int_{E}\psi(X_{t})1_{\{Y_{t}=i\}}d\nu(X,Y).
\]

\end{thm}

\subsubsection{Proof of Theorem \ref{thm:weaksolutionFin}}

Let $p$ be a solution to $V_{Fin}(\mu)$ with the properties stated
in Theorem \ref{thm:VFin}. To show that $(p_{1}(t,x)dx,...,p_{d}(t,x)dx)_{t\in(0,T]}$
satisfies the variational formulation in the sense of distributions
(\ref{eq:distributionsolution}), we check that for $1\leq i\leq d$,
$\tilde{\sigma}_{Dup}^{2}\frac{\sum_{k=1}^{d}p_{k}}{\sum_{k=1}^{d}\lambda_{k}p_{k}}\lambda_{i}p_{i}\in H^{1}(\mathbb{R}),$
and $\partial_{x}\left(\tilde{\sigma}_{Dup}^{2}\frac{\sum_{k=1}^{d}p_{k}}{\sum_{k=1}^{d}\lambda_{k}p_{k}}\lambda_{i}p_{i}\right)=\tilde{\sigma}_{Dup}^{2}\left(A(p)p\right)_{i}+\frac{\sum_{k=1}^{d}p_{k}}{\sum_{k=1}^{d}\lambda_{k}p_{k}}\lambda_{i}p_{i}\partial_{x}\tilde{\sigma}_{Dup}^{2}$.
As $\tilde{\sigma}_{Dup}\in L^{\infty}([0,T],W^{1,\infty}(\mathbb{R}))$,
it is sufficient to check that $\frac{\sum_{k=1}^{d}p_{k}}{\sum_{k=1}^{d}\lambda_{k}p_{k}}\lambda_{i}p_{i}\in H^{1}(\mathbb{R})$
and that $\partial_{x}\left(\frac{\sum_{k=1}^{d}p_{k}}{\sum_{k=1}^{d}\lambda_{k}p_{k}}\lambda_{i}p_{i}\right)=\left(A(p)p\right)_{i}$,
and the proof is exactly the same as the one of Lemma \ref{lem:VFtoFigalli},
thanks to Lemma \ref{lem:LVpositivecontinuous} which ensures the
positivity of $\sum_{i=1}^{d}p_{i}$ a.e. on $(0,T]\times\mathbb{R}$.
Then by Theorem \ref{thm:FigalliGeneralization}, there exists a measure
$\nu$ under which $(X_{0},Y_{0})\sim\mu$ and for any function $\phi$
defined on $\mathcal{S}$ s.t. $\ensuremath{\forall i\in\{1,...,d\},\phi(\cdot,i)\in C_{b}^{2}(\mathbb{R})}$,
the function 

\begin{eqnarray*}
 &  & \phi(X_{t},Y_{t})-\phi(X_{0},Y_{0})-\int_{0}^{t}\frac{1}{2}\tilde{\sigma}_{Dup}^{2}f^{2}(Y_{s})\frac{\sum_{i=1}^{d}p_{k}}{\sum_{i=1}^{d}\lambda_{k}p_{k}}(s,X_{s})\partial_{xx}^{2}\phi(X_{s},Y_{s})ds\\
 &  & -\int_{0}^{t}\left(r-\frac{1}{2}\tilde{\sigma}_{Dup}^{2}f^{2}(Y_{s})\frac{\sum_{i=1}^{d}p_{k}}{\sum_{i=1}^{d}\lambda_{k}p_{k}}(s,X_{s})\right)\partial_{x}\phi(X_{s},Y_{s})-\sum_{l=1}^{d}q_{Y_{s}l}(X_{s})\phi(X_{s},l)ds,
\end{eqnarray*}
is a $\nu$-martingale. Moreover, for $h:\mathcal{Y}\rightarrow\mathbb{R}$
and $g:\mathbb{R}\rightarrow\mathbb{R}$ continuous and bounded functions,
we have that: 
\[
\mathbb{E}\left(h^{2}(Y_{s})g(X_{s})\right)=\sum_{i=1}^{d}\mathbb{E}\left(h^{2}(Y_{s})1_{\{Y_{s}=y_{i}\}}g(X_{s})\right)=\int_{\mathbb{R}}g(x)\sum_{i=1}^{d}h^{2}(y_{i})p_{i}(t,x)dx.
\]
Taking $h\equiv1$ and $h\equiv f$, we check that the time marginals
of $X$ are given by $\sum_{i=1}^{d}p_{i}$ and that $\mathbb{E}\left(f^{2}(Y_{s})|X_{s}\right)=\frac{\sum_{k=1}^{d}\lambda_{k}p_{k}}{\sum_{k=1}^{d}p_{k}}(s,X_{s})$.
Therefore $\nu$ is a solution to the martingale problem associated
to the SDE (\ref{eq:sdefinance1}), and we obtain existence of a weak
solution to the SDE (\ref{eq:sdefinance1}) by \cite[Theorem 2.3]{Kurtz2011},
and which has the same time marginals as the solution of the SDE (\ref{eq:SDEFinDupireLogPrice2}).

\subsection{A more general fake Brownian motion}

We consider the SDE:
\begin{equation}
dX_{t}=\frac{f(Y_{t})}{\sqrt{\mathbb{E}[f^{2}(Y_{t})|X_{t}]}}dW_{t},\label{eq:SDEJump}
\end{equation}
where the process $\left(Y_{t}\right)_{t\geq0}$ takes values in $\mathcal{Y}$
and $\mathbb{P}\left(Y_{t+dt}=j|\left(X_{s},Y_{s}\right),0\leq s\leq t\right)=q_{Y_{t}j}(X_{t})dt$,
for $1\leq j\leq d$ and $j\neq Y_{t}$, with the functions $\left(q_{ij}\right)_{1\leq i\neq j\leq d}$
non negative and bounded, and $q_{ii}=-\sum_{j\neq i}q_{ij}$ for
$1\leq i\leq d$. The vector $\left(X_{0},Y_{0}\right)$ has the probability
distribution $\mu$ on $\mathbb{R}\times\mathcal{S}$ and is independent
from $\left(W_{t}\right)_{t\geq0}$. The associated Fokker-Planck
PDS writes:
\begin{eqnarray*}
\forall i\in\{1,...,d\},\ \partial_{t}p_{i} & = & \frac{1}{2}\partial_{xx}^{2}\left(\frac{\sum_{i=1}^{d}p_{k}}{\sum_{i=1}^{d}\lambda_{k}p_{k}}\lambda_{i}p_{i}\right)+\sum_{j=1}^{d}q_{ji}p_{j}\\
p_{i}(0,\cdot) & = & \alpha_{i}\mu_{i},
\end{eqnarray*}
with $\left(\alpha_{i}\right)_{1\leq i\leq d}$ and $\left(\mu_{i}\right)_{1\leq i\leq d}$
defined as in Section \ref{sec:CalibrationofRSLVmain}. We introduce
an associated variational formulation $V_{Jump}(\mu)$:

\begin{eqnarray*}
\text{Find} &  & p=(p_{1},...,p_{d})\ \text{satisfying:}
\end{eqnarray*}
\[
p\in L_{loc}^{2}((0,T];H)\cap L_{loc}^{\infty}((0,T];L),
\]
\[
\text{\ensuremath{p} takes values in \ensuremath{\mathcal{D}}, a.e. on \ensuremath{(0,T)\times\mathbb{R}},}
\]
\begin{eqnarray*}
\forall v\in H,\ \frac{d}{dt}(v,p)_{d}+(\partial_{x}v,A(p)\partial_{x}p)_{d} & = & (Qv,p)_{d},
\end{eqnarray*}
\[
\text{in the sense of distributions on \ensuremath{(0,T)}, and}
\]

\[
p(t,\cdot)\underset{t\rightarrow0^{+}}{\overset{\text{weakly-*}}{\rightarrow}}p_{0}:=(\alpha_{1}\mu_{1},...,\alpha_{d}\mu_{d})
\]
With the same arguments used to prove Theorems \ref{thm:VFin} and
\ref{thm:weaksolutionFin}, we can prove the following results. The
main difference is that $\sum_{i=1}^{d}p$ is solution to the heat
equation and not the Dupire PDE, but existence and uniqueness, positivity
and Aronson-like estimates of the solution hold in both cases, so
it does not affect the proof.
\begin{thm}
Under Condition (C), $V_{Jump}(\mu)$ has a solution $p\in C((0,T],L)$.
Moreover, $\sum_{i=1}^{d}p_{i}(t,x)=\mu_{X_{0}}*h_{t}(x)$ a.e. on
$(0,T]\times\mathbb{R}$.
\end{thm}

\begin{thm}
\label{thm:5.13}Under Condition (C), SDE (\ref{eq:SDEJump}) has
a weak solution, and its time marginals are those of $\left(Z+W_{t}\right)_{t\geq0}$,
where $Z$ has the law $\mu_{X_{0}}$ and is independent from $\left(W_{t}\right)_{t\geq0}$.
\end{thm}
Moreover, the solutions to SDE (\ref{eq:SDEJump}) are also continuous
fake Brownian motions provided that $f(Y_{0})$ can take at least
two distinct values with positive probability. Let us define $\tilde{\mathcal{Y}}:=\{i\in\mathcal{Y},\ \mathbb{P}(Y_{0}=i)>0\}$. 
\begin{prop}
\label{thm:JumpfakeBM}Under Condition (C), if $f$ is non constant
on $\tilde{\mathcal{Y}}$, the solutions to SDE (\ref{eq:SDEJump})
with initial condition $\mu=\delta_{0}$ are continuous fake Brownian
motions.\end{prop}
\begin{proof}
Let $\left(X_{t}\right)_{t\geq0}$ be a solution to SDE (\ref{eq:SDEJump}),
with initial condition $X_{0}=0$. The process $X$ is a continuous
martingale, and by \cite[Theorem 4.6]{Gyongy}, for $t\geq0$, $X_{t}\sim\mathcal{N}(0,t)$.
We consider its quadratic variation $d\langle X\rangle_{t}=\frac{f^{2}\left(Y_{t}\right)}{\mathbb{E}\left[f^{2}\left(Y_{t}\right)|X_{t}\right]}dt$.
We reason by contraposition and first suppose that a.s., for a.e.
$t>0$, the equality $f^{2}\left(Y_{t}\right)=\mathbb{E}\left[f^{2}\left(Y_{t}\right)|X_{t}\right]$.
Then a.s., for $t>0$, $X_{t}=W_{t}$ and there exists a measurable
function $\psi_{t}$, such that $\psi_{t}(X_{t})=f^{2}\left(Y_{t}\right)$.
For $i\in\tilde{\mathcal{Y}}$ and $g:\mathbb{R}\rightarrow\mathbb{R}$
measurable and non negative, let us fix $t>0$ such that those properties
hold and consider the term: 
\begin{eqnarray*}
\mathbb{E}\left[g\left(X_{t}\right)1_{\{\forall s\in[0,t],Y_{s}=i\}}\right] & = & \mathbb{E}\left[g\left(X_{t}\right)\mathbb{E}\left[1_{\{\forall s\in[0,t],Y_{s}=i\}}|\left(X_{s}\right)_{0\leq s\leq t},Y_{0}\right]\right]\\
 & = & \mathbb{E}\left[g\left(X_{t}\right)\exp\left(\int_{0}^{t}q_{ii}\left(s,X_{s}\right)ds\right)1_{\{Y_{0}=i\}}\right]\\
 & \geq & \exp\left(-\overline{q}t\right)\mathbb{E}\left[1_{\{Y_{0}=i\}}g\left(W_{t}\right)\right]=\alpha_{i}\frac{\exp\left(-\overline{q}t\right)}{\sqrt{2\pi t}}\int_{\mathbb{R}^{2}}g\left(y\right)\exp\left(-\frac{y^{2}}{2t}\right)dy,
\end{eqnarray*}
by independence between $Y_{0}$ and $\left(W_{t}\right)_{t\geq0}$.
Therefore on the event $\{\forall s\in[0,t],Y_{s}=y_{i}\}$, the random
variable $X_{t}$ belongs with positive probability to any Borel set
with positive Lebesgue measure. By the equality $\psi_{t}(X_{t})=f^{2}\left(Y_{t}\right)$,
the previous observation implies that $f$ is constant on $\tilde{\mathcal{Y}}$.
To conclude the proof, if $f$ is non constant on $\tilde{\mathcal{Y}}$,
by contraposition, the function $t\rightarrow\langle X\rangle_{t}$
is not equal to the identity function, and $\left(X_{t}\right)_{t\ge0}$
is a fake Brownian motion.
\end{proof}
\appendix

\section{\label{sec:AppendixA}Local Lipschitz property}

\subsection{Proof of Lemma \ref{lem:FespLocalLipschitz}}
\begin{proof}
Let $m\geq1$. It is sufficient to show that the function $z\in\mathbb{R}^{m}\rightarrow K_{\epsilon}^{m}(z)$
is locally Lipschitz to have $z\in\mathbb{R}^{m}\rightarrow F_{\epsilon}^{m}(z)$
locally Lipschitz. Let us remark that for $a,b\in\{1,...,m\}$, $\left(K_{\epsilon}(z)-K_{\epsilon}(\tilde{z})\right)_{ab}$
only contains integrals of type: 
\[
\int_{\mathbb{R}}\Theta\left(M_{\epsilon,ij}\left(\left(\sum_{k=1}^{m}z_{k}w_{k}\right)^{+}\right)-M_{\epsilon,ij}\left(\left(\sum_{k=1}^{m}\tilde{z}_{k}w_{k}\right)^{+}\right)\right)dx,\ 1\leq i,j\leq d,
\]
where $\Theta\in L^{1}(\mathbb{R})$, as for $c,d\in\{1,...,d\}$,
$\partial_{x}w_{ac}\partial_{x}w_{bd}\in L^{1}(\mathbb{R})$. For
$\Theta\in L^{1}(\mathbb{R})$, $1\leq i,j\leq d$ and $z\in\mathbb{R}^{m}$,
let us define 
\[
g_{ij}(z):=\int_{\mathbb{R}}\Theta M_{\epsilon,ij}\left(\left(\sum_{k=1}^{m}z_{k}w_{k}\right)^{+}\right)dx.
\]
Let $C$ be a compact subset of $\mathbb{R}^{m}$. Let us show that
the function $z\rightarrow g_{ij}(z)$ is Lipschitz on $C$ and conclude
by linearity. For $i\neq j$, $\rho,\tilde{\rho}\in\mathbb{R}^{d}$,
\begin{eqnarray*}
|M_{\epsilon,ij}(\rho^{+})-M_{\epsilon,ij}(\tilde{\rho}^{+})| & = & \left|\lambda_{i}\rho_{i}^{+}\frac{\sum_{l\neq j}(\lambda_{l}-\lambda_{j})\rho_{l}^{+}}{\left(\epsilon\vee\left(\sum_{l}\lambda_{l}\rho_{l}^{+}\right)\right)^{2}}-\lambda_{i}\tilde{\rho}_{i}^{+}\frac{\sum_{l\neq j}(\lambda_{l}-\lambda_{j})\tilde{\rho}_{l}^{+}}{\left(\epsilon\vee\left(\sum_{l}\lambda_{l}\tilde{\rho}_{l}^{+}\right)\right)^{2}}\right|\\
 & \leq & |\Delta_{1}(\rho,\tilde{\rho})|+|\Delta_{2}(\rho,\tilde{\rho})|+|\Delta_{3}(\rho,\tilde{\rho})|
\end{eqnarray*}
with

\begin{eqnarray*}
|\Delta_{1}(\rho,\tilde{\rho})|:=\left|\lambda_{i}\left(\rho_{i}^{+}-\tilde{\rho_{i}}^{+}\right)\frac{\sum_{l\neq j}(\lambda_{l}-\lambda_{j})\rho_{l}^{+}}{\left(\epsilon\vee\left(\sum_{l}\lambda_{l}\rho_{l}^{+}\right)\right)^{2}}\right| & \leq & \left|\frac{\sum_{l\neq j}(\lambda_{l}-\lambda_{j})\rho_{l}^{+}}{\epsilon^{2}}\right|\lambda_{i}|\rho_{i}-\tilde{\rho_{i}}|\\
 & \leq & \frac{d-1}{\epsilon^{2}}|\lambda_{max}-\lambda_{min}|\lambda_{max}||\rho||_{\infty}||\rho-\tilde{\rho}||_{\text{\ensuremath{\infty}}}\\
|\Delta_{2}(\rho,\tilde{\rho})|:=\left|\lambda_{i}\tilde{\rho}_{i}^{+}\frac{\sum_{l\neq j}(\lambda_{l}-\lambda_{j})\left(\rho_{l}^{+}-\tilde{\rho_{l}}^{+}\right)}{\left(\epsilon\vee\left(\sum_{l}\lambda_{l}\rho_{l}^{+}\right)\right)^{2}}\right| & \leq & \frac{d-1}{\epsilon^{2}}|\lambda_{max}-\lambda_{min}|\lambda_{max}||\tilde{\rho}||_{\infty}||\rho-\tilde{\rho}||_{\text{\ensuremath{\infty}}}
\end{eqnarray*}

\begin{eqnarray*}
|\Delta_{3}(\rho,\tilde{\rho})| & := & \left|\lambda_{i}\tilde{\rho}_{i}^{+}\left(\sum_{l\neq j}(\lambda_{l}-\lambda_{j})\tilde{\rho_{l}}^{+}\right)\left(\frac{1}{\left(\epsilon\vee\left(\sum_{l}\lambda_{l}\rho_{l}^{+}\right)\right)^{2}}-\frac{1}{\left(\epsilon\vee\left(\sum_{l}\lambda_{l}\tilde{\rho}_{l}^{+}\right)\right)^{2}}\right)\right|\\
 & \leq & \frac{1}{\epsilon^{4}}\lambda_{i}\tilde{\rho}_{i}^{+}\left|\sum_{l\neq j}(\lambda_{l}-\lambda_{j})\left(\tilde{\rho_{l}}^{+}\right)\right|\left|2\epsilon+\sum_{l}\lambda_{l}\left(\rho_{l}^{+}+\tilde{\rho}_{l}^{+}\right)\right|\left|\sum_{l}\lambda_{l}\left(\rho_{l}-\tilde{\rho}_{l}\right)\right|\\
 & \leq & \frac{\lambda_{max}^{2}}{\epsilon^{4}}d(d-1)|\lambda_{max}-\lambda_{min}|\left(2\epsilon+d\lambda_{max}\left(||\rho||_{\infty}+||\tilde{\rho}||_{\infty}\right)\right)||\tilde{\rho}||_{\infty}^{2}||\rho-\tilde{\rho}||_{\infty}.
\end{eqnarray*}
where we used the fact that $\forall a,b\in\mathbb{R},|a^{+}-b^{+}|\leq|a-b|$
and that in the last inequality, $\forall\epsilon,a,b\geq0,|\left(\epsilon\vee a\right)^{2}-\left(\epsilon\vee b\right)^{2}|=|\epsilon\vee a+\epsilon\vee b||\epsilon\vee a-\epsilon\vee b|\leq|2\epsilon+a+b||a-b|$.
We now replace $\rho$ by $\sum_{k=1}^{m}z_{k}w_{k}$, and $\tilde{\rho}$
by $\sum_{k=1}^{m}\tilde{z}_{k}w_{k}$. As the sequence $\left(w_{k}\right)_{k\geq1}$
belongs to $H$, by \cite[Corollary VIII.8]{Brezis}, for $k\geq1$
and $1\leq i\leq d$, $w_{ki}\in L^{\infty}(\mathbb{R})$ and the
function $(x,z)\in\mathbb{R}\times C\rightarrow\sum_{k=1}^{m}z_{k}w_{k}(x)\in\mathbb{R}^{d}$
takes values for a.e. $x\in\mathbb{R}$ in a bounded subset of $\mathbb{R}^{d}$.
Then there exists an uniform bound $B<\infty$ s.t. $\forall z,\tilde{z}\in C,$
\[
\left|\left|\Delta_{1}\left(\sum_{k=1}^{m}z_{k}w_{k},\sum_{k=1}^{m}\tilde{z}_{k}w_{k}\right)\right|\right|_{\infty}+\left|\left|\Delta_{2}\left(\sum_{k=1}^{m}z_{k}w_{k},\sum_{k=1}^{m}\tilde{z}_{k}w_{k}\right)\right|\right|_{\infty}+\left|\left|\Delta_{3}\left(\sum_{k=1}^{m}z_{k}w_{k},\sum_{k=1}^{m}\tilde{z}_{k}w_{k}\right)\right|\right|_{\infty}\leq B||z-\tilde{z}||_{\infty},
\]
so after integration against $\Theta\in L^{1}(\mathbb{R})$, 
\[
|g_{ij}(z)-g_{ij}(\tilde{z})|\leq||\Theta||_{L^{1}}B||z-\tilde{z}||_{\infty},
\]
and $g_{ij}$ is Lipschitz on $C$. Exactly in the same way, we also
obtain that the functions $\left(g_{ii}\right)_{i\in\{1,...,d\}}$,
are Lipschitz on $C$, so $K_{\epsilon}^{m}$ is locally Lipschitz
and this concludes the proof.
\end{proof}

\subsection{Proof of Lemma \ref{lem:F1F2localLip}}
\begin{proof}
Let $m\geq0$. It is sufficient to prove that $K_{\epsilon,1}^{m}$
and $K_{\epsilon,2}^{m}$ are locally Lipschitz in $z$ uniformly
in $t$. In the proof of Lemma \ref{lem:FespLocalLipschitz}, we have
shown that the function 
\[
z\in\mathbb{R}^{m}\rightarrow\int_{\mathbb{R}}\Theta A_{\epsilon,ij}\left(\left(\sum_{k=1}^{m}z_{k}w_{k}\right)^{+}\right)dx,
\]
is locally Lipschitz, uniformly for $\Theta$ in a bounded subset
of $L^{1}(\mathbb{R})$. As the family $\left(\tilde{\sigma}_{Dup}^{2}(t,\cdot)\partial_{x}w_{ik}\partial_{x}w_{jl}\right)_{1\leq i,j\leq m,1\leq k,l\leq d,t\in[0,T]}$
belongs to a bounded subset of $L^{1}(\mathbb{R})$, we then have
that $F_{\epsilon,2}$ is locally Lipschitz in $z$ uniformly in $t$
. To show the local Lipschitz property of $F_{\epsilon,1}$, it is
sufficient to prove that for any function $\Theta\in L^{1}(\mathbb{R})$,
the function 
\[
z\in\mathbb{R}^{m}\rightarrow\int_{\mathbb{R}}\Theta R_{\epsilon}\left(\left(\sum_{k=1}^{m}z_{k}w_{k}\right)^{+}\right)dx,
\]
is locally Lipschitz in $z$ uniformly in $t$, as the functions $\left(w_{jk}\partial_{x}w_{il}\right)_{1\leq i,j\leq m,1\leq k,l\leq d}$
belong to $L^{1}(\mathbb{R})$, and $\tilde{\sigma}_{Dup}$ and $\partial_{x}\tilde{\sigma}_{Dup}$
are uniformly bounded. The result is obtained since the function $\rho\rightarrow R_{\epsilon}\left(\rho^{+}\right)$
is locally Lipschitz and the functions $\left(w_{ki}\right)_{1\leq k\leq m,1\leq i\leq d}$
belong to $L^{\infty}(\mathbb{R})$, as in the proof of Lemma \ref{lem:FespLocalLipschitz}.
\end{proof}

\section{\label{sub:ConditionC}About Condition $(C)$}

In Subsection \ref{sub:DiagCase}, we give a necessary and sufficient
condition for a diagonal matrix to satisfy Condition $(C)$. We also
give a numerical procedure to check if there exists a diagonal matrix
that satisfies $(C)$. Then, in Subsection \ref{sub:d=00003D3}, we
focus on the case $d=3$, and give a simple necessary and sufficient
condition for $(C)$ to be satisfied. When $d=3$, $(C)$ is satisfied
if and only if it is satisfied by a diagonal matrix. When $d\geq4$,
we do not know if this property still holds.

\subsection{\label{sub:DiagCase}The diagonal case}

For $k\geq1$, $\delta:=\left(\delta_{1},...,\text{\ensuremath{\delta}}_{k}\right)\in\mathbb{R}^{k}$,
let us denote by $Diag(\delta)\in\mathcal{M}_{k}\left(\mathbb{R}\right)$
the diagonal matrix with coefficients $\delta_{1},...,\text{\ensuremath{\delta}}_{k}$.
\begin{prop}
\label{prop:CondCdiag}For $d\geq2$ and $\alpha:=\left(\alpha_{1},...,\alpha_{d}\right)\in\left(\mathbb{R}_{+}^{*}\right)^{d}$,
$Diag\left(\alpha\right)$ satisfies Condition $(C)$ if and only
if

\begin{equation}
\frac{2}{\alpha_{k}}+\sum_{i\neq k}\frac{1}{\alpha_{i}}>\sqrt{\sum_{i\neq k}\frac{\lambda_{i}}{\alpha_{i}}\sum_{i\neq k}\frac{1}{\lambda_{i}\alpha_{i}}},\ 1\leq k\leq d.\label{eq:diag1}
\end{equation}
\end{prop}
\begin{proof}
For $1\leq k\leq d$, the symmetric matrix $D^{(k)}$ with coefficients
\[
D_{ij}^{(k)}=\frac{\lambda_{i}+\lambda_{j}}{2}\left(\alpha_{i}1_{\{i=j\}}+\alpha_{k}-\alpha_{k}1_{\{i=k\}}-\alpha_{k}1_{\{j=k\}}\right)
\]
for $1\leq i,j\leq d$, is positive definite on $e_{k}^{\perp}$ if
and only if the matrix $\tilde{D}^{(k)}$ defined as $D^{(k)}$ with
its $k$-th row and $k$-th column removed, is positive definite on
$\mathbb{R}^{d-1}$. Here we only show how to deal with the case $k=d$,
but the same arguments can be used for the indices $1\leq k\leq d-1$.
The matrix $\tilde{D}^{(d)}$ has coefficients
\[
\tilde{D}_{ij}^{(d)}=\frac{\lambda_{i}+\lambda_{j}}{2}\left(\alpha_{i}1_{\{i=j\}}+\alpha_{d}\right).
\]
for $1\leq i,j\leq d-1$. We define $\Delta:=Diag\left(\left(\sqrt{\lambda_{i}\alpha_{i}}\right)_{1\leq i\leq d-1}\right)$.
The matrix $\tilde{D}^{(d)}$ rewrites 

\begin{eqnarray*}
\tilde{D}^{(d)} & = & \Delta\Delta+\frac{\alpha_{d}}{2}\left(\left(\lambda_{i}+\lambda_{j}\right)\right)_{1\leq i,j\leq d-1}\\
 & = & \Delta\left(I_{d-1}+\frac{\alpha_{d}}{2}\left(\frac{\lambda_{i}+\lambda_{j}}{\sqrt{\lambda_{i}\alpha_{i}}\sqrt{\lambda_{j}\alpha_{j}}}\right)_{1\leq i,j\leq d-1}\right)\Delta\\
 & = & \Delta\left(I_{d-1}+\frac{\alpha_{d}}{2}\left(ab^{*}+ba^{*}\right)\right)\Delta,
\end{eqnarray*}
where $a=\left(\sqrt{\frac{\lambda_{i}}{\alpha_{i}}}\right)_{1\leq i\leq d-1}$
and $b=\left(\frac{1}{\sqrt{\lambda_{i}\alpha_{i}}}\right)_{1\leq i\leq d-1}$.
The matrix $\tilde{D}^{(d)}$ is positive definite if and only if
the matrix 
\[
\left(I_{d-1}+\frac{\alpha_{d}}{2}\left(ab^{*}+ba^{*}\right)\right),
\]
has positive eigenvalues. The columns of the matrix $ab^{*}+ba^{*}$
are linear combinations of $a$ and $b$. If $a$ and $b$ are not
colinear (resp. colinear), then the matrix $ab^{*}+ba^{*}$ has eigenvalues
$0$ with multiplicity $d-2$ and $\sum_{i=1}^{d-1}a_{i}b_{i}-\sqrt{\sum_{i=1}^{d-1}a_{i}^{2}}\sqrt{\sum_{i=1}^{d-1}b^{2}}<0$
for the eigenvector $a-\frac{\sqrt{\sum_{i=1}^{d-1}a_{i}^{2}}}{\sqrt{\sum_{i=1}^{d-1}b^{2}}}b$
(resp. $0$ with multiplicity $d-1$), and $\sum_{i=1}^{d-1}a_{i}b_{i}+\sqrt{\sum_{i=1}^{d-1}a_{i}^{2}}\sqrt{\sum_{i=1}^{d-1}b^{2}}>0$
for the eigenvector $a+\frac{\sqrt{\sum_{i=1}^{d-1}a_{i}^{2}}}{\sqrt{\sum_{i=1}^{d-1}b^{2}}}b$.
Thus, $\tilde{D}^{(d)}$ is definite positive if and only if 
\[
1+\frac{\alpha_{d}}{2}\left(\sum_{i\neq d}\frac{1}{\alpha_{i}}-\sqrt{\sum_{i\neq d}\frac{\lambda_{i}}{\alpha_{i}}\sum_{i\neq d}\frac{1}{\lambda_{i}\alpha_{i}}}\right)>0.
\]
which is equivalent to (\ref{eq:diag1}) for $k=d$. Using the same
arguments on $\tilde{D}^{(k)}$ for $1\leq k\leq d-1$, we obtain
(\ref{eq:diag1}).
\end{proof}
The choice $\alpha=\left(1,...,1\right)\in\mathbb{R}^{d}$ in Inequality
(\ref{eq:diag1}) gives a sufficient condition for the identity matrix
$I_{d}$ to satisfy Condition $(C)$. 
\begin{cor}
\label{cor:Id}If the condition

\begin{equation}
\underset{1\leq k\leq d}{\max}\ \sqrt{\sum_{i\neq k}\lambda_{i}\sum_{i\neq k}\frac{1}{\lambda_{i}}}<d+1,\label{eq:diag2}
\end{equation}
is satisfied then Condition $(C)$ is satisfied for the choice $\Gamma=I_{d}$.
In particular, if $\lambda_{1}=...=\lambda_{d}$, then $(C)$ is satisfied.
\end{cor}
Moreover, for $d=2$, Inequality (\ref{eq:diag2}) is always satisfied,
as for $k=1,2$, $\sqrt{\lambda_{k}\frac{1}{\lambda_{k}}}=1<3$.
\begin{cor}
If $d=2$, then Condition $(C)$ is satisfied for the choice $\Gamma=I_{2}$.
\end{cor}
From Inequality (\ref{eq:diag1}), we deduce a method to check numerically
whether there exists a diagonal matrix that satisfies Condition $(C)$.
We suppose that the values of $\lambda_{1},...,\lambda_{d}$ are not
equal, otherwise by Corollary \ref{cor:Id}, $I_{d}$ satisfies Condition
$(C)$. For $z=\left(z_{1},...,z_{d}\right)\in\left(\mathbb{R}^{2}\right)^{d}$,
let us denote by $\mathcal{C}\left(z\right)=\{x\in\mathbb{R}^{2}|\exists\mu_{1},...,\mu_{d}>0,\sum_{i=1}^{d}\mu_{i}=1,x=\sum_{i=1}^{d}\mu_{i}z_{i}\}$,
the strict convex envelope of $z$. Let $\alpha:=\left(\alpha_{1},...,\alpha_{d}\right)\in\left(\mathbb{R}_{+}^{*}\right)^{d}$,
such that $Diag\left(\alpha\right)$ satisfies Condition $(C)$. By
Proposition \ref{prop:CondCdiag}, (\ref{eq:diag1}) holds and rewrites
\begin{equation}
\left(\frac{1}{\alpha_{k}}+\sum_{i=1}^{d}\frac{1}{\alpha_{i}}\right)^{2}>\left(\sum_{i=1}^{d}\frac{\lambda_{i}}{\alpha_{i}}-\frac{\lambda_{k}}{\alpha_{k}}\right)\left(\sum_{i=1}^{d}\frac{1}{\lambda_{i}\alpha_{i}}-\frac{1}{\lambda_{k}\alpha_{k}}\right),\ 1\leq k\leq d.\label{eq:B3_1}
\end{equation}
If we define $\tilde{\lambda}=\sum_{i=1}^{d}\lambda_{i}\frac{\frac{1}{\alpha_{i}}}{\sum_{k=1}^{d}\frac{1}{\alpha_{k}}}\in\mathcal{C}\left(\left(\lambda_{i}\right)_{1\leq i\leq d}\right)$,
$\widetilde{\lambda^{-1}}=\sum_{i=1}^{d}\frac{1}{\lambda_{i}}\frac{\frac{1}{\alpha_{i}}}{\sum_{k=1}^{d}\frac{1}{\alpha_{k}}}\in\mathcal{C}\left(\left(\frac{1}{\lambda_{i}}\right)_{1\leq i\leq d}\right)$,
then Inequality (\ref{eq:B3_1}) writes 
\begin{equation}
1+\frac{\frac{1}{\alpha_{k}}}{\sum_{i=1}^{d}\frac{1}{\alpha_{i}}}\left(2+\frac{1}{\lambda_{k}}\tilde{\lambda}+\lambda_{k}\widetilde{\lambda^{-1}}\right)>\tilde{\lambda}\widetilde{\lambda^{-1}},\ 1\leq k\leq d.\label{eq:diag3}
\end{equation}
We deduce that there exists a diagonal matrix that satisfies Condition
$(C)$ if and only if there exists $\left(x,y\right)\in\mathcal{C}\left(\left(\lambda_{1},\frac{1}{\lambda_{1}}\right),...,\left(\lambda_{d},\frac{1}{\lambda_{d}}\right)\right)$
and a probability distribution $\left(p_{1},...,p_{d}\right)\in\left(\mathbb{R}_{*}^{+}\right)^{d}$
such that 
\begin{equation}
\sum_{i=1}^{d}\lambda_{i}p_{i}=x,\ \sum_{i=1}^{d}\frac{1}{\lambda_{i}}p_{i}=y,\text{ and }\ \forall k\in\{1,...,d\},\ p_{k}>\left(xy-1\right)\left(2+\frac{x}{\lambda_{k}}+\lambda_{k}y\right)^{-1}.\label{eq:conditionp}
\end{equation}
For $\left(x,y\right)\in\mathcal{C}\left(\left(\lambda_{1},\frac{1}{\lambda_{1}}\right),...,\left(\lambda_{d},\frac{1}{\lambda_{d}}\right)\right)$
, let us define

\begin{eqnarray*}
M_{0}(x,y) & = & \left(xy-1\right)\sum_{i=1}^{d}\left(2+\frac{x}{\lambda_{i}}+\lambda_{i}y\right)^{-1},\\
M_{-1}(x,y) & = & \left(xy-1\right)\sum_{i=1}^{d}\frac{1}{\lambda_{i}}\left(2+\frac{x}{\lambda_{i}}+\lambda_{i}y\right)^{-1},\\
M_{1}(x,y) & = & \left(xy-1\right)\sum_{i=1}^{d}\lambda_{i}\left(2+\frac{x}{\lambda_{i}}+\lambda_{i}y\right)^{-1},
\end{eqnarray*}
and if $M_{0}(x,y)<1$,
\[
X(x,y)=\frac{x-M_{1}(x,y)}{1-M_{0}(x,y)},\ Y(x,y)=\frac{y-M_{-1}(x,y)}{1-M_{0}(x,y)}.
\]

\begin{prop}
If $\lambda_{1},...,\lambda_{d}$ are not all equal, there exists
a diagonal matrix that satisfies Condition $(C)$ if and only if there
exists $(x,y)\in\mathcal{C}\left(\left(\lambda_{1},\frac{1}{\lambda_{1}}\right),...,\left(\lambda_{d},\frac{1}{\lambda_{d}}\right)\right)$
such that
\end{prop}
\begin{equation}
M_{0}(x,y)<1\text{ and }\left(X(x,y),Y(x,y)\right)\in\mathcal{C}\left(\left(\lambda_{i},\frac{1}{\lambda_{i}}\right)_{1\leq i\leq d}\right).\label{eq:c1}
\end{equation}

\begin{proof}
If a diagonal matrix $Diag(\alpha)$, where $\alpha:=\left(\alpha_{1},...,\alpha_{d}\right)\in\left(\mathbb{R}_{+}^{*}\right)^{d}$,
satisfies Condition $(C)$, then (\ref{eq:diag3}) holds and it is
easy to check that if we set $p_{i}=\frac{\frac{1}{\alpha_{i}}}{\sum_{k=1}^{d}\frac{1}{\alpha_{k}}}$
for $1\leq i\leq d$, $x=\sum_{i=1}^{d}\lambda_{i}p_{i}$ and $y=\sum_{i=1}^{d}\frac{p_{i}}{\lambda_{i}}$,
the conditions in (\ref{eq:conditionp}) hold so the conditions in
(\ref{eq:c1}) are satisfied, as 
\[
\left(\begin{array}{c}
X(x,y)\\
Y(x,y)
\end{array}\right)=\frac{1}{1-M_{\text{0}}(x,y)}\sum_{k=1}^{d}\left(p_{k}-\left(xy-1\right)\left(2+\frac{x}{\lambda_{k}}+\lambda_{k}y\right)^{-1}\right)\left(\begin{array}{c}
\lambda_{k}\\
\frac{1}{\lambda_{k}}
\end{array}\right).
\]
Conversely, if $\left(x,y\right)\in\mathcal{C}\left(\left(\lambda_{i},\frac{1}{\lambda_{i}}\right)_{1\leq i\leq d}\right)$
and satisfies (\ref{eq:c1}), then there exists a probability distribution
$\left(q_{1},...,q_{d}\right)\in\left(\mathbb{R}_{*}^{+}\right)^{d}$
such that $X(x,y)=\sum_{i=1}^{d}\lambda_{i}q_{i}$ and $Y(x,y)=\sum_{i=1}^{d}\frac{1}{\lambda_{i}}q_{i}$
and it is easy to check that if we set $p_{i}=\left(1-M_{0}(x,y)\right)q_{i}+\left(xy-1\right)\left(2+\frac{x}{\lambda_{i}}+\lambda_{i}y\right)^{-1}$
for $1\leq i\leq d$ then the conditions in (\ref{eq:conditionp})
are satisfied.
\end{proof}
Given a discretization parameter $n$ and the vector $\left(l_{1},...,l_{d}\right)$,
which is the nondecreasing reordering of $\left(\lambda_{1},...,\lambda_{d}\right)$,
the numerical procedure that follows builds a grid $\mathcal{G}$
that consists in $\left(d-1\right)\left(n-1\right)^{2}$ points of
$\mathcal{C}\left(\left(\lambda_{1},\frac{1}{\lambda_{1}}\right),...,\left(\lambda_{d},\frac{1}{\lambda_{d}}\right)\right)$
and returns the list $V$ of the points $(x,y)\in\mathcal{G}$ satisfying
(\ref{eq:c1}). The border of the convex $\mathcal{C}\left(\left(\lambda_{1},\frac{1}{\lambda_{1}}\right),...,\left(\lambda_{d},\frac{1}{\lambda_{d}}\right)\right)$
has a simple shape. Indeed, it is a polygon with vertices $\left(l_{i},\frac{1}{l_{i}}\right)_{1\leq i\leq d}$
and edges $\left\{ \left(l_{i},\frac{1}{l_{i}}\right),\left(l_{i+1},\frac{1}{l_{i+1}}\right)\right\} _{1\leq i\leq d}$,
where we define $\left(l_{d+1},\frac{1}{l_{d+1}}\right)=\left(l_{1},\frac{1}{l_{1}}\right)$.

\begin{algorithm}[h]
\begin{algorithmic}[1] 
\State{$V=[]$}
\For{$i=1,...,d-1$} 
\For{$k_1=1,...,n-1$}
\State{$x=l_i+(l_{i+1}-l_{i})\frac{k_1}{n}$}
\State{$y_{min} = \frac{1}{l_i} \frac{n-k_1}{n} + \frac{1}{l_{i+1}} \frac{k_1}{n}$}
\State{$y_{max} = \frac{1}{l_1} - \frac{x-l_1}{l_1 l_{d}}$}
\For{$k_2 = 1,...,n-1$} 
\State{$y = y_{min} + (y_{max} - y_{min})\frac{k_2}{n}$}
\State{$M_0 = (xy-1)\sum_{k=1}^{d} \left(2+\frac{x}{l_i} + l_i y \right)^{-1} $}
\If{$M_0 < 1$}
\State{$M_1 = (xy-1)\sum_{k=1}^{d} l_i \left(2+\frac{x}{l_i} + l_i y \right)^{-1} $}
\State{$X = \frac{x-M_1}{1-M_{0}}$}
\If{$l_1 < X < l_d$}
\State{$j = \text{Sum}(X > l)$}
\State{$z_{min} = \frac{1}{l_j} - \frac{X-l_j}{l_j l_{j+1}}$}
\State{$z_{max} = \frac{1}{l_1} - \frac{X-l_1}{l_1 l_{d}}$}
\State{$M_{-1} = (xy-1)\sum_{k=1}^{d} \frac{1}{l_i} \left(2+\frac{x}{l_i} + l_i y \right)^{-1} $}
\State{$Y = \frac{y-M_{-1}}{1-M_{0}}$}
\If{$z_{min} < Y < z_{max}$} 
\State{$V=(x,y)::V$} 
\EndIf
\EndIf
\EndIf
\EndFor
\EndFor
\EndFor
\State{\Return{V}}

\end{algorithmic}

\end{algorithm}

In Figure 1 below, we give an idea of the shape of this convex and
we illustrate the output of the numerical procedure for $n=200$,
$d=5$, $\lambda_{1}=1$, $\lambda_{2}=2$, $\lambda_{3}=3$, $\lambda_{4}=5$
and $\lambda_{5}=10$. The border of the convex envelope $\mathcal{C}\left(\left(\lambda_{1},\frac{1}{\lambda_{1}}\right),...,\left(\lambda_{5},\frac{1}{\lambda_{5}}\right)\right)$
is colored in red, and the points $(x,y)$ in $\mathcal{C}\left(\left(\lambda_{1},\frac{1}{\lambda_{1}}\right),...,\left(\lambda_{5},\frac{1}{\lambda_{5}}\right)\right)$
satisfying (\ref{eq:c1}) are colored in black. Condition $(C)$ is
thus satisfied in this situation.

\begin{figure}[H]
\begin{centering}
\includegraphics[scale=0.4]{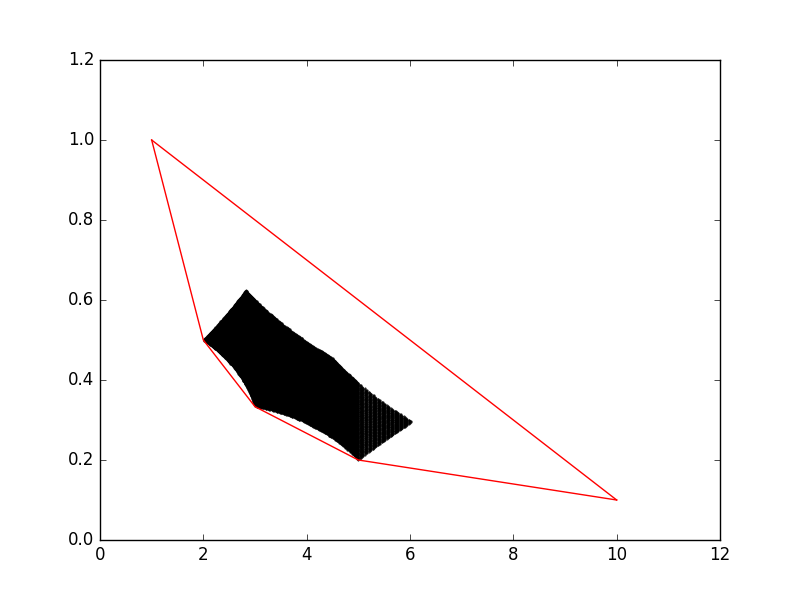}
\par\end{centering}

\caption{Condition $(C)$ is satisfied for $d=5$, $\lambda_{1}=1$, $\lambda_{2}=2$,
$\lambda_{3}=3$, $\lambda_{4}=5$ and $\lambda_{5}=10$, where $n=200$.}
\end{figure}
The advantage of such a numerical procedure is that it operates in
a bounded convex of $\mathbb{R}^{2}$ with a simple shape, instead
of a bounded convex of $\mathbb{R}^{d}$, as Inequality (\ref{eq:diag1})
would suggest.

\subsection{\label{sub:d=00003D3}The case $d=3$}

In the following, we study the case $d=3$. We recall that 
\[
r_{1}=\frac{\lambda_{3}}{\lambda_{2}}+\frac{\lambda_{2}}{\lambda_{3}}\geq2,\ r_{2}=\frac{\lambda_{3}}{\lambda_{1}}+\frac{\lambda_{1}}{\lambda_{3}}\ge2,\ r_{3}=\frac{\lambda_{1}}{\lambda_{2}}+\frac{\lambda_{2}}{\lambda_{1}}\geq2.
\]
Let us first explicit the link between the values of $r_{1},r_{2},r_{3}$.
\begin{lem}
\label{lem:linkr1r2r3}The values of $r_{1},r_{2},r_{3}$ are linked
by
\[
r_{3}\in\left\{ \frac{1}{2}\left(r_{1}r_{2}-\sqrt{\left(r_{1}^{2}-4\right)\left(r_{2}^{2}-4\right)}\right),\frac{1}{2}\left(r_{1}r_{2}+\sqrt{\left(r_{1}^{2}-4\right)\left(r_{2}^{2}-4\right)}\right)\right\} .
\]
\end{lem}
\begin{proof}
As $r_{1}=\frac{\lambda_{3}}{\lambda_{2}}+\frac{\lambda_{2}}{\lambda_{3}}\geq2$
and $r_{2}=\frac{\lambda_{1}}{\lambda_{3}}+\frac{\lambda_{3}}{\lambda_{1}}\geq2$,
we have that $\frac{\lambda_{2}}{\lambda_{3}},\frac{\lambda_{3}}{\lambda_{2}}\in\left\{ \frac{1}{2}\left(r_{1}-\sqrt{r_{1}^{2}-4}\right),\frac{1}{2}\left(r_{1}+\sqrt{r_{1}^{2}-4}\right)\right\} $
and $\frac{\lambda_{1}}{\lambda_{3}},\frac{\lambda_{3}}{\lambda_{1}}\in\left\{ \frac{1}{2}\left(r_{2}-\sqrt{r_{2}^{2}-4}\right),\frac{1}{2}\left(r_{2}+\sqrt{r_{2}^{2}-4}\right)\right\} $.
As $\frac{\lambda_{1}}{\lambda_{2}}=\frac{\lambda_{1}}{\lambda_{3}}\frac{\lambda_{3}}{\lambda_{2}}$,
we deduce the two values that $r_{3}$ can take.
\end{proof}
We now give the main result concerning the case $d=3$.
\begin{prop}
\label{prop:CondC3D}There is equivalence between:

(i) The inequality 
\begin{equation}
\frac{1}{\sqrt{\left(r_{1}-2\right)\left(r_{2}-2\right)}}+\frac{1}{\sqrt{\left(r_{2}-2\right)\left(r_{3}-2\right)}}+\frac{1}{\sqrt{\left(r_{1}-2\right)\left(r_{3}-2\right)}}>\frac{1}{4},\label{eq:ineqdiag3}
\end{equation}
holds, with the convention $\frac{1}{0}=+\infty$.

(ii) Condition $(C)$ is satisfied by a diagonal matrix.

(iii) Condition $(C)$ is satisfied.
\end{prop}
We now prove Proposition \ref{prop:CondC3D}. To show that $(i)\Rightarrow(ii)$,
we first study the case where there cardinality of $\{\lambda_{1},\lambda_{2},\lambda_{3}\}$
is smaller than $3$, which is equivalent to $\min(r_{1},r_{2},r_{3})=2$
and implies (\ref{eq:ineqdiag3}).
\begin{lem}
\label{lem:min=00003D2}If $\min\left(r_{1},r_{2},r_{3}\right)=2$,
then there exists a diagonal matrix that satisfies Condition $(C)$,
i.e. $(ii)$ holds.\end{lem}
\begin{proof}
Let $\alpha_{1},\alpha_{2},\alpha_{3}>0$ and $p_{i}=\frac{\frac{1}{\alpha_{i}}}{\frac{1}{\alpha_{1}}+\frac{1}{\alpha_{2}}+\frac{1}{\alpha_{3}}}$,
$i=1,2,3$. Inequality (\ref{eq:B3_1}) rewrites, for $k=1$,
\[
1+2p_{1}+p_{1}^{2}>\left(\frac{\lambda_{2}}{\lambda_{3}}+\frac{\lambda_{3}}{\lambda_{2}}\right)p_{2}p_{3}+p_{2}^{2}+p_{3}^{2}.
\]
Writing $p_{1}^{2}=\left(1-p_{2}-p_{3}\right)^{2}$, and then using
the fact that $p_{1}+p_{2}+p_{3}=1$, we obtain that
\begin{equation}
4>p_{2}p_{3}\frac{\left(r_{1}-2\right)}{p_{1}}.\label{eq:ineqdiag3p1p2p3_1}
\end{equation}
With similar computations for $k=2,3$ we moreover obtain 
\begin{equation}
4>p_{1}p_{3}\frac{\left(r_{2}-2\right)}{p_{2}}\text{ and }4>p_{1}p_{2}\frac{\left(r_{3}-2\right)}{p_{3}}.\label{eq:ineqdiag3p1p2p3_2}
\end{equation}
To prove Lemma \ref{lem:min=00003D2}, it is sufficient to exhibit
a probability distribution $(p_{1},p_{2},p_{3})\in\left(\mathbb{R}_{*}^{+}\right)^{3}$
such that Inequalities (\ref{eq:ineqdiag3p1p2p3_1})- (\ref{eq:ineqdiag3p1p2p3_2})
are satisfied. In the case where $r_{1}=r_{2}=2$, we have that $\lambda_{1}=\lambda_{2}=\lambda_{3}$
so $r_{3}=2$, and the choice $p_{1}=p_{2}=p_{3}=\frac{1}{3}$ satisfies
(\ref{eq:ineqdiag3p1p2p3_1})-(\ref{eq:ineqdiag3p1p2p3_2}). In the
case where $r_{1}=2$, $r_{2}>2$ and $r_{3}>2$, we have that it
is sufficient to choose $p_{1}\in\left(0,\min\left(1,\frac{4}{r_{2}-2},\frac{4}{r_{3}-2}\right)\right)$
and $p_{2}=p_{3}=\frac{1-p_{1}}{2}$ to satisfy (\ref{eq:ineqdiag3p1p2p3_1})-(\ref{eq:ineqdiag3p1p2p3_2}).
\end{proof}
To complete the proof of $(i)\Rightarrow(ii)$, we show that in the
case $\min\left(r_{1},r_{2},r_{3}\right)>2$, if (\ref{eq:ineqdiag3})
holds, then Condition $(C)$ is satisfied by a diagonal matrix.
\begin{lem}
\label{lem:DiagSaturee} Let us assume that $\min(r_{1},r_{2},r_{3})>2$.
If Inequality (\ref{eq:ineqdiag3}) holds, then $(ii)$ holds.\end{lem}
\begin{proof}
Using the proof of Lemma \ref{lem:min=00003D2}, it is easy to check
that that if Inequality (\ref{eq:ineqdiag3}) is satisfied then (\ref{eq:ineqdiag3p1p2p3_1})-(\ref{eq:ineqdiag3p1p2p3_2})
are satisfied for the choice $p_{i}=\frac{\sqrt{r_{i}-2}}{\sqrt{r_{1}-2}+\sqrt{r_{2}-2}+\sqrt{r_{3}-2}}>0$,
for $i=1,2,3$, so that the matrix $Diag\left(\frac{1}{p_{1}},\frac{1}{p_{2}},\frac{1}{p_{3}}\right)$
satisfies Condition $(C)$.
\end{proof}
As the relation $(ii)\Rightarrow(iii)$ is trivial, we have obtained
$(i)\Rightarrow(ii)\Rightarrow(iii)$. To prove $(iii)\Rightarrow(i)$,
by Lemma \ref{lem:min=00003D2}, it is sufficient to show that in
the case $\min\left(r_{1},r_{2},r_{3}\right)>2$, if Condition $(C)$
is satisfied then (\ref{eq:ineqdiag3}) holds. Let us remark that
we can assume without loss of generality that $r_{1}\leq r_{2}\leq r_{3}$,
so in what follows we suppose that 
\[
2<r_{1}\leq r_{2}\leq r_{3}.
\]
The next lemma deals with the case $\frac{r_{3}-2}{r_{3}+2}\geq\frac{r_{1}-2}{r_{1}+2}+\frac{r_{2}-2}{r_{2}+2}$.
\begin{lem}
\label{lem:g3leqg1plusg2}Let us assume that $\min(r_{1},r_{2},r_{3})>2$.
Then Inequality (\ref{eq:ineqdiag3}) is equivalent to 
\begin{equation}
\left\{ \sqrt{\left(r_{1}-2\right)\left(r_{2}-2\right)}\leq4\right\} \text{ or }\left\{ \sqrt{\left(r_{1}-2\right)\left(r_{2}-2\right)}>4\text{ and }r_{3}<16\left(\frac{\sqrt{r_{1}-2}+\sqrt{r_{2}-2}}{\sqrt{\left(r_{1}-2\right)\left(r_{2}-2\right)}-4}\right)^{2}+2\right\} .\label{eq:equidiag}
\end{equation}
In particular, if moreover $\frac{r_{3}-2}{r_{3}+2}\geq\frac{r_{1}-2}{r_{1}+2}+\frac{r_{2}-2}{r_{2}+2}$,
then Inequality (\ref{eq:ineqdiag3}) holds.\end{lem}
\begin{proof}
Under the assumption that $\min(r_{1},r_{2},r_{3})>2$, Inequality
(\ref{eq:ineqdiag3}) rewrites $\frac{1}{\sqrt{r_{3}-2}}\left(\frac{1}{\sqrt{r_{1}-2}}+\frac{1}{\sqrt{r_{2}-2}}\right)>\frac{1}{4}-\frac{1}{\sqrt{\left(r_{1}-2\right)\left(r_{2}-2\right)}},$
so it is equivalent to (\ref{eq:equidiag}). The term $\frac{r_{1}-2}{r_{1}+2}+\frac{r_{2}-2}{r_{2}+2}$
rewrites

\begin{eqnarray*}
\frac{r_{1}-2}{r_{1}+2}+\frac{r_{2}-2}{r_{2}+2} & = & \frac{2\sqrt{\left(r_{1}-2\right)\left(r_{2}-2\right)}\left(\sqrt{\left(r_{1}-2\right)\left(r_{2}-2\right)}-4\right)+4\left(\sqrt{r_{1}-2}+\sqrt{r_{2}-2}\right)^{2}}{\left(\sqrt{\left(r_{1}-2\right)\left(r_{2}-2\right)}-4\right)^{2}+4\left(\sqrt{r_{1}-2}+\sqrt{r_{2}-2}\right)^{2}}.
\end{eqnarray*}
If $\sqrt{\left(r_{1}-2\right)\left(r_{2}-2\right)}>4$, then $\frac{r_{1}-2}{r_{1}+2}+\frac{r_{2}-2}{r_{2}+2}>1$.
Thus, since $1>\frac{r_{3}-2}{r_{3}+2}$, if 
\[
\frac{r_{3}-2}{r_{3}+2}\geq\frac{r_{1}-2}{r_{1}+2}+\frac{r_{2}-2}{r_{2}+2},
\]
then $\sqrt{\left(r_{1}-2\right)\left(r_{2}-2\right)}\leq4$ and Inequality
(\ref{eq:ineqdiag3}) holds.
\end{proof}
Let us now suppose that Condition $(C)$ is satisfied by a matrix
$\Gamma\in\mathcal{S}_{3}^{++}\left(\mathbb{R}\right)$. We consider,
for $k=1,2,3$, the matrices $\Gamma^{(k)}$ with coefficients 
\[
\Gamma_{ij}^{(k)}=\frac{\lambda_{i}+\lambda_{j}}{2}\left(\Gamma_{ij}+\Gamma_{kk}-\Gamma_{ik}-\Gamma_{jk}\right),\ 1\leq i,j\leq3.
\]
We define 
\[
v_{1}:=\Gamma_{22}+\Gamma_{33}-2\Gamma_{23},\ v_{2}:=\Gamma_{11}+\Gamma_{33}-2\Gamma_{13},\ v_{3}:=\Gamma_{11}+\Gamma_{22}-2\Gamma_{12}.
\]
The matrix $\Gamma^{(3)}$ rewrites
\[
\Gamma^{(3)}=\left(\begin{array}{ccc}
\lambda_{1}v_{2} & \frac{\lambda_{1}+\lambda_{2}}{2}\times\frac{v_{1}+v_{2}-v_{3}}{2} & 0\\
\frac{\lambda_{1}+\lambda_{2}}{2}\times\frac{v_{1}+v_{2}-v_{3}}{2} & \lambda_{2}v_{1} & 0\\
0 & 0 & 0
\end{array}\right).
\]
We deduce that $\Gamma^{(3)}$ is positive definite on $e_{3}^{\perp}$
if and only if 
\begin{equation}
v_{1}>0,v_{2}>0\label{eq:v1v2pos}
\end{equation}
and the determinant of the matrix $\left(\Gamma^{(3)}\right)_{1\leq i,j\leq2}$
is positive, which rewrites $16v_{1}v_{2}>\left(\sqrt{\frac{\lambda_{2}}{\lambda_{1}}}+\sqrt{\frac{\lambda_{1}}{\lambda_{2}}}\right)^{2}\left(v_{1}+v_{2}-v_{3}\right)^{2}=\left(r_{3}+2\right)\left(v_{1}+v_{2}-v_{3}\right)^{2}$
and therefore
\begin{equation}
-4\frac{r_{3}-2}{r_{3}+2}v_{1}v_{2}>v_{1}^{2}+v_{2}^{2}+v_{3}^{2}-2\left(v_{1}v_{2}+v_{2}v_{3}+v_{1}v_{3}\right),\label{eq:u1u2u3_1}
\end{equation}
With similar computations for $\Gamma^{(1)}$ and $\Gamma^{(2)}$,
we obtain the additional inequalities 
\begin{eqnarray}
v_{3} & > & 0,\label{eq:v3pos}\\
-4\frac{r_{2}-2}{r_{2}+2}v_{1}v_{3} & > & v_{1}^{2}+v_{2}^{2}+v_{3}^{2}-2\left(v_{1}v_{2}+v_{2}v_{3}+v_{1}v_{3}\right),\label{eq:u1u2u3_2}\\
-4\frac{r_{1}-2}{r_{1}+2}v_{2}v_{3} & > & v_{1}^{2}+v_{2}^{2}+v_{3}^{2}-2\left(v_{1}v_{2}+v_{2}v_{3}+v_{1}v_{3}\right).\label{eq:u1u2u3_3}
\end{eqnarray}
If Condition $(C)$ is satisfied, then there exists $v_{1},v_{2},v_{3}$
satisfying (\ref{eq:v1v2pos})-(\ref{eq:u1u2u3_3}). Let us define
for $i=1,2,3$, $\gamma_{i}=4\frac{r_{i}-2}{r_{i}+2}\in(0,4)$. If
we assume that moreover, $\gamma_{3}v_{1}v_{2}=\gamma_{1}v_{2}v_{3}=\gamma_{2}v_{1}v_{3}$,
that is $\frac{\gamma_{3}}{v_{3}}=\frac{\gamma_{1}}{v_{1}}=\frac{\gamma_{2}}{v_{2}}$,
then we have that
\begin{equation}
0>\gamma_{1}^{2}+\gamma_{2}^{2}+\gamma_{3}^{2}-2\left(\gamma_{1}\gamma_{2}+\gamma_{2}\gamma_{3}+\gamma_{1}\gamma_{3}\right)+\gamma_{1}\gamma_{2}\gamma_{3}.\label{eq:ineqMatpleineSaturee}
\end{equation}

In the case $\frac{r_{3}-2}{r_{3}+2}\geq\frac{r_{2}-2}{r_{2}+2}+\frac{r_{1}-2}{r_{1}+2}$,
by Lemma \ref{lem:g3leqg1plusg2}, Inequality (\ref{eq:ineqdiag3})
holds. We show in Lemma \ref{lem:lastonehardcore} below that under
the assumption $\frac{r_{3}-2}{r_{3}+2}<\frac{r_{2}-2}{r_{2}+2}+\frac{r_{1}-2}{r_{1}+2}$,
Condition $(C)$ implies Inequality (\ref{eq:ineqMatpleineSaturee}).
To conclude the proof of $(iii)\Rightarrow(i)$ and therefore the
proof of Proposition \ref{prop:CondC3D}, we show in Lemma \ref{lem:equivdiagpleine}
below that (\ref{eq:ineqMatpleineSaturee}) implies (\ref{eq:ineqdiag3}).
\begin{lem}
\label{lem:equivdiagpleine}Let us assume that $2<r_{1}\leq r_{2}\leq r_{3}$,
then Inequality (\ref{eq:ineqMatpleineSaturee}) implies Inequality
(\ref{eq:ineqdiag3}).\end{lem}
\begin{proof}
We see the term on the r.h.s. of Inequality (\ref{eq:ineqMatpleineSaturee})
as a second degree polynomial in the variable $\gamma_{3}$, which
has two distinct roots $z_{-}<z_{+}$. Indeed, $\gamma_{1},\gamma_{2}\in(0,4)$,
and the discriminant of the polynomial is $\gamma_{1}\gamma_{2}\left(4-\gamma_{1}\right)\left(4-\gamma_{2}\right)>0$.
As $\gamma_{3}=4\frac{r_{3}-2}{r_{3}+2}$, Inequality (\ref{eq:ineqMatpleineSaturee})
is equivalent to $r_{3}\in\left(\frac{8+2z_{-}}{4-z_{-}},\frac{8+2z_{+}}{4-z_{+}}\right)$,
where 
\[
\frac{8+2z_{\pm}}{4-z_{\pm}}=16\left(\frac{\sqrt{r_{1}-2}\pm\sqrt{r_{2}-2}}{\sqrt{\left(r_{1}-2\right)\left(r_{2}-2\right)}\mp4}\right)^{2}+2.
\]
By Lemma \ref{lem:g3leqg1plusg2}, we conclude that Inequality (\ref{eq:ineqMatpleineSaturee})
implies Inequality (\ref{eq:ineqdiag3}).\end{proof}
\begin{lem}
\label{lem:lastonehardcore}Let us assume that $2<r_{1}\leq r_{2}\leq r_{3}$,
and that $\frac{r_{3}-2}{r_{3}+2}<\frac{r_{2}-2}{r_{2}+2}+\frac{r_{1}-2}{r_{1}+2}$.
If Condition $(C)$ is satisfied then Inequality (\ref{eq:ineqMatpleineSaturee})
holds.\end{lem}
\begin{proof}
Let us define the function $f:\left(\mathbb{R}_{+}^{*}\right)^{3}\rightarrow\mathbb{R}$
by $f\left(a_{1},a_{2},a_{3}\right)=2\left(a_{1}+a_{2}+a_{3}\right)-\frac{a_{1}a_{2}}{a_{3}}-\frac{a_{2}a_{3}}{a_{1}}-\frac{a_{1}a_{3}}{a_{2}}$.
Reformulating Inequalities (\ref{eq:u1u2u3_1})-(\ref{eq:u1u2u3_3})
with the change of variables $a_{i}=\frac{1}{v_{i}},i=1,2,3$, we
obtain that
\[
f(a_{1},a_{2},a_{3})>\max\{\gamma_{1}a_{1},\gamma_{2}a_{2},\gamma_{3}a_{3}\}.
\]
Under the assumption that $2<r_{1}\leq r_{2}\leq r_{3}$, we have
that $0<\gamma_{1}\leq\gamma_{2}\leq\gamma_{3}$, and let us remark
that 
\[
\max\{\gamma_{1}a_{1},\gamma_{2}a_{2},\gamma_{3}a_{3}\}\geq\max\{\gamma_{3}a_{(1)},\gamma_{2}a_{(2)},\gamma_{1}a_{(3)}\},
\]
where $\left(a_{(1)},a_{(2)},a_{(3)}\right)$ is the nondecreasing
reordering of $(a_{1},a_{2},a_{3})$. Let $\mathcal{R}$ be the set
of elements $(a_{1},a_{2},a_{3})\in\left(\mathbb{R}_{+}^{*}\right)^{3}$
such that $a_{1}\leq a_{2}\leq a_{3}$ and 
\begin{equation}
f(a_{1},a_{2},a_{3})>\max\{\gamma_{1}a_{3},\gamma_{2}a_{2},\gamma_{3}a_{1}\}.\label{eq:critere}
\end{equation}
If Condition $(C)$ is satisfied then $\mathcal{R}$ is not empty.
Let us remark first that as both sides of (\ref{eq:critere}) are
homogeneous of order 1, $\mathcal{R}$ is stable by scaling: if $(a_{1},a_{2},a_{3})\in\mathcal{R}$
then for $\zeta>0$, $(\zeta a_{1},\zeta a_{2},\zeta a_{3})\in\mathcal{R}$.
Moreover, if we assume that there exists $\left(a_{1},a_{2},a_{3}\right)\in\mathcal{R}$
such that $\gamma_{1}a_{3}=\gamma_{2}a_{2}=\gamma_{3}a_{1}=:\varDelta$,
then we can check that Inequality (\ref{eq:ineqMatpleineSaturee})
is satisfied. Indeed, we have that $a_{3}=\frac{\Delta}{\gamma_{1}},a_{2}=\frac{\Delta}{\gamma_{2}},a_{3}=\frac{\Delta}{\gamma_{1}}$,
and 
\begin{equation}
f\left(\frac{\Delta}{\gamma_{3}},\frac{\Delta}{\gamma_{2}},\frac{\Delta}{\gamma_{1}}\right)>\Delta.\label{eq:f}
\end{equation}
Multiplying both sides of (\ref{eq:f}) by $\frac{\gamma_{1}\gamma_{2}\gamma_{3}}{\Delta}$,
we obtain (\ref{eq:ineqMatpleineSaturee}).

Let $\left(q_{1},q_{2},q_{3}\right)\in\mathcal{R}$. To prove that
there exists $\left(u_{1},u_{2},u_{3}\right)\in\mathcal{R}$, such
that $\gamma_{1}u_{3}=\gamma_{2}u_{2}=\gamma_{3}u_{1}$, and conclude
with the previous argument, we construct a path included into $\mathcal{R}$
that goes from $\left(q_{1},q_{2},q_{3}\right)$ to $\left(u_{1},u_{2},u_{3}\right)$.
We now distinguish the three cases.

Case $1$: $\gamma_{3}q_{1}\leq\max\left(\gamma_{2}q_{2},\gamma_{1}q_{3}\right)=\gamma_{1}q_{3}$.
Let us remark that if $a_{1}\leq a_{2}\leq a_{3}$, the partial derivative
$\partial_{a_{1}}f$ satisfies 
\begin{equation}
\partial_{a_{1}}f(a_{1},a_{2},a_{3})=\frac{a_{2}a_{3}}{a_{1}^{2}}-\left(\sqrt{\frac{a_{2}}{a_{3}}}-\sqrt{\frac{a_{3}}{a_{2}}}\right)^{2}\geq\frac{a_{3}}{a_{2}}\left(\left(\frac{a_{2}}{a_{1}}\right)^{2}-1\right)\geq0.\label{eq:diffa1}
\end{equation}
For $x\in\left[q_{1},\min\left(q_{2},\frac{\gamma_{1}}{\gamma_{3}}q_{3}\right)\right]$,
$f(x,q_{2},q_{3})\geq f(q_{1},q_{2},q_{3})$. If $\min\left(q_{2},\frac{\gamma_{1}}{\gamma_{3}}q_{3}\right)=\frac{\gamma_{1}}{\gamma_{3}}q_{3}$,
then for $\tilde{q}_{1}=\frac{\gamma_{1}}{\gamma_{3}}q_{3}$, we have
that
\begin{equation}
f(\tilde{q}_{1},q_{2},q_{3})\geq f(q_{1},q_{2},q_{3})>\gamma_{1}q_{3}=\gamma_{3}\tilde{q}_{1}\geq\gamma_{2}q_{2}.\label{eq:gamma123_1}
\end{equation}
We scale (\ref{eq:gamma123_1}) and define $\zeta=\frac{\gamma_{1}}{\tilde{q}_{1}}=\frac{\gamma_{3}}{q_{3}}>0$
so that $\zeta\tilde{q}_{1}=\gamma_{1}$ and $\zeta q_{3}=\gamma_{3}$.
We have that 
\begin{equation}
f(\zeta\tilde{q}_{1},\zeta q_{2},\zeta q_{3})\geq f(\zeta q_{1},\zeta q_{2},\zeta q_{3})>\gamma_{3}\gamma_{1}\geq\gamma_{2}\zeta q_{2},\label{eq:gamma123_1bis}
\end{equation}
We now increase $q_{2}$ in (\ref{eq:gamma123_1bis}). For $a_{1},a_{2},a_{3}>0$,
such that $a_{1}\leq a_{2}\leq a_{3}$ and $a_{2}\leq\frac{a_{3}a_{1}}{a_{3}-a_{1}}$
we have that 
\[
\partial_{a_{2}}f(a_{1},a_{2},a_{3})=a_{1}a_{3}\left(\frac{1}{a_{2}^{2}}-\left(\frac{1}{a_{1}}-\frac{1}{a_{3}}\right)^{2}\right)\geq0.
\]
By hypothesis, 
\[
q_{2}\leq\frac{\gamma_{1}\gamma_{3}}{\zeta\gamma_{2}}\leq\frac{1}{\zeta}\frac{\gamma_{1}\gamma_{3}}{\gamma_{3}-\gamma_{1}}=\frac{q_{3}\tilde{q}_{1}}{q_{3}-\tilde{q}_{1}},
\]
so for $z\in[q_{2},\frac{\gamma_{1}\gamma_{3}}{\zeta\gamma_{2}}]$,
$f(\zeta\tilde{q}_{1},\zeta z,\zeta q_{3})\geq f(\zeta\tilde{q}_{1},\zeta q_{2},\zeta q_{3})$
and in particular for $\hat{q}_{2}=\frac{\gamma_{1}\gamma_{3}}{\zeta\gamma_{2}}$,
we have that 
\[
f(\zeta\tilde{q}_{1},\zeta\hat{q}_{2},\zeta q_{3})>\gamma_{1}\zeta q_{3}=\gamma_{3}\zeta\tilde{q}_{1}=\gamma_{2}\zeta\hat{q}_{2},
\]
so we obtain (\ref{eq:ineqMatpleineSaturee}). If $\min\left(q_{2},\frac{\gamma_{1}}{\gamma_{3}}q_{3}\right)=q_{2}$,
as the function $x\rightarrow f(x,q_{2},q_{3})$ is nondecreasing
for $x\in[q_{1},q_{2}]$, we have that 
\[
f(q_{2},q_{2},q_{3})\geq f(q_{1},q_{2},q_{3})>\gamma_{1}q_{3}\geq\gamma_{3}q_{2}\geq\gamma_{2}q_{2}.
\]
As the function $g:\mathbb{R}^{2}\rightarrow\mathbb{R}$ defined by
$g(a_{2},a_{3})=f(a_{2},a_{2},a_{3})$ satisfies $\partial_{a_{2}}g\left(a_{2},a_{3}\right)=4-2\frac{a_{2}}{a_{3}}\geq0$
if $a_{2}\leq a_{3}$, we have that for $y\in\left[q_{2},\frac{\gamma_{1}q_{3}}{\gamma_{3}}\right]$,
$f(y,y,q_{3})\geq f(q_{1},q_{2},q_{3})$. In particular for $\tilde{q}_{1}=\tilde{q}_{2}=\frac{\gamma_{1}q_{3}}{\gamma_{3}}$,
we have 
\begin{equation}
f(\tilde{q}_{1},\tilde{q}_{2},q_{3})\geq f(q_{1},q_{2},q_{3})>\gamma_{1}q_{3}=\gamma_{3}\tilde{q}_{1}\geq\gamma_{2}\tilde{q}_{2}.\label{eq:gamma123_2}
\end{equation}
and (\ref{eq:gamma123_2}) is treated in the same way as (\ref{eq:gamma123_1}).

Case $2$: $\gamma_{3}q_{1}\leq\max\left(\gamma_{2}q_{2},\gamma_{1}q_{3}\right)=\gamma_{2}q_{2}$.
Using (\ref{eq:diffa1}), as $\partial_{a_{1}}f(x,q_{2},q_{3})\geq0$
for $x\in[q_{1},q_{2}]$, we have that for $\tilde{q}_{1}=\frac{\gamma_{2}}{\gamma_{3}}q_{2}\in[q_{1},q_{2}]$,
\[
f(\tilde{q}_{1},q_{2},q_{3})\geq f(q_{1},q_{2},q_{3})>\gamma_{2}q_{2}=\gamma_{3}\tilde{q}_{1}\geq\gamma_{1}q_{3}.
\]
For $\theta=\frac{\gamma_{3}}{q_{2}}=\frac{\gamma_{2}}{\tilde{q}_{1}}$,
we have that 
\[
f(\theta\tilde{q}_{1},\theta q_{2},\theta q_{3})>\gamma_{3}\gamma_{2}\geq\gamma_{1}\theta q_{3}.
\]
As $\gamma_{3}-\gamma_{2}<\gamma_{1}$, we have that 
\[
q_{3}\leq\frac{\gamma_{2}\gamma_{3}}{\theta\gamma_{1}}<\frac{1}{\theta}\frac{\gamma_{2}\gamma_{3}}{\gamma_{3}-\gamma_{2}}=\frac{\tilde{q}_{1}q_{2}}{q_{2}-\tilde{q}_{1}},
\]
Moreover, for $0<a_{1}\leq a_{2}\leq a_{3}$ such that $a_{3}\leq\frac{a_{2}a_{1}}{a_{2}-a_{1}}$,
\begin{equation}
\partial_{a_{3}}f(a_{1},a_{2},a_{3})=a_{1}a_{2}\left(\frac{1}{a_{3}^{2}}-\left(\frac{1}{a_{1}}-\frac{1}{a_{2}}\right)^{2}\right)\geq0,\label{eq:dfa3}
\end{equation}
so for $\hat{q}_{3}=\frac{\gamma_{2}\gamma_{3}}{\theta\gamma_{1}}$,
we obtain 
\[
f(\theta\tilde{q}_{1},\theta q_{2},\theta\hat{q}_{3})\geq f(\theta\tilde{q}_{1},\theta q_{2},\theta q_{3})>\gamma_{2}\theta q_{2}=\gamma_{3}\theta\tilde{q}_{1}=\gamma_{1}\theta\hat{q}_{3},
\]
and we deduce (\ref{eq:ineqMatpleineSaturee}).

Case 3: $\gamma_{3}q_{1}>\max\left(\gamma_{2}q_{2},\gamma_{1}q_{3}\right)$.
We show that there exists $\tilde{q}_{3}\geq q_{3}$ such that 
\begin{equation}
f(q_{1},q_{2},\tilde{q}_{3})\geq f(q_{1},q_{2},q_{3})>\gamma_{3}q_{1}=\gamma_{1}\tilde{q}_{3}\geq\gamma_{2}q_{2}.\label{eq:lastsituation}
\end{equation}
Let $\eta=\frac{\gamma_{1}}{q_{1}}$ so that $\eta q_{1}=\gamma_{1}$.
We have that 
\[
f(\eta q_{1},\eta q_{2},\eta q_{3})>\gamma_{3}\gamma_{1}>\max\left(\gamma_{1}\eta q_{3},\gamma_{2}\eta q_{2}\right).
\]
As $q_{2}<\frac{\gamma_{1}\gamma_{3}}{\eta\gamma_{2}}$ and $\gamma_{3}-\gamma_{2}<\gamma_{1}$,
we have that $\frac{\gamma_{1}}{\gamma_{3}-\gamma_{2}}>1$ and 
\[
q_{3}\leq\frac{\gamma_{3}}{\eta}<\frac{1}{\eta}\frac{\gamma_{1}\gamma_{3}}{\gamma_{3}-\gamma_{2}}=\frac{1}{\eta}\frac{\gamma_{1}\frac{\gamma_{3}\gamma_{1}}{\gamma_{2}}}{\frac{\gamma_{3}\gamma_{1}}{\gamma_{2}}-\gamma_{1}}\leq\frac{q_{1}q_{2}}{q_{2}-q_{1}}.
\]
Using (\ref{eq:dfa3}), we deduce that for $y\in\left[q_{3},\frac{\gamma_{3}}{\eta}\right]$,
the function $y\rightarrow f(\eta q_{1},\eta q_{2},\eta y)$ is non
decreasing and for $\tilde{q}_{3}=\frac{\gamma_{3}}{\eta}$, we obtain
\[
f(\eta q_{1},\eta q_{2},\eta\tilde{q}_{3})\geq f(\eta q_{1},\eta q_{2},\eta q_{3})>\gamma_{3}\eta q_{1}=\gamma_{1}\eta\tilde{q}_{3}\geq\gamma_{2}\eta q_{2},
\]
which is equivalent to (\ref{eq:lastsituation}) and we conclude in
the same manner as for (\ref{eq:gamma123_1}) in Case 1.
\end{proof}

\section{\label{sec:Additional-proofs-of-Section4}Additional proofs of Section
4}

\subsection{Proof of Lemma \ref{lem:LVpositivecontinuous}}
\begin{proof}
Let $\nu\in\mathcal{P}(\mathbb{R})$ and let $u$ be a solution to
$LV(\nu)$. As $\tilde{\sigma}_{Dup}\in L^{\infty}([0,T],W^{1,\infty}(\mathbb{R}))$
and $u\in L_{loc}^{2}((0,T];H^{1}(\mathbb{R}))$, we have that $dt$-a.e.
on $(0,T]$, $\left(\tilde{\sigma}_{Dup}^{2}u\right)(t,\cdot)\in H^{1}(\mathbb{R})$
and $\tilde{\sigma}_{Dup}^{2}(t,\cdot)\partial_{x}u(t,\cdot)=\partial_{x}\left(\tilde{\sigma}_{Dup}^{2}u\right)(t,\cdot)-\left(u\partial_{x}\tilde{\sigma}_{Dup}^{2}\right)(t,\cdot)$
in the sense of distributions on $\mathbb{R}$. Then for any function
$\phi$ defined for $(t,x)\in[0,T]\times\mathbb{R}$ by $\phi(t,x):=g_{1}(t)g_{2}(x)$,
with $g_{1}\in C_{c}^{\infty}((0,T))$ and $g_{2}\in C_{c}^{\infty}(\mathbb{R})$
, the Borel measure $dm:=udxdt$ satisfies the equality:
\[
\int_{(0,T)\times\mathbb{R}}\left[\partial_{t}\phi+\frac{1}{2}\tilde{\sigma}_{Dup}^{2}\partial_{xx}^{2}\phi+\partial_{x}\left(\frac{1}{2}\tilde{\sigma}_{Dup}^{2}\right)\partial_{x}\phi+\left(r-\frac{1}{2}\tilde{\sigma}_{Dup}^{2}-\partial_{x}\left(\frac{1}{2}\tilde{\sigma}_{Dup}^{2}\right)\right)\partial_{x}\phi\right]dm=0.
\]
By density of the space spanned by the functions of type $g_{1}(t)g_{2}(x)$
in $C_{c}^{\infty}((0,T)\times\mathbb{R})$ for the norm $\phi\in C_{c}^{\infty}((0,T)\times\mathbb{R})\rightarrow||\phi||_{\infty}+||\partial_{x}\phi||_{\infty}+||\partial_{xx}^{2}\phi||_{\infty}+||\partial_{t}\phi||_{\infty}$,
the previous equality is also satisfied for any function $\phi\in C_{c}^{\infty}((0,T)\times\mathbb{R})$.
The variational formulation of the PDE (\ref{eq:DupireFP}) as defined
in \cite[equality 1.5]{Rockner} is then satisfied. 

We recall that $h_{1}(x)=\frac{1}{\sqrt{2\pi}}\exp\left(-\frac{x^{2}}{2}\right)$
and it is easy to check that $h_{1}\in H^{1}(\mathbb{R})$. As $(u(t),h_{1})_{1}\underset{t\rightarrow0}{\rightarrow}\int h_{1}d\nu>0$
and as $u$ is non negative, we obtain that for any $\tau\in(0,T)$,
$\int_{0}^{\tau}(u(t),h_{1})_{1}dt>0$, so $\underset{(0,\tau)\times\mathbb{R}}{ess\sup}\ u(t,x)>0$.
In addition, with Assumption $(B)$, the functions $\frac{1}{2}\tilde{\sigma}_{Dup}^{2}$,
$\partial_{x}\left(\frac{1}{2}\tilde{\sigma}_{Dup}^{2}\right)$ and
$\left(r-\frac{1}{2}\tilde{\sigma}_{Dup}^{2}-\partial_{x}\left(\frac{1}{2}\tilde{\sigma}_{Dup}^{2}\right)\right)$
are uniformly bounded, and $\underline{\sigma}^{2}\leq\tilde{\sigma}_{Dup}^{2}$
a.e. on $[0,T]\times\mathbb{R}$, so by \cite[Corollary 3.1]{Rockner}
we obtain that $u$ is continuous and positive on $(0,T]\times\mathbb{R}$.
\end{proof}

\subsection{Proof of Proposition \ref{prop:uniquenessandaronson}}

Let us first remark that if the initial condition $\nu$ has a density
$u_{0}\in L^{2}(\mathbb{R})$, then by energy estimates, uniqueness
holds without Assumption $(H)$ for a slightly stronger variational
formulation called $LV^{*}(u_{0})$ which is

\[
u\in L^{2}([0,T];H^{1}(\mathbb{R}))\cap L^{\infty}([0,T];L^{2}(\mathbb{R})),u\geq0
\]
\begin{eqnarray*}
\forall v\in H^{1}(\mathbb{R}),\ 0 & = & \frac{d}{dt}(v,u)_{1}-\left(\partial_{x}v,\left(r-\frac{1}{2}\tilde{\sigma}_{Dup}^{2}-\partial_{x}\left(\frac{1}{2}\tilde{\sigma}_{Dup}^{2}\right)\right)u\right)_{1}+\frac{1}{2}\left(\partial_{x}v,\tilde{\sigma}_{Dup}^{2}\partial_{x}u\right)_{1},\\
\\
 &  & \text{in the sense of distributions on \ensuremath{(0,T)}, and \ensuremath{u(0,\cdot)=u_{0}}. }
\end{eqnarray*}
We now prove Proposition \ref{prop:uniquenessandaronson}.
\begin{proof}
The main ingredient to obtain uniqueness to $LV(\nu)$ is \cite[Proposition 4.2]{Figalli}.
In the proof of Theorem \ref{thm:VFin}, we obtained existence of
a solution to $LV(\nu)$. Moreover, if $u$ is a solution to $LV(\nu)$,
then with the same arguments as in the proof of Lemma \ref{lem:LVpositivecontinuous},
we show that the Borel measure $udtdx$ solves the PDE (\ref{eq:DupireFP})
with initial condition $\nu$ in the sense of distributions, which
means that for any $\phi\in C_{c}^{\infty}(\mathbb{R})$, the equality
\begin{eqnarray}
\frac{d}{dt}\int_{\mathbb{R}}\phi(x)u(t,x)dx & = & \int_{\mathbb{R}}\left[\left(r-\frac{1}{2}\tilde{\sigma}_{Dup}^{2}(t,x)\right)\phi'(x)+\frac{1}{2}\tilde{\sigma}_{Dup}^{2}(t,x)\phi''(x)\right]u(t,x)dx,\label{eq:solvesmeasure1}
\end{eqnarray}
holds in the sense of distributions on $(0,T)$, and $u$ converges
to $\nu$ in duality with $C_{c}^{\infty}(\mathbb{R})$ as $t\rightarrow0$.
Under Assumptions $(B)$ and $(H)$, by \cite[Proposition 4.2]{Figalli},
the measure $udtdx$ is the unique solution to the PDE (\ref{eq:DupireFP})
with initial condition $\nu$ in the sense of distributions and therefore
$u$ is the unique solution to $LV(\nu)$.

It is then sufficient to exhibit a solution $udtdx$ to the PDE (\ref{eq:DupireFP})
in the distributional sense with initial condition $\nu$ and such
that $\int u^{2}(t,x)dx\leq\frac{\zeta}{\sqrt{t}}$ for a.e. $t\in(0,T)$,
where $\zeta>0$ is a constant that does not depend on $\nu$. Under
Assumptions $(B)$ and $(H)$, the martingale problem associated to
the SDE (\ref{eq:SDEFinDupireLogPrice2}) is well posed by \cite[Theorems 6.1.7 and 7.2.1]{StroockV},
and as mentioned in \cite[Paragraph 4.1]{menozzi2010}, there exists
two constants $c:=c(\underline{\sigma},\chi)$ and $C:=C(T,\underbar{\ensuremath{\sigma}},H_{0},\chi)$
such that for $y\in\mathbb{R}$, the solution $\left(X_{t}^{y}\right)_{t\geq0}$
to the SDE (\ref{eq:SDEFinDupireLogPrice2}) with initial distribution
$\delta_{y}$ has the density $p^{y}(t,x)$ that satisfies $\forall(t,x)\in(0,T]\times\mathbb{R},\ p^{y}(t,x)\leq Cu_{c}(t,x,y)$,
with $u_{c}(t,x,y):=\sqrt{\frac{c}{2\pi t}}\exp\left(-c\frac{\left(x-y\right)^{2}}{2t}\right)$
for $t\in(0,T],x,y\in\mathbb{R}$, and the function $y\rightarrow p^{y}(t,x)dx$
is measurable. For $\phi\in C_{c}^{\infty}(\mathbb{R})$, by Itô's
Lemma, 
\begin{eqnarray*}
\mathbb{E}\left[\phi(X_{t}^{y})\right] & = & \phi(y)+\int_{0}^{t}\mathbb{E}\left[\phi'(X_{s}^{y})\left(r-\frac{1}{2}\tilde{\sigma}_{Dup}^{2}(s,X_{s}^{y})\right)+\frac{1}{2}\phi''(X_{s}^{y})\tilde{\sigma}_{Dup}^{2}(s,X_{s}^{y})\right]ds.
\end{eqnarray*}
In the sense of distributions, we have 
\[
\frac{d}{dt}\int_{\mathbb{R}}\phi(x)p^{y}(t,x)dx=\int_{\mathbb{R}}\left[\left(r-\frac{1}{2}\tilde{\sigma}_{Dup}^{2}(t,x)\right)\phi'(x)+\frac{1}{2}\tilde{\sigma}_{Dup}^{2}(t,x)\phi''(x)\right]p^{y}(t,x)dx,
\]
and $\underset{t\rightarrow0}{\lim}\int_{\mathbb{R}}\phi(x)p^{y}(t,x)dx=\phi(y)$.
We set $u:=\int_{\mathbb{R}}p^{y}\nu(dy)$ and we check that $udtdx$
solves the PDE (\ref{eq:DupireFP}) with initial condition $\nu$
in the distributional sense. Indeed, for $\psi\in C_{c}^{\infty}((0,T),\mathbb{R})$,
we check, using Fubini's theorem, that

\begin{eqnarray*}
-\int_{0}^{T}\psi'(t)\int_{\mathbb{R}}\phi(x)u(t,x)dxdt & = & \int_{0}^{T}\psi(t)\int_{\mathbb{R}}\left[\left(r-\frac{1}{2}\tilde{\sigma}_{Dup}^{2}(t,x)\right)\phi'(x)+\frac{1}{2}\tilde{\sigma}_{Dup}^{2}(t,x)\phi''(x)\right]u(t,x)dxdt,
\end{eqnarray*}
and using Lebesgue's theorem, that $\underset{t\rightarrow0}{\lim}\int_{\mathbb{R}^{2}}\phi(x)p^{y}(t,x)dx\nu(dy)=\int_{\mathbb{R}}\left(\underset{t\rightarrow0}{\lim}\int_{\mathbb{R}}\phi(x)p^{y}(t,x)dx\right)\nu(dy)=\int_{\mathbb{R}}\phi(y)\nu(dy)$,
and we conclude that $u$ is the unique solution to $LV(\nu)$ and
that moreover $u$ coincides with the time marginals of the solution
to the SDE (\ref{eq:SDEFinDupireLogPrice2}) with initial distribution
$\nu$. Finally we define $\zeta:=\frac{C^{2}\sqrt{c}}{2\sqrt{\pi}}$
and for a.e. $t\in(0,T]$, by Jensen's inequality, $\int_{\mathbb{R}}u^{2}(t,x)dx\leq\int_{\mathbb{R}^{2}}\left(p^{y}(t,x)\right)^{2}\nu(dy)dx\leq\frac{cC^{2}}{2\pi t}\int_{\mathbb{R}}\exp\left(-c\frac{x^{2}}{t}\right)dx\leq\frac{\zeta}{\sqrt{t}}$.
\end{proof}

\section{\label{sec:FigGen}Proof of Theorem \ref{thm:FigalliGeneralization}}

Let us define $S:=\mathbb{R}\times\mathcal{Y}$, $\overline{b}:=\underset{i\in\{1,...,d\}}{\max}||b_{i}||_{\infty},\ \overline{a}:=\underset{i\in\{1,...,d\}}{\max}||a_{i}||_{\infty}$
and $\overline{q}:=\underset{i,j\in\{1,...,d\}}{\max}||q_{ij}||_{\infty}$.
For $(x,y)\in\mathcal{S}$ we denote by $\delta_{x,y}$ the Dirac
distribution on $\{(x,y)\}$. Lemma \ref{lem:LemmaMPPDE} below is
a consequence of the constant expectation of a martingale combined
with Fubini's theorem. 
\begin{lem}
\label{lem:LemmaMPPDE}Let $\{\nu_{x,y}\}_{(x,y)\in\mathcal{S}}$
be a measurable family of probability measures on $E$ such that for
$\mu_{0}-a.e.$ $(x,y)$, $\nu_{x,y}$ is a martingale solution to
the $SDE$ (\ref{eq:SDEJumpFig}) with initial distribution $\delta_{x,y}$.
Define the measure $\mu_{i}(t,\cdot)$, for $1\leq i\leq d$ and $t\in(0,T]$
by 
\[
\int_{\mathbb{R}}\psi(x)\mu_{i}(t,dx):=\int_{E\times\mathcal{S}}\psi(X_{t})1_{\{Y_{t}=i\}}d\nu_{x,y}(X,Y)\mu_{0}(dx,dy)
\]
for any function $\psi\in C_{b}^{2}(\mathbb{R})$. Then $(\mu_{1}(t,\cdot),...,\mu_{d,t}(t,\cdot))_{t\in(0,T]}$
is a solution to the PDS (\ref{eq:PDSJump1})-(\ref{eq:PDSJump2}).
\end{lem}
We now prove Theorem \ref{thm:FigalliGeneralization}.

Step 1: We first establish the result for a regularized version of
the PDS (\ref{eq:PDSJump1})-(\ref{eq:PDSJump2}). Let $\rho_{X}$
and $\rho_{T}$ be convolution kernels defined by $\rho_{X}(x)=\rho_{T}(x)=C_{0}e^{-\sqrt{1+x^{2}}}$,
for $x\in\mathbb{R}$, and where $C_{0}=\left(\int_{\mathbb{R}^{2}}e^{-\sqrt{1+x^{2}}}dx\right)^{-1}$.
For $\epsilon>0$ and $\left(t,x\right)\in\mathbb{R}^{2}$, we define
the functions $\rho_{X}^{\epsilon}(x):=\frac{1}{\epsilon}\rho_{X}(\frac{x}{\epsilon})$,
$\rho_{T}^{\epsilon}:=\frac{1}{\epsilon}\rho_{T}(\frac{t}{\epsilon})$
and $\rho^{\epsilon}(t,x):=\rho_{T}^{\epsilon}(t)\rho_{X}^{\epsilon}(x)$.
For $1\leq i\leq d$, we extend the definition of $t\rightarrow\mu_{i}(t,\cdot)$
to $\mathbb{R}$, by setting for $t\leq0$, $\mu_{i}(t,\cdot)=\mu_{0}(\cdot,\{i\})$
and for $t>T$, $\mu_{i}(t,\cdot)=\mu_{i}(T,\cdot)$. As a consequence,
given $\psi\in C_{b}^{2}(\mathbb{R})$ and $1\leq i\leq d$, the extended
function $t\rightarrow\int_{\mathbb{R}}\psi(x)\mu_{i}(t,dx)$ is now
continuous and bounded on $\mathbb{R}$. For $1\leq i\leq d$, let
us also extend the definitions of the functions $\left(a_{i}\right)_{1\leq i\leq d}$
on $\mathbb{R}^{2}$ by setting $a_{i}(t,\cdot)=0$ if $t\notin[0,T]$.
In the same way, we extend the functions $\left(b_{i}\right)_{1\leq i\leq d}$
and $\left(q_{ij}\right)_{1\leq i,j\leq d}$. Then the family $(\mu_{1}(t,\cdot),...,\mu_{d,t}(t,\cdot))_{t\in\mathbb{R}}$
satisfies the equality (\ref{eq:distributionsolution}) in the distributional
sense on $\mathbb{R}$. We define for $1\leq i\leq d$ and $\left(t,x\right)\in\mathbb{R}^{2}$,
\[
\mu_{i}^{\epsilon}(t,x):=\mu_{i}*\rho^{\epsilon}(t,x)=\int_{\mathbb{R}^{2}}\rho^{\epsilon}(t-s,x-y)\mu_{i}(s,dy)ds.
\]
We define $a_{i}^{\epsilon}=\frac{(a_{i}\mu_{i})*\rho^{\epsilon}}{\mu_{i}^{\epsilon}},\ b_{i}^{\epsilon}=\frac{(b_{i}\mu_{i})*\rho^{\epsilon}}{\mu_{i}^{\epsilon}},$
$q_{ij}^{\epsilon}=\frac{(q_{ij}\mu_{i})*\rho^{\epsilon}}{\mu_{i}^{\epsilon}}$
for $1\leq i,j\leq d$. The family $\left(\mu_{1}^{\epsilon}(t,\cdot),...,\mu_{d}^{\epsilon}(t,\cdot)\right)_{t\in\mathbb{R}}$
is a smooth solution of the following PDS, denoted by $(PDS)_{\epsilon}$,
where for $1\leq i\leq d$,

\begin{eqnarray}
\partial_{t}\mu_{i}^{\epsilon}+\partial_{x}(b_{i}^{\epsilon}\mu_{i}^{\epsilon})-\frac{1}{2}\partial_{xx}^{2}(a_{i}^{\epsilon}\mu_{i}^{\epsilon})-\sum_{j=1}^{d}q_{ji}^{\epsilon}\mu_{i}^{\epsilon} & = & 0\nonumber \\
\mu_{i}^{\epsilon}(0) & = & \mu_{i}*\rho^{\epsilon}(0,\cdot).\label{eq:PDSepsilonCI}
\end{eqnarray}
Without loss of generality, we suppose that for $1\leq i\leq d$,
$\mu_{i}^{\epsilon}$ has a positive density on $\mathbb{R}^{2}$.
If not, then $\mu_{i}(t,\mathbb{R})$ is equal to zero for all $t\in(0,T]$.
In that case, it is sufficient to consider the PDS (\ref{eq:PDSJump1})-(\ref{eq:PDSJump2})
without the state $i$. Under this assumption, the functions $a_{i}^{\epsilon},b_{i}^{\epsilon},q_{ij}^{\epsilon}$
for $1\leq i,j\leq d$ are well defined and for $i,j\in\{1,...,d\}$,
$||a_{i}^{\epsilon}||_{\infty}\leq||a_{i}||_{\infty},$ $||b_{i}^{\epsilon}||_{\infty}\leq||b_{i}||_{\infty}$,
$||q_{ji}^{\epsilon}||_{\infty}\leq||q_{ji}||_{\infty}$. It is easy
to check that for $x\in\mathbb{R}$, and $k\geq1$, there exists constants
$C_{X,k}>0$ s.t. $\left|\frac{\partial^{k}\rho_{X}}{\partial x^{k}}(x)\right|\leq C_{X,k}|\rho_{X}(x)|$,
so $a_{i}^{\epsilon},b_{i}^{\epsilon},q_{ij}^{\epsilon}$ are continuous
and bounded on $\mathbb{R}^{2}$, as well as their derivatives. Let
us denote by $(SDE)_{\epsilon}$, the SDE 
\[
dX_{t}^{\epsilon}=b_{Y_{t}^{\epsilon}}^{\epsilon}(t,X_{t}^{\epsilon})dt+\sqrt{a_{Y_{t}^{\epsilon}}^{\epsilon}(t,X_{t}^{\epsilon})}dW_{t},
\]
where $Y_{t}^{\epsilon}$ is a stochastic process with values in $\mathcal{Y}$,
and that satisfies 
\[
\text{\ensuremath{\mathbb{P}}}\left(Y_{t+dt}^{\epsilon}=j|\left(X_{s}^{\epsilon},Y_{s}^{\epsilon}\right)_{0\leq s\leq t}\right)=q_{Y_{t}^{\epsilon}j}^{\epsilon}(t,X_{t}^{\epsilon})dt,
\]
for $j\neq Y_{t}^{\epsilon}$. As the functions $\left(a_{i}^{\epsilon}\right)_{1\leq i\leq d},\left(b_{i}^{\epsilon}\right)_{1\leq i\leq d},\left(q_{ij}^{\epsilon}\right)_{1\leq i,j\leq d}$
are continuous, Lipschitz and bounded, by \cite[Theorem 5.3]{Menaldi}
and the Kunita-Watanabe theorem, for any $(x,y)\in\mathcal{S}$ there
exists a unique martingale solution $\nu_{x,y}$ to $(SDE)_{\epsilon}$
with initial distribution $\delta_{(x,y)}$ and by \cite[Proposition 5.52]{Menaldi},
the function $(x,y)\rightarrow\nu_{x,y}^{\epsilon}$ is measurable.
We define $\nu^{\epsilon}:=\sum_{i=1}^{d}\int_{\mathbb{R}}\nu_{x,i}^{\epsilon}\left(\mu_{i}*\rho^{\epsilon}\right)(0,x)dx$.
For $t\in(0,T]$, $1\leq i\leq d$, we define the measure $\tilde{\mu}_{i}^{\epsilon}(t,\cdot)$
by 
\[
\int_{\mathbb{R}}\psi(x)\tilde{\mu}_{i}^{\epsilon}(t,dx)=\int_{E}\psi(X_{t})1_{\{Y_{t}=i\}}d\nu^{\epsilon}(X,Y),
\]
for $\psi\in C_{b}^{2}(\mathbb{R})$. By Lemma \ref{lem:LemmaMPPDE},
$(\tilde{\mu}_{i}^{\epsilon})_{1\leq i\leq d}$ solves $(PDS)_{\epsilon}$
with initial condition $(\mu_{i}*\rho^{\epsilon}(0,\cdot))_{1\leq i\leq d}$.
Since $(PDS)_{\epsilon}$ has a unique solution by Proposition \ref{prop:RegularizedUniqueness}
below, we obtain that for $t\in(0,T]$, $1\leq i\leq d$, $\tilde{\mu}_{i}^{\epsilon}(t,\cdot)=\mu_{i}^{\epsilon}(t,\cdot)$,
and for $\psi\in C_{b}^{2}(\mathbb{R})$, 
\begin{equation}
\int_{\mathbb{R}}\psi(x)\mu_{i}^{\epsilon}(t,dx)=\int_{E}\psi(X_{t})1_{\{Y_{t}=i\}}d\nu^{\epsilon}(X,Y).\label{eq:egaliteMUepsilonNUepsilon}
\end{equation}
Step 2: Let $\left(\epsilon_{n}\right)_{n\geq0}$ be a positive sequence
decreasing to 0, we check the family of measures $\left(\nu^{\epsilon_{n}}\right)_{n\geq0}$
has a converging subsequence. For $n\geq0$, as we might not have
$\sum_{i=1}^{d}\int_{\mathbb{R}}\mu_{i}*\rho^{\epsilon}(0,x)dx=1$,
we define $\overline{\nu}^{\epsilon_{n}}:=\frac{\nu^{\epsilon_{n}}}{\sum_{i=1}^{d}\int_{\mathbb{R}}\mu_{i}*\rho^{\epsilon_{n}}(0,x)dx}$,
so that $\overline{\nu}^{\epsilon_{n}}$ is a probability measure.
We use Aldous' criterion to show the tightness of the family $\left(\overline{\nu}^{\epsilon_{n}}\right)_{n\geq0}$.
Let us denote by $\left(X^{\epsilon_{n}},Y^{\epsilon_{n}}\right)$
the solution to $(SDE)_{\epsilon_{n}}$ where the initial law satisfies,
for a non negative and measurable function $\psi$ on $\mathbb{R}$:
\[
\mathbb{E}\left[\psi(X_{0}^{\epsilon_{n}})1_{\{Y_{0}^{\epsilon_{n}}=i\}}\right]=\int_{\mathbb{R}}\psi(x)\frac{\mu_{i}*\rho^{\epsilon_{n}}(0,x)}{\sum_{i=1}^{d}\int_{\mathbb{R}}\mu_{i}*\rho^{\epsilon_{n}}(0,y)dy}dx,
\]
for $1\leq i\leq d$. The process $\left(X^{\epsilon_{n}},Y^{\epsilon_{n}}\right)$
has the law $\overline{\nu}^{\epsilon_{n}}$. First we check that
for any $\eta>0$, there exists a constant $K_{\eta}>0$ such that
$\forall n\geq0,\mathbb{P}(\sup_{t\in[0,T]}|X_{t}^{\epsilon_{n}}|+|Y_{t}^{\epsilon_{n}}|>K_{\eta})\leq\eta$.
For $K>0$, $n\geq0$, 
\begin{eqnarray*}
\mathbb{P}\left(\sup_{t\in[0,T]}\left(|X_{t}^{\epsilon_{n}}|+|Y_{t}^{\epsilon_{n}}|\right)>K\right) & \leq & \mathbb{P}\left(|X_{0}^{\epsilon_{n}}|+\sup_{t\in[0,T]}|X_{t}^{\epsilon_{n}}-X_{0}^{\epsilon_{n}}|+\sup_{t\in[0,T]}|Y_{t}^{\epsilon_{n}}|>K\right)\\
 & \leq & \mathbb{P}\left(\sup_{t\in[0,T]}|Y_{t}^{\epsilon_{n}}|>K/3\right)+\mathbb{P}\left(\sup_{t\in[0,T]}|X_{t}^{\epsilon_{n}}-X_{0}^{\epsilon_{n}}|>K/3\right)+\mathbb{P}\left(|X_{0}^{\epsilon_{n}}|>K/3\right).
\end{eqnarray*}
In the second line, for the first term of the r.h.s., since $Y^{\epsilon_{n}}$
only takes values in $\{1,...,d\}$, we have that $\forall n\geq0,\forall K>3d$,
$\mathbb{P}(\sup_{t\in[0,T]}|Y_{t}^{\epsilon_{n}}|\geq K/3)=0$. For
the second term, by Markov's inequality and Lemma \ref{lem:Moments}
below, we have that for $n\geq0$, $\mathbb{P}\left(\sup_{t\in[0,T]}|X_{t}^{\epsilon_{n}}-X_{0}^{\epsilon_{n}}|\geq K/3\right)\leq\frac{3}{K}\mathbb{E}\left(\sup_{t\in[0,T]}|X_{t}^{\epsilon_{n}}-X_{0}^{\epsilon_{n}}|\right)\leq\frac{3}{K}\left(T\overline{b}+2\sqrt{T}\overline{a}^{1/2}\right)$.
For $K\geq K_{1}:=\frac{6}{\eta}\left(T\overline{b}+2\sqrt{T}\overline{a}^{1/2}\right)$,
$\mathbb{P}\left(\sup_{t\in[0,T]}|X_{t}^{\epsilon_{n}}-X_{0}^{\epsilon_{n}}|\geq K/3\right)\leq\frac{\eta}{2}$.
To study the last term, let us prove that for $t\in\mathbb{R}$, $1\leq i\leq d$
and $\phi\in C_{b}^{2}(\mathbb{R})$, 
\[
\int_{\mathbb{R}}\phi(x)\mu_{i}^{\epsilon_{n}}(t,dx)\underset{n\rightarrow\infty}{\rightarrow}\int_{\mathbb{R}}\phi(x)\mu_{i}(t,dx).
\]
We study the difference: 
\begin{eqnarray*}
\left|\int_{\mathbb{R}}\phi(x)\mu_{i}^{\epsilon_{n}}(t,dx)-\int_{\mathbb{R}}\phi(x)\mu_{i}(t,dx)\right| & = & \left|\int_{\mathbb{R}^{2}}\rho_{T}^{\epsilon_{n}}(t-s)\left(\left(\phi*\rho_{X}^{\epsilon_{n}}(x)\right)\mu_{i}(s,dx)-\phi(x)\mu_{i}(t,dx)\right)ds\right|\\
 & \leq & \left|\int_{\mathbb{R}^{2}}\rho_{T}^{\epsilon_{n}}(t-s)\phi(x)\left(\mu_{i}(s,dx)-\mu_{i}(t,dx)\right)ds\right|\\
 & \ + & \left|\int_{\mathbb{R}^{2}}\rho_{T}^{\epsilon_{n}}(t-s)\left(\phi*\rho_{X}^{\epsilon_{n}}(x)-\phi(x)\right)\mu_{i}(s,dx)ds\right|.
\end{eqnarray*}
The first term of the r.h.s converges to $0$ as $n\rightarrow\infty$
as the function $t\rightarrow\int\phi(x)\mu_{i}(t,dx)$ is continuous
and bounded. For the last term, as $\phi\in C_{b}^{2}(\mathbb{R})$,
$\phi$ is globally lipschitz and we observe that for $x\in\mathbb{R}$,
\begin{eqnarray*}
\left|\phi*\rho_{X}^{\epsilon_{n}}(x)-\phi(x)\right| & \leq & \int_{\mathbb{R}}|\phi(y)-\phi(x)|\rho_{X}^{\epsilon_{n}}(x-y)dy\leq||\phi'||_{\infty}\int_{\mathbb{R}}|y-x|\rho_{X}^{\epsilon_{n}}(x-y)dy\\
 & \leq & \epsilon_{n}C_{0}||\phi'||_{\infty}\left(\int_{\mathbb{R}}|y|\exp\left(-\sqrt{1+y^{2}}\right)dy\right),
\end{eqnarray*}
so $\phi*\rho_{X}^{\epsilon_{n}}$ converges uniformly to $\phi$
and as $\mu(t,\mathbb{R})\leq B$ for $t\in\mathbb{R}$, we conclude
that for $t\in[0,T]$ and $i\in\{1,...,d\}$, 
\begin{equation}
\left|\int_{\mathbb{R}}\phi(x)\mu_{i}^{\epsilon_{n}}(t,dx)-\int_{\mathbb{R}}\phi(x)\mu_{i}(t,dx)\right|\underset{n\rightarrow\infty}{\rightarrow}0.\label{eq:convergencemeasuresMuEpsMu}
\end{equation}
In particular, as the random variable $X_{0}^{\epsilon_{n}}$ has
the density $\frac{\sum_{i=1}^{d}\mu_{i}*\rho^{\epsilon_{n}}(0,x)}{\sum_{i=1}^{d}\int_{\mathbb{R}}\mu_{i}*\rho^{\epsilon_{n}}(0,x)dx}dx$,
the convergence (\ref{eq:convergencemeasuresMuEpsMu}) implies that
for $\psi\in C_{b}^{2}(\mathbb{R})$, 
\[
\mathbb{E}\left[\psi\left(X_{0}^{\epsilon_{n}}\right)\right]\rightarrow\sum_{i=1}^{d}\int_{\mathbb{R}}\psi(x)\mu_{0}(dx,\{i\}),
\]
so there exists $K_{2}>0$ such that $\mathbb{P}\left(|X_{0}^{\epsilon_{n}}|\geq K_{2}/3\right)\leq\frac{\eta}{2}.$
As a consequence, for $K_{\eta}:=\max\left(3d,K_{1},K_{2}\right)$
and $\forall n\geq0$, 
\[
\mathbb{P}\left(\sup_{t\in[0,T]}\left(|X_{t}^{\epsilon_{n}}|+|Y_{t}^{\epsilon_{n}}|\right)>K_{\eta}\right)\leq\eta.
\]
We now check that for $\zeta>0,K>0$, there exists $\delta_{\zeta,K}>0$
such that $\forall n\geq0$,
\[
\underset{\underset{0\leq\delta\leq\delta_{\zeta,K}}{\tau\in\mathcal{T}_{[0,T]}^{f}}}{\sup}\mathbb{P}\left(\max\left(|X_{\tau+\delta}^{\epsilon_{n}}-X_{\tau}^{\epsilon_{n}}|,|Y_{\tau+\delta}^{\epsilon_{n}}-Y_{\tau}^{\epsilon_{n}}|\right)>K\right)<\zeta,
\]
where $\mathcal{T}_{[0,T]}^{f}$ denotes the set of stopping times
taking values in a finite subset of $[0,T]$. For $K>\text{0}$, $\delta\geq0$,
$n\geq0$ and $\tau\in\mathcal{T}_{[0,T]}^{f}$, we have the following
inequality,
\[
\mathbb{P}\left(\max\left(|X_{\tau+\delta}^{\epsilon_{n}}-X_{\tau}^{\epsilon_{n}}|,|Y_{\tau+\delta}^{\epsilon_{n}}-Y_{\tau}^{\epsilon_{n}}|\right)>K\right)\leq\mathbb{P}\left(|X_{\tau+\delta}^{\epsilon_{n}}-X_{\tau}^{\epsilon_{n}}|>K\right)+\mathbb{P}\left(|Y_{\tau+\delta}^{\epsilon_{n}}-Y_{\tau}^{\epsilon_{n}}|>K\right)=:P_{1}+P_{2}.
\]
Using Markov's inequality on $P_{1}$, we have
\begin{eqnarray*}
P_{1} & \leq & \frac{1}{K^{2}}\mathbb{E}\left[\left|\int_{\tau}^{\tau+\delta}b_{Y_{s}^{\epsilon_{n}}}^{\epsilon_{n}}(s,X_{s}^{\epsilon_{n}})ds+\int_{\tau}^{\tau+\delta}\sqrt{a_{Y_{s}^{\epsilon_{n}}}^{\epsilon_{n}}(s,X_{s}^{\epsilon_{n}})}dW_{s}\right|^{2}\right]\\
 & \leq & \frac{2}{K^{2}}\mathbb{E}\left[\left|\int_{\tau}^{\tau+\delta}b_{Y_{s}^{\epsilon_{n}}}^{\epsilon_{n}}(s,X_{s}^{\epsilon_{n}})ds\right|^{2}+\left|\int_{\tau}^{\tau+\delta}\sqrt{a_{Y_{s}^{\epsilon_{n}}}^{\epsilon_{n}}(s,X_{s}^{\epsilon_{n}})}dW_{s}\right|^{2}\right].
\end{eqnarray*}
On the one hand $\mathbb{E}\left[\left|\int_{\tau}^{\tau+\delta}b_{Y_{s}^{\epsilon_{n}}}^{\epsilon_{n}}(s,X_{s}^{\epsilon_{n}})ds\right|^{2}\right]\leq\delta\mathbb{E}\left[\int_{\tau}^{\tau+\delta}\left(b_{Y_{s}^{\epsilon_{n}}}^{\epsilon_{n}}\right)^{2}(s,X_{s}^{\epsilon_{n}})ds\right]\leq\delta^{2}\overline{b}^{2}$
by the Cauchy-Schwarz inequality, and on the other hand $\mathbb{E}\left[\left|\int_{\tau}^{\tau+\delta}\sqrt{a_{Y_{s}^{\epsilon_{n}}}^{\epsilon_{n}}(s,X_{s}^{\epsilon_{n}})}dW_{s}\right|^{2}\right]=\mathbb{E}\left[\left|\int_{0}^{T}1_{\{s\in[\tau,\left(\tau+\delta\right)\land T]\}}\sqrt{a_{Y_{s}^{\epsilon_{n}}}^{\epsilon_{n}}(s,X_{s}^{\epsilon_{n}})}dW_{s}\right|^{2}\right]\leq\delta\overline{a}$
by Ito's isometry. For $P_{2}$ we notice that 
\[
P_{2}\leq1-\mathbb{P}\left(\forall s\in[\tau,\tau+\delta],Y_{\tau}^{\epsilon_{n}}=Y_{\tau+s}^{\epsilon_{n}}\right)\leq1-e^{-\overline{q}\delta}.
\]
We gather the upper bounds on $P_{1}$ and $P_{2}$, and to satisfy
Aldous' criterion, it is sufficient to choose $\delta_{\zeta,K}$
small enough so that 
\[
\frac{2}{K^{2}}\delta_{\zeta,K}\left(\delta_{\zeta,K}\overline{b}^{2}+\overline{a}\right)+1-e^{-\overline{q}\delta_{\zeta,K}}\leq\epsilon.
\]
We have then shown that the family of processes $\left(X^{\epsilon_{n}},Y^{\epsilon_{n}}\right)_{n\geq0}$
is tight so the family of measures $\left(\overline{\nu}^{\epsilon_{n}}\right)_{n\geq0}$
is tight. As $E$ is Polish, by Prohorov's theorem there exists a
measure $\nu$ which is a limit point of $\left(\overline{\nu}^{\epsilon_{n}}\right)_{n\geq0}$.
Let us remark that by (\ref{eq:convergencemeasuresMuEpsMu}), $\sum_{i=1}^{d}\int_{\mathbb{R}}\mu_{i}*\rho^{\epsilon_{n}}(0,x)dx\rightarrow1$,
so $\nu$ is also a limit point of $\left(\nu^{\epsilon_{n}}\right)_{n\geq0}$.
For notational simplicity, we suppose in what follows that $\left(\nu^{\epsilon_{n}}\right)_{n\geq0}$
and $\left(\overline{\nu}^{\epsilon_{n}}\right)_{n\geq0}$ converge
to $\nu$ as $n\rightarrow\infty$. By \cite[Lemma 7.7 and Theorem 7.8]{EthierKurtz},
$D(\nu):=\{t\in[0,T],\nu\left(\left(X_{t^{-}},Y_{t^{-}}\right)=\left(X_{t},Y_{t}\right)\right)=1\}$
has a complement in $[0,T]$ which is at most countable. Moreover,
for any $k\geq1$, and $t_{1},t_{2},...,t_{k}\in D(\nu)$, the sequence
$\left(\left(X_{t_{1}}^{\epsilon_{n}},Y_{t_{1}}^{\epsilon_{n}}\right),...,\left(X_{t_{k}}^{\epsilon_{n}},Y_{t_{k}}^{\epsilon_{n}}\right)\right)_{n\geq0}$
converges to $\left(\left(X_{t_{1}},Y_{t_{1}}\right),...,\left(X_{t_{k}},Y_{t_{k}}\right)\right)$
in distribution as $n\rightarrow\infty$. Consequently, for $t\in D(\nu)$,
$1\leq i\leq d$, and $\phi\in C_{c}^{\infty}(\mathbb{R})$, 
\[
\int_{E}\phi(X_{t})1_{\{Y_{t}=i\}}d\nu^{\epsilon_{n}}(X,Y)\underset{n\rightarrow\infty}{\rightarrow}\int_{E}\phi(X_{t})1_{\{Y_{t}=i\}}d\nu(X,Y).
\]
Moreover, by (\ref{eq:egaliteMUepsilonNUepsilon}) and (\ref{eq:convergencemeasuresMuEpsMu})
we obtain the following equality, for $t\in D(\nu)$: 
\[
\int_{\mathbb{R}}\phi(x)\mu_{i}(t,dx)=\int_{E}\phi(X_{t})1_{\{Y_{t}=i\}}d\nu(X,Y),
\]
As the function $t\rightarrow\int_{E}\phi(X_{t})1_{\{Y_{t}=i\}}d\nu(X,Y)$
is right-continuous and the function $t\rightarrow\int_{\mathbb{R}}\phi(x)\mu_{i}(t,dx)$
is continuous, the previous equality holds for $t\in(0,T]$.

\[
\]
Step 3: we check that $\nu$ is a martingale solution of the SDE (\ref{eq:SDEJumpFig})
with initial distribution $\mu_{0}$. We define 
\[
L_{t}\phi(x,i):=b_{i}(t,x)\partial_{x}\phi(x,i)+\frac{1}{2}a_{i}(t,x)\partial_{xx}^{2}\phi(x,i)+\sum_{j=1}^{d}q_{ij}(t,x)\phi(x,j),
\]
and for $n\geq0$, 
\[
L_{t}^{\epsilon_{n}}\phi(x,i):=b_{i}^{\epsilon_{n}}(t,x)\partial_{x}\phi+\frac{1}{2}a_{i}^{\epsilon_{n}}(t,x)\partial_{xx}^{2}\phi+\sum_{j=1}^{d}q_{ij}^{\epsilon_{n}}(t,x)\phi(x,j).
\]
Let $s\in[0,T]$, $p\in\mathbb{N}^{*}$, $0\leq s_{1}\leq...\leq s_{p}\leq s$
and let $\psi_{1},...,\psi_{p}$ be bounded and continuous functions
on $\mathcal{S}$, with $||\psi_{i}||_{\infty}\leq1$ for $1\leq i\leq d$.
Since for $n\geq0$, $\overline{\nu}^{\epsilon_{n}}$ is a martingale
solution to $(SDE)_{\epsilon_{n}}$, and $\nu^{\epsilon_{n}}=\left(\sum_{i=1}^{d}\int_{\mathbb{R}}\mu_{i}*\rho^{\epsilon_{n}}(0,x)dx\right)\overline{\nu}^{\epsilon_{n}}$,
we have that for $t\geq s$, 
\[
\int_{E}\left[\phi(X_{t},Y_{t})-\phi(X_{s},Y_{s})-\int_{s}^{t}L_{u}^{\epsilon_{n}}\phi(X_{u},Y_{u})du\right]\prod_{i=1}^{p}\psi_{i}(X_{s_{i}},Y_{s_{i}})d\nu^{\epsilon_{n}}(X,Y)=0.
\]
Let $\tilde{b}_{i}:\mathbb{R}_{+}\times\mathbb{R}\rightarrow\mathbb{R}$,
$\tilde{a}_{i}:\mathbb{R}_{+}\times\mathbb{R}\rightarrow\mathbb{R}$,
$\tilde{q}_{ij}:\mathbb{R}_{+}\times\mathbb{R}\rightarrow\mathbb{R}$
be bounded and continuous functions for $1\leq i,j\leq d$. Let us
also define

\[
\tilde{L}_{t}\phi(x,i):=\tilde{b}_{i}(t,x)\partial_{x}\phi+\frac{1}{2}\tilde{a}_{i}(t,x)\partial_{xx}^{2}\phi+\sum_{j=1}^{d}\tilde{q}_{ij}(t,x)\phi(x,j),
\]
and for $n\geq0$, 
\[
\tilde{L}_{t}^{\epsilon_{n}}\phi(x,i):=\tilde{b}_{i}^{\epsilon_{n}}(t,x)\partial_{x}\phi+\frac{1}{2}\tilde{a}_{i}^{\epsilon_{n}}(t,x)\partial_{xx}^{2}\phi+\sum_{j=1}^{d}\tilde{q}_{ij}^{\epsilon_{n}}(t,x)\phi(x,j),
\]
where $\tilde{b}_{i}^{\epsilon_{n}},\tilde{a}_{i}^{\epsilon_{n}},\tilde{q}_{ij}^{\epsilon_{n}}$
are built analogously to $b_{i}^{\epsilon_{n}},a_{i}^{\epsilon_{n}},q_{ij}^{\epsilon_{n}}$.
Then, recalling that $||\psi_{i}||_{\infty}\leq1$ for $1\leq i\leq d$,
we get
\begin{eqnarray}
 &  & \left|\int_{E}\left[\phi(X_{t},Y_{t})-\phi(X_{s},Y_{s})-\int_{s}^{t}\tilde{L}_{u}^{\epsilon_{n}}\phi(X_{u},Y_{u})du\right]\prod_{i=1}^{p}\psi_{i}(X_{s_{i}},Y_{s_{i}})d\nu^{\epsilon_{n}}(X,Y)\right|\label{eq:inequalitylimit1}\\
 &  & \ \leq\int_{E}\left[\int_{s}^{t}\left|\left(L_{u}^{\epsilon_{n}}-\tilde{L}_{u}^{\epsilon_{n}}\right)\phi(X_{u},Y_{u})\right|du\right]d\nu^{\epsilon_{n}}(X,Y)\nonumber \\
 &  & \ \leq\frac{1}{2}\sum_{i=1}^{d}\int_{s}^{t}\int_{\mathcal{\mathbb{R}}}\left|\left(\frac{(a_{i}\mu_{i})*\rho^{\epsilon_{n}}}{\mu_{i}^{\epsilon_{n}}}-\frac{(\tilde{a}_{i}\mu_{i})*\rho^{\epsilon_{n}}}{\mu_{i}^{\epsilon_{n}}}\right)\partial_{xx}^{2}\phi(x,i)\right|\mu_{i}^{\epsilon_{n}}(u,dx)du\nonumber \\
 &  & \ +\sum_{i=1}^{d}\int_{s}^{t}\int_{\mathcal{\mathbb{R}}}\left|\left(\frac{(b_{i}\mu_{i})*\rho^{\epsilon_{n}}}{\mu_{i}^{\epsilon_{n}}}-\frac{(\tilde{b}_{i}\mu_{i})*\rho^{\epsilon_{n}}}{\mu_{i}^{\epsilon_{n}}}\right)\partial_{x}\phi(x,i)\right|\mu_{i}^{\epsilon_{n}}(u,dx)du\nonumber \\
 &  & \ +\sum_{i,j=1}^{d}\int_{s}^{t}\int_{\mathcal{\mathbb{R}}}\left|\left(\frac{(q_{ij}\mu_{i})*\rho^{\epsilon_{n}}}{\mu_{i}^{\epsilon_{n}}}-\frac{(\tilde{q}_{ij}\mu_{i})*\rho^{\epsilon_{n}}}{\mu_{i}^{\epsilon_{n}}}\right)\phi(x,j)\right|\mu_{i}^{\epsilon_{n}}(u,dx)du\nonumber \\
 &  & \ \leq\frac{1}{2}\sum_{i=1}^{d}\int_{s}^{t}\left(\int_{\mathcal{\mathbb{R}}^{2}}\rho_{T}^{\epsilon_{n}}(u-v)|a_{i}(v,x)-\tilde{a}_{i}(v,x)|\left(\mbox{\ensuremath{\rho}}_{X}^{\epsilon_{n}}*|\partial_{xx}^{2}\phi|(x,i)\right)\mu_{i}(v,dx)dv\right)du\label{eq:inequalitylimit2}\\
 &  & \ +\sum_{i=1}^{d}\int_{s}^{t}\left(\int_{\mathcal{\mathbb{R}}^{2}}\rho_{T}^{\epsilon_{n}}(u-v)|b_{i}(v,x)-\tilde{b}_{i}(v,x)|\left(\mbox{\ensuremath{\rho}}_{X}^{\epsilon_{n}}*|\partial_{x}\phi|(x,i)\right)\mu_{i}(v,dx)dv\right)du\label{eq:inequalitylimit3}\\
 &  & \ +\sum_{i,j=1}^{d}\int_{s}^{t}\left(\int_{\mathcal{\mathbb{R}}^{2}}\rho_{T}^{\epsilon_{n}}(u-v)|q_{ij}(v,x)-\tilde{q}_{ij}(v,x)|\left(\mbox{\ensuremath{\rho}}_{X}^{\epsilon_{n}}*|\phi|(x,i)\right)\mu_{i}(v,dx)dv\right)du\label{eq:inequalitylimit4}
\end{eqnarray}
Our goal now is to let $n\rightarrow\infty$ in the terms (\ref{eq:inequalitylimit1})-(\ref{eq:inequalitylimit4}),
and obtain that for any $s_{1},...,s_{p},s,t\in[0,T]$, such that
$0\leq s_{1}\leq...\leq s_{p}\le s\leq t$, 
\begin{eqnarray}
 &  & \left|\int_{E}\left[\phi(X_{t},Y_{t})-\phi(X_{s},Y_{s})-\int_{s}^{t}\tilde{L}_{u}\phi(X_{u},Y_{u})du\right]\prod_{i=1}^{p}\psi_{i}\left(X_{s_{i}},Y_{s_{i}}\right)d\nu(X,Y)\right|\label{eq:finalinequality1}\\
 &  & \ \leq\frac{1}{2}\sum_{i=1}^{d}\int_{s}^{t}\int_{\mathcal{\mathbb{R}}}|a_{i}(u,x)-\tilde{a}_{i}(u,x)||\partial_{xx}^{2}\phi(x,i)|\mu_{i}(u,dx)du\label{eq:finalinequality2}\\
 &  & \ +\sum_{i=1}^{d}\int_{s}^{t}\int_{\mathcal{\mathbb{R}}}|b_{i}(u,x)-\tilde{b}_{i}(u,x)||\partial_{x}\phi(x,i)|\mu_{i}(u,dx)du\label{eq:finalinequality3}\\
 &  & \ +\sum_{i,j=1}^{d}\int_{s}^{t}\int_{\mathcal{\mathbb{R}}}|q_{ij}(u,x)-\tilde{q}_{ij}(u,x)||\partial_{x}\phi(x,i)|\mu_{i}(u,dx)du.\label{eq:finalinequality4}
\end{eqnarray}
Here we explain how to conclude the proof of Theorem \ref{thm:FigalliGeneralization},
if we suppose that (\ref{eq:finalinequality1})-(\ref{eq:finalinequality4})
hold. Let us notice that the term 
\[
\left|\int_{E}\left[\int_{s}^{t}\left(\tilde{L}_{u}-L_{u}\right)\phi(X_{u},Y_{u})du\right]\prod_{i=1}^{p}\psi_{i}\left(X_{s_{i}},Y_{s_{i}}\right)d\nu(X,Y)\right|,
\]
is also dominated by the sum of the terms in the lines (\ref{eq:finalinequality2})-(\ref{eq:finalinequality4}).
By \cite[Theorem 3.45]{etaconvergenceLaw}, for $1\leq i\leq d$,
we can choose sequences of continuous functions $\left(\tilde{a}_{i}^{k}\right)_{k\geq0},\left(\tilde{b}_{i}^{k}\right)_{k\ge0},$
and $\left(\tilde{q}_{ij}^{k}\right)_{k\geq0}$ for $1\leq j\leq d$
converging respectively to $a_{i},b_{i},q_{ij}$ in $L^{1}([0,T]\times\mathbb{R},\eta_{i})$,
with $\eta_{i}:=\mu_{i}(t,\cdot)dt$, to finally obtain that 
\[
\int_{E}\left[\phi(X_{t},Y_{t})-\phi(X_{s},Y_{s})-\int_{s}^{t}L_{u}\phi(X_{u},Y_{u})du\right]\prod_{i=1}^{p}\psi_{i}\left(X_{s_{i}},Y_{s_{i}}\right)d\nu(X,Y)=0.
\]
This is enough to conclude that $\nu$ is a martingale solution to
the SDE (\ref{eq:SDEJumpFig}) with initial condition $\mu_{0}$ and
end the proof of Theorem \ref{thm:FigalliGeneralization}.

In what follows, we show how to obtain (\ref{eq:finalinequality1})-(\ref{eq:finalinequality4}).
As the function
\[
\left(s_{1},...,s_{k},s,t\right)\rightarrow\int_{E}\left[\phi(X_{t},Y_{t})-\phi(X_{s},Y_{s})-\int_{s}^{t}\tilde{L}_{u}\phi(X_{u},Y_{u})du\right]\prod_{i=1}^{p}\psi_{i}\left(X_{s_{i}},Y_{s_{i}}\right)d\nu(X,Y),
\]
is right-continuous, and as the terms in (\ref{eq:finalinequality2})-(\ref{eq:finalinequality4})
are continuous in the variables $(s,t)$, it is sufficient to show
that the inequality in the lines (\ref{eq:finalinequality1})-(\ref{eq:finalinequality4})
holds if we moreover assume that $s_{1},...,s_{p},s,t\in D(\nu)$.
Let us show that for $s_{1},...,s_{p},s,t\in D(\nu)$,

\begin{eqnarray}
 &  & \int_{E}\left[\phi(X_{t},Y_{t})-\phi(X_{s},Y_{s})-\int_{s}^{t}\tilde{L}_{u}^{\epsilon_{n}}\phi(X_{u},Y_{u})du\right]\prod_{i=1}^{p}\psi_{i}\left(X_{s_{i}},Y_{s_{i}}\right)d\nu^{\epsilon_{n}}(X,Y)\nonumber \\
 &  & \ \rightarrow\int_{E}\left[\phi(X_{t},Y_{t})-\phi(X_{s},Y_{s})-\int_{s}^{t}\tilde{L}_{u}\phi(X_{u},Y_{u})du\right]\prod_{i=1}^{p}\psi_{i}\left(X_{s_{i}},Y_{s_{i}}\right)d\nu(X,Y).\label{eq:CV101}
\end{eqnarray}
By convergence of the finite dimensional distributions of $\left(\overline{\nu}^{\epsilon_{n}}\right)_{n\geq0}$
and the fact that $\nu^{\epsilon_{n}}=$$\left(\sum_{i=1}^{d}\int_{\mathbb{R}}\mu_{i}*\rho^{\epsilon_{n}}(0,x)dx\right)\overline{\nu}^{\epsilon_{n}}$
with $\sum_{i=1}^{d}\int_{\mathbb{R}}\mu_{i}*\rho^{\epsilon_{n}}(0,x)dx\underset{n\rightarrow\infty}{\rightarrow}1$,
the following convergence holds:
\begin{eqnarray*}
 &  & \int_{E}\left[\phi(X_{t},Y_{t})-\phi(X_{s},Y_{s})\right]\prod_{i=1}^{p}\psi_{i}\left(X_{s_{i}},Y_{s_{i}}\right)d\nu^{\epsilon_{n}}(X,Y)\\
 &  & \ \rightarrow\int_{E}\left[\phi(X_{t},Y_{t})-\phi(X_{s},Y_{s})\right]\prod_{i=1}^{p}\psi_{i}\left(X_{s_{i}},Y_{s_{i}}\right)d\nu(X,Y).
\end{eqnarray*}
With the same argument, we have that for $u\in D(\nu)\cap(s,t)$,
\begin{equation}
\int_{E}\tilde{L}_{u}\phi(X_{u},Y_{u})\prod_{i=1}^{p}\psi_{i}\left(X_{s_{i}},Y_{s_{i}}\right)d\nu^{\epsilon_{n}}(X,Y)\rightarrow\int_{E}\tilde{L}_{u}\phi(X_{u},Y_{u})\prod_{i=1}^{p}\psi_{i}\left(X_{s_{i}},Y_{s_{i}}\right)d\nu(X,Y).\label{eq:uparu}
\end{equation}
Then by Fubini's theorem,
\begin{eqnarray}
 &  & \int_{s}^{t}\left|\int_{E}\tilde{L}_{u}^{\epsilon_{n}}\phi(X_{u},Y_{u})\prod_{i=1}^{p}\psi_{i}\left(X_{s_{i}},Y_{s_{i}}\right)d\nu^{\epsilon_{n}}(X,Y)-\int_{E}\tilde{L}_{u}\phi(X_{u},Y_{u})\prod_{i=1}^{p}\psi_{i}\left(X_{s_{i}},Y_{s_{i}}\right)d\nu(X,Y)\right|du\nonumber \\
 &  & \leq\int_{s}^{t}\left|\int_{E}\tilde{L}_{u}\phi(X_{u},Y_{u})\prod_{i=1}^{p}\psi_{i}\left(X_{s_{i}},Y_{s_{i}}\right)d\nu^{\epsilon_{n}}(X,Y)-\int_{E}\tilde{L}_{u}\phi(X_{u},Y_{u})\prod_{i=1}^{p}\psi_{i}\left(X_{s_{i}},Y_{s_{i}}\right)d\nu(X,Y)\right|du\label{eq:ineqpart1}\\
 &  & \ +\int_{s}^{t}\int_{E}|\tilde{L}_{u}^{\epsilon_{n}}\phi(X_{u},Y_{u})-\tilde{L}_{u}\phi(X_{u},Y_{u})|\prod_{i=1}^{p}|\psi_{i}\left(X_{s_{i}},Y_{s_{i}}\right)|d\nu^{\epsilon_{n}}(X,Y)du\label{eq:ineqpart2}
\end{eqnarray}
The term on line (\ref{eq:ineqpart1}) converges to zero as $n\rightarrow\infty$,
by (\ref{eq:uparu}) and Lebesgue's theorem. For the term (\ref{eq:ineqpart2}),
we write 
\[
\int_{s}^{t}\int_{E}|\tilde{L}_{u}^{\epsilon_{n}}\phi(X_{u},Y_{u})-\tilde{L}_{u}\phi(X_{u},Y_{u})|\prod_{i=1}^{p}|\psi_{i}\left(X_{s_{i}},Y_{s_{i}}\right)|d\nu^{\epsilon_{n}}(X,Y)du\leq\sum_{k=1}^{d}\int_{s}^{t}\int_{\mathbb{R}}|\tilde{L}_{u}^{\epsilon_{n}}\phi(x,k)-\tilde{L}_{u}\phi(x,k)|\mu_{k}^{\epsilon_{n}}(u,x)dxdu.
\]
As the terms in $\left(\tilde{L}_{u}^{\epsilon_{n}}\phi(\cdot,k)\right)_{1\leq k\leq d}$
and $\left(\tilde{L}_{u}\phi(\cdot,k)\right)_{1\leq k\leq d}$ are
continuous and bounded, it is sufficient to show that for $1\leq k\leq d$,
$z_{k}:\mathbb{R}^{2}\rightarrow\mathbb{R}$ a continuous and bounded
function and $z_{k}^{\epsilon}:=\frac{\left(z_{k}\mu_{k}\right)*\rho^{\epsilon}}{\mu_{k}^{\epsilon}}$,
for $\epsilon>0$, the following convergence holds:
\[
\int_{s}^{t}\int_{\mathbb{R}}|z_{k}^{\epsilon}(u,x)-z_{k}(u,x)|\mu_{k}^{\epsilon}(u,x)dxdu\underset{\epsilon\rightarrow0}{\rightarrow}0.
\]
The previous term rewrites:
\begin{eqnarray}
 &  & \int_{s}^{t}\int_{\mathbb{R}}|z_{k}^{\epsilon}(u,x)-z_{k}(u,x)|\mu_{k}^{\epsilon}(u,x)dxdu=\int_{s}^{t}\int_{\mathbb{R}}|\left(z_{k}\mu_{k}\right)*\rho^{\epsilon}(u,x)-z_{k}(u,x)\mu_{k}^{\epsilon}(u,x)|dxdu\nonumber \\
 &  & \leq\int_{s}^{t}\int_{\mathbb{R}}\rho_{T}^{\epsilon}(u-v)\left(\int_{\mathbb{R}}\left(\int_{\mathbb{R}}\rho_{X}^{\epsilon}(x-y)|z_{k}(v,y)-z_{k}(u,x)|dx\right)\mu_{k}(v,dy)\right)dvdu.\label{eq:CV}
\end{eqnarray}
Let us notice that $\underset{M\rightarrow\infty}{\lim}\underset{t\in\mathbb{R}}{\sup}\ \mu_{k}(t,\mathbb{R}\backslash(-M,M))=\underset{M\rightarrow\infty}{\lim}\underset{t\in[0,T]}{\sup}\ \mu_{k}(t,\mathbb{R}\backslash[-M,M])=0$.
Indeed, if there exists $\delta>0$, and a sequence $(t_{n})_{n\geq1}$
in the compact $[0,T]$ converging to $t\in[0,T]$ such that $\forall n\geq1,\mu_{k}(t_{n},\mathbb{R}\backslash(-n,n))>\delta$,
then by continuity, $\forall n\geq1,\mu_{k}(t,\mathbb{R}\backslash(-n,n))>\delta$,
which is impossible. 

For $\zeta>0$, let $M>0$ be such that $\underset{t\in\mathbb{R}}{\sup}\ \mu_{k}(t,\mathbb{R}\backslash[-M,M])\leq\zeta$,
and let $\eta\in(0,1)$ such that $|z(v,y)-z(u,x)|\leq\zeta$ if $v,u\in[s-1,t+1]$,
$x,y\in[-M-1,M+1]$, and $\max(|x-y|,|u-v|)\leq\eta$. Then, as 
\[
1\leq1_{\{|x-y|>\eta\}}+1_{\{|u-v|>\eta\}}+1_{\{|x-y|\leq\eta,|u-v|\leq\eta,\left(x,y\right)\in[-M-1,M+1]^{2}\}}+1_{\{|x-y|\leq\eta,|u-v|\leq\eta,\left(x,y\right)\notin[-M-1,M+1]^{2}\}},
\]
we have that $\forall u\in[s,t]$,
\begin{eqnarray*}
 &  & \int_{\mathbb{R}}\rho_{T}^{\epsilon}(u-v)\left(\int_{\mathbb{R}}\left(\int_{\mathbb{R}}\rho_{X}^{\epsilon}(x-y)|z(v,y)-z(u,x)|dx\right)\mu_{i}(v,dy)\right)dv\\
 &  & \ \leq2||z||_{\infty}B\left(\int_{|u-v|>\eta}\rho_{T}^{\epsilon}(u-v)dv+\int_{|x-y|>\eta}\rho_{T}^{\epsilon}(x-y)dx\right)+\zeta B+2||z||_{\infty}\zeta\leq\left(4||z||_{\infty}+B\right)\zeta,
\end{eqnarray*}
for $\epsilon$ small enough. Then we obtain that the term in line
(\ref{eq:CV}) converges to $0$ as $\epsilon\rightarrow0$. We now
analyze the term on the line (\ref{eq:inequalitylimit2}). We define
the function $g:\mathbb{R}\rightarrow\mathbb{R}$ by 
\[
g(v)=\int_{\mathbb{R}}|a_{i}(v,x)-\tilde{a}_{i}(v,x)||\partial_{xx}^{2}\phi(x,i)|\mu_{i}(v,dx),
\]
and for $\epsilon>0$, the function $g^{\epsilon}:\mathbb{R}\rightarrow\mathbb{R}$
by 
\[
g^{\epsilon}(v)=\int_{\mathbb{R}}|a_{i}(v,x)-\tilde{a}_{i}(v,x)|\left(\mbox{\ensuremath{\rho}}_{X}^{\epsilon}*|\partial_{xx}^{2}\phi(x,i)|\right)\mu_{i}(v,dx).
\]
We study, for $c>0$, 
\begin{eqnarray*}
\int_{s}^{t}\left(\int_{\mathcal{\mathbb{R}}}\rho_{T}^{\epsilon}(u-v)g^{\epsilon}(v)dv\right)du & = & \int_{s}^{t}\left(\int_{\mathcal{\mathbb{R}}}1_{\{|u-v|>c\epsilon\}}\rho_{T}^{\epsilon}(u-v)g^{\epsilon}(v)dv\right)du\\
 & \ + & \int_{s}^{t}\left(\int_{\mathcal{\mathbb{R}}}1_{\{|u-v|\leq c\epsilon\}}\rho_{T}^{\epsilon}(u-v)g^{\epsilon}(v)dv\right)du.
\end{eqnarray*}
For the first term of the r.h.s., as there exists $\gamma>0$ such
that $\underset{\epsilon>0}{\sup}||g^{\epsilon}||_{\infty}\leq\gamma$
, the change of variables $w:=\frac{u-v}{\epsilon}$ gives 
\[
\int_{s}^{t}\left(\int_{\mathcal{\mathbb{R}}}1_{\{|u-v|>c\epsilon\}}\rho_{T}^{\epsilon}(u-v)g^{\epsilon}(v)dv\right)du\leq\gamma(t-s)\int_{\mathbb{R}}1_{\{|w|>c\}}\rho_{T}(w)dw.
\]
For the second term of the r.h.s, 
\[
\int_{s}^{t}\left(\int_{\mathcal{\mathbb{R}}}1_{\{|u-v|\leq c\epsilon\}}\rho_{T}^{\epsilon}(u-v)g^{\epsilon}(v)dv\right)du\leq\int_{s-c\epsilon}^{t+c\epsilon}g^{\epsilon}(v)\left(\int_{s}^{t}\rho_{T}^{\epsilon}(u-v)du\right)dv\leq\int_{s-c\epsilon}^{t+c\epsilon}g^{\epsilon}(v)dv.
\]
We then set $c=\frac{1}{\sqrt{\epsilon}}$, and obtain: 
\[
\int_{s}^{t}\left(\int_{\mathcal{\mathbb{R}}}\rho_{T}^{\epsilon}(u-v)g^{\epsilon}(v)dv\right)du\leq\gamma(t-s)\int_{\mathbb{R}}1_{\{|w|>\frac{1}{\sqrt{\epsilon}}\}}\rho_{T}(w)+\int_{s-\sqrt{\epsilon}}^{t+\sqrt{\epsilon}}g^{\epsilon}(u)du,
\]
where the r.h.s of the inequality converges to $\int_{s}^{t}g(u)du$
as $\epsilon\rightarrow0$, since $\mbox{\ensuremath{\rho}}_{X}^{\epsilon_{n}}*|\partial_{xx}^{2}\phi|\rightarrow|\partial_{xx}^{2}\phi|$
uniformly. A similar argument applies to the two last terms of Inequality
(\ref{eq:inequalitylimit3})-(\ref{eq:inequalitylimit4}), so that
(\ref{eq:finalinequality1})-(\ref{eq:finalinequality4}) holds.

\begin{lem}
\label{lem:Moments}Let $\epsilon>0$, and let $(X^{\epsilon},Y^{\epsilon})$
be the solution to $(SDE)_{\epsilon}$. Then,
\[
\mathbb{E}[\sup_{t\in[0,T]}|X_{t}^{\epsilon}-X_{0}^{\epsilon}|]\leq T\overline{b}+2\sqrt{T}\overline{a}^{1/2}.
\]
\end{lem}
\begin{proof}
For $t\in[0,T]$, 
\begin{eqnarray*}
|X_{t}^{\epsilon}-X_{0}^{\epsilon}| & \leq & \left|\int_{0}^{t}b_{Y_{s}^{\epsilon}}^{\epsilon}(X_{s})ds\right|+\left|\int_{0}^{t}\sqrt{a_{Y_{s}^{\epsilon}}^{\epsilon}(X_{s}^{\epsilon})}dW_{s}\right|\\
 & \leq & \int_{0}^{t}|b_{Y_{s}}^{\epsilon}(X_{s}^{\epsilon})|ds+\underset{u\leq t}{\sup}\left|\int_{0}^{u}\sqrt{a_{Y_{s}}^{\epsilon}(X_{s}^{\epsilon})}dW_{s}\right|\\
 & \leq & T\overline{b}+\underset{u\leq T}{\sup}\left|\int_{0}^{u}\sqrt{a_{Y_{s}^{\epsilon}}(X_{s}^{\epsilon})}dW_{s}\right|.
\end{eqnarray*}
Taking the supremum on $t\in[0,T]$ and then the expectation, we obtain: 

\[
\mathbb{E}\left[\underset{t\in[0,T]}{\sup}|X_{t}^{\epsilon}-X_{0}^{\epsilon}|\right]\leq\left(T\overline{b}+\mathbb{E}\left[\underset{u\leq T}{\sup}\left|\int_{0}^{u}\sqrt{a_{Y_{s}}^{\epsilon}(X_{s})}dW_{s}\right|\right]\right)
\]
Finally, according to Doob's inequality, 
\[
\mathbb{E}\left[\underset{u\leq T}{\sup}\left|\int_{0}^{u}\sqrt{a_{Y_{s}}^{\epsilon}(X_{s}^{\epsilon})}dW_{s}\right|\right]\leq2\sqrt{T}\overline{a}^{1/2},
\]
and this concludes the proof.
\end{proof}

\begin{prop}
\label{prop:RegularizedUniqueness}For $\epsilon>\text{0}$, $(PDS)_{\epsilon}$
has a unique solution.\end{prop}
\begin{proof}
For $0<s\leq t\leq T$, let us define $\left(X_{t}(s,x,y),Y_{t}(s,x,y)\right)$
the value at $t$ of the solution to the SDE (\ref{eq:SDEJumpFig})
starting at $s$ from $\left(x,y\right)\in\mathcal{S}$. For $\phi:\mathcal{S}\rightarrow\mathbb{R}$
such that for $i\in\{1,...,d\}$, $\phi(\cdot,i)\in C_{c}^{\infty}(\mathbb{R})$,
we define $\mathcal{T}_{t}\phi(s,x,y)=\mathbb{E}\left[\phi\left(X_{t}(s,x,y),Y_{t}(s,x,y)\right)\right]$,
and for $f:[0,t]\times\mathcal{S}\rightarrow\mathbb{R}$, such that
for $1\leq i\leq d$, $f(\cdot,\cdot,i)\in C_{b}^{1,2}([0,t]\times\mathbb{R}),$
we define 
\[
L^{\epsilon}f(s,x,i):=b_{i}^{\epsilon}(s,x)\partial_{x}f(s,x,i)+\frac{1}{2}a_{i}^{\epsilon}(s,x)\partial_{xx}^{2}f(s,x,i)+\sum_{j=1}^{d}q_{ij}^{\epsilon}(s,x)f(s,x,j).
\]
By \cite[Theorem 5.2]{YinRegimeSwitchingProperties}, for $1\leq y\leq d$,
the function $\left(s,x\right)\rightarrow\left(\mathcal{T}_{t}\phi\right)(s,x,y)$
belongs to $C_{b}^{1,2}([0,t]\times\mathbb{R})$, and the function
$s,x,y\rightarrow\mathcal{T}_{t}\phi(s,x,y)$ satisfies the backward
Kolmogorov equation $\partial_{s}\left(\mathcal{T}_{t}\phi\right)+L^{\epsilon}\left(\mathcal{T}_{t}\phi\right)=0$.
Now, let $\left(\hat{\mu}_{1}^{\epsilon}(t,\cdot),...,\hat{\mu}_{d}^{\epsilon}(t,\cdot)\right)_{t\in(0,T]}$
and $\left(\check{\mu}_{1}^{\epsilon}(t,\cdot),...,\check{\mu}_{d}^{\epsilon}(t,\cdot)\right)_{t\in(0,T]}$
be two solutions to $(PDS)_{\epsilon}$ with the same initial condition,
and let us define for $1\leq i\leq d$ and $t\in(0,T]$, $z_{i}(t,\cdot):=\hat{\mu}_{i}^{\epsilon}(t,\cdot)-\check{\mu}_{i}^{\epsilon}(t,\cdot)$.
It is sufficient to prove 
\begin{eqnarray}
\frac{d}{ds}\sum_{i=1}^{d}\int\left(\mathcal{T}_{t}\phi\right)(s,x,i)z_{i}(s,dx) & = & \sum_{i=1}^{d}\int\left[\partial_{s}\left(\mathcal{T}_{t}\phi\right)(s,x,i)+L^{\epsilon}\left(\mathcal{T}_{t}\phi\right)(s,x,i)\right]z_{i}(s,dx)\label{eq:Tt}\\
 & = & 0.\label{eq:egal0}
\end{eqnarray}
to obtain uniqueness. Indeed, $s\rightarrow\sum_{i=1}^{d}\int\left(\mathcal{T}_{t}\phi\right)(s,x,i)z_{i}(s,dx)$
would be constant on $(0,t)$, and moreover, for $s\in(0,t)$, 
\begin{eqnarray*}
\sum_{i=1}^{d}\int\left(\mathcal{T}_{t}\phi\right)(s,x,i)z_{i}(s,dx) & = & \sum_{i=1}^{d}\int\left|\left(\mathcal{T}_{t}\phi\right)(s,x,i)-\left(\mathcal{T}_{t}\phi\right)(0,x,i)\right|z_{i}(s,dx)+\sum_{i=1}^{d}\int\left(\mathcal{T}_{t}\phi\right)(0,x,i)z_{i}(s,dx)\\
 & \leq & 2Bs\sum_{i=1}^{d}||\partial_{s}\mathcal{T}_{t}(\cdot,\cdot,i)||_{\infty}+\sum_{i=1}^{d}\int\left(\mathcal{T}_{t}\phi\right)(0,x,i)z_{i}(s,dx),
\end{eqnarray*}
so that 
\[
\sum_{i=1}^{d}\int\left(\mathcal{T}_{t}\phi\right)(s,x,i)z_{i}(s,dx)\underset{s\rightarrow0}{\rightarrow}0,
\]
as $\hat{\mu}_{i}^{\epsilon}$ and $\check{\mu}_{i}^{\epsilon}$ satisfy
the same initial conditions. Moreover, at the limit $s\rightarrow t$,
we obtain that
\[
\sum_{i=1}^{d}\int\phi(x,i)z_{i}(t,dx)=0,
\]
for any function $\phi:\mathcal{S}\rightarrow\mathbb{R}$ such that
for $i\in\{1,...,d\}$, $\phi(\cdot,i)\in C_{c}^{\infty}(\mathbb{R})$.
Thus, we can conclude that $z_{i}(t,\cdot)=0$ for $1\leq i\leq d$
and $t\in(0,T]$, hence the uniqueness. 

We now prove equality (\ref{eq:Tt}) by a density argument. Let $\kappa\in C_{c}^{\infty}((0,t))$,
and let us show that:
\[
\int_{0}^{t}-\kappa'(s)\sum_{i=1}^{d}\int\left(\mathcal{T}_{t}\phi\right)(s,x,i)z_{i}(s,dx)ds=\int_{0}^{t}\kappa(s)\sum_{i=1}^{d}\int\left[\partial_{s}\left(\mathcal{T}_{t}\phi\right)(s,x,i)+L^{\epsilon}\left(\mathcal{T}_{t}\phi\right)(s,x,i)\right]z_{i}(s,dx)ds.
\]
For $\psi,\varphi$ two real valued functions defined on $\mathcal{S}$
such that for $i\in\{1,...,d\}$, $\varphi(\cdot,i)\in C_{c}^{\infty}((0,t))$
and $\psi(\cdot,i)\in C_{c}^{\infty}(\mathbb{R})$, we have in the
sense of distributions, 
\[
\frac{d}{ds}\sum_{i=1}^{d}\int\varphi(s,i)\psi(x,i)z_{i}(s,dx)=\sum_{i=1}^{d}\int\left[\varphi'(s,i)\psi(x,i)+L^{\epsilon}\left(\varphi\psi\right)(s,x,i)\right]z_{i}(s,dx).
\]
Let $h\in C_{c}^{\infty}(\mathbb{R})$ be such that $0\leq h\leq1$,
$h(0)=1$, and $h(x)=0$ if $x\notin(-1,1)$. For $M>0$, we define
$h_{M}\in C_{c}^{\infty}(\mathbb{R})$ such that $h_{M}(x)=1$ if
$x\in[-M,M]$, $h_{M}(x)=0$ if $x\notin[-M-1,M+1]$, $h_{M}(x)=h(x+M)$
if $x\in[-M-1,-M]$, and $h_{M}(x)=h(x-M)$ if $x\in[M,M+1]$. Let
us remark that the family $\left(h_{M}T_{t}\phi\right)_{M>0}$ has
uniform in $M$ bounds. Let $\eta_{1}>0$, and choose $M$ such that
$\sum_{i=1}^{d}\int_{0}^{t}\int_{\mathbb{R}\backslash[-M,M]}\hat{\mu}_{1}^{\epsilon}(s,dx)<\eta_{1}$,
and $\sum_{i=1}^{d}\int_{0}^{t}\int_{\mathbb{R}\backslash[-M,M]}\check{\mu}_{1}^{\epsilon}(s,dx)<\eta_{1}$.
For $\eta_{2}>0$, there exists $p\in\mathbb{N}^{*}$, a family $\left(\psi_{k}\right)_{1\leq k\leq p}$
of functions defined on $(0,t)\times\{1,...,d\}$, such that for $1\leq i\leq d$,
and $1\leq k\leq p$, $\psi_{k}(\cdot,i)\in C_{c}^{\infty}((0,t))$
and a family $\left(\varphi_{k}\right)_{1\leq k\leq p}$ of functions
defined on $\mathcal{S}$ such that for $1\leq i\leq d$, and $1\leq k\leq p$,
$\varphi_{k}(\cdot,i)\in C_{c}^{\infty}(\mathbb{R})$, and moreover,
$\sum_{k=0}^{p}\psi_{k}(\cdot,i)\varphi_{k}(\cdot,i)$ has support
in $(0,t)\times[-M-2;M+2]$, and such that the function $\sum_{k=0}^{p}\psi_{k}\varphi_{k}-h_{M}\mathcal{T}_{t}\phi$
and its derivatives are smaller than $\eta_{2}$. Using the decomposition
$\mathcal{T}_{t}\phi=\left(h_{M}+1-h_{M}\right)\mathcal{T}_{t}\phi$,
it is easy to check that 
\[
\left|\sum_{i=1}^{d}\int_{0}^{t}-\kappa'(s)\int_{\mathbb{R}}\left(\sum_{k=0}^{p}\psi_{k}\varphi_{k}-\mathcal{T}_{t}\phi\right)(s,x,i)z_{i}(s,dx)ds\right|\leq K_{1}\left(\eta_{1}+\eta_{2}\right),
\]

\begin{eqnarray*}
\left|\sum_{i=1}^{d}\int_{0}^{t}\kappa(s)\int_{\mathbb{R}}\left(\sum_{k=0}^{p}\left(\psi_{k}^{'}\varphi_{k}+L^{\epsilon}\left(\psi_{k}\varphi_{k}\right)\right)-\left(\partial_{s}\left(\mathcal{T}_{t}\phi\right)+L^{\epsilon}\left(\mathcal{T}_{t}\phi\right)\right)\right)(s,x,i)z_{i}(s,dx)ds\right| & \leq & K_{2}\left(\eta_{1}+\eta_{2}\right),
\end{eqnarray*}
where $K_{j},j=1,2,$ are positive values that do not depend on $M$
or $p$, and this concludes the proof.
\end{proof}
\bibliographystyle{plain}
\bibliography{RSLV_final}

\begin{thebibliography}{10}

\bibitem{AbergelTachet}
F.~Abergel and R.~Tachet.
\newblock A non-linear partial integro-differential equation from mathematical
  finance.
\newblock {\em Discrete Contin. Dyn. Syst.(3)}, 907-917, 2010.

\bibitem{Albin}
J.~M.~P. Albin.
\newblock {A continuous non-Brownian motion martingale with Brownian motion
  marginal distributions}.
\newblock {\em {Statistics and Probability Letters}}, 78(6):682, January 2010.

\bibitem{Alfonsi}
A.~Alfonsi, C.~Labart, and J.~Lelong.
\newblock Stochastic local intensity loss models with interacting particle
  systems.
\newblock {\em Mathematical Finance}, 26(2):366--394, 2016.

\bibitem{Baker2011}
D.~Baker, C.~Donati-Martin, and M.~Yor.
\newblock {\em A Sequence of Albin Type Continuous Martingales with Brownian
  Marginals and Scaling}, pages 441--449.
\newblock Springer Berlin Heidelberg, Berlin, Heidelberg, 2011.

\bibitem{Rockner}
V.~I. Bogachev, M.~R{\"o}ckner, and S.~V. Shaposhnikov.
\newblock Positive densities of transition probabilities of diffusion
  processes.
\newblock {\em Theory of Probability \& Its Applications}, 53(2):194--215,
  2009.

\bibitem{BreedenLitzenberger}
D.~T. Breeden and R.~H. Litzenberger.
\newblock Prices of state-contingent claims implicit in option prices.
\newblock {\em The Journal of Business}, 51(4):621--51, 1978.

\bibitem{Brezis}
H.~Brezis.
\newblock {\em Analyse fonctionnelle, Th\'eorie et applications}.
\newblock Masson, Paris, 1983.

\bibitem{etaconvergenceLaw}
P.~Cannarsa and T.~D'Aprile.
\newblock Lecture notes on measure theory and functional analysis, 2006.

\bibitem{Dupire}
B.~Dupire.
\newblock Pricing with a \textsc{S}mile.
\newblock {\em Risk}, 7, 18-20, 1994.

\bibitem{EthierKurtz}
S.~N. Ethier and T.~G. Kurtz.
\newblock {\em Markov processes: characterization and convergence}.
\newblock Wiley series in probability and mathematical statistics. Probability
  and mathematical statistics. Wiley, 1986.

\bibitem{fan2015}
J.~Y. Fan, K.~Hamza, and F.~Klebaner.
\newblock Mimicking self-similar processes.
\newblock {\em Bernoulli}, 21(3):1341--1360, 08 2015.

\bibitem{Figalli}
A.~Figalli.
\newblock Existence and uniqueness of martingale solutions for sdes with rough
  or degenerate coefficients.
\newblock {\em Journal of Functional Analysis}, 254(1):109 -- 153, 2008.

\bibitem{FournierHauray}
N.~{Fournier} and M.~{Hauray}.
\newblock {Propagation of chaos for the Landau equation with moderately soft
  potentials}.
\newblock {\em ArXiv e-prints}, January 2015.

\bibitem{GuennounLabordere}
H.~Guennoun and P.~Henry-Labord\`ere.
\newblock Local volatility models enhanced with jumps.
\newblock {\em Available at SSRN: https://ssrn.com/abstract=2781102}, May.
  2016.

\bibitem{GuyonLabordere}
J.~Guyon and P.~Henry-Labord\`ere.
\newblock Being particular about calibration.
\newblock {\em Risk magazine}, pages 88--93, Jan. 2012.

\bibitem{Gyongy}
I.~Gy\"ongy.
\newblock Mimicking the one-dimensional marginal distribution of processes
  having an \textsc{I}to differential.
\newblock {\em Probab. Theory Relat. Fields}, 71(4), 501-516, 1986.

\bibitem{hale2009ordinary}
J.K. Hale.
\newblock {\em Ordinary Differential Equations}.
\newblock Dover Books on Mathematics Series. Dover Publications, 2009.

\bibitem{HamzaKlebaner}
K.~Hamza and F.~C. Klebaner.
\newblock A family of non-gaussian martingales with gaussian marginals.
\newblock {\em J. Appl. Math. and Stochastic Analysis}, 2007.

\bibitem{HLabordereTanTouzi}
P.~Henry-Labordere, X.~Tan, and N.~Touzi.
\newblock An explicit martingale version of the one-dimensional
  \textsc{B}renier's theorem with full marginals constraint.
\newblock {\em Stochastic Processes and their Applications}, 126(9):2800 --
  2834, 2016.

\bibitem{Hirsch2011}
F.~Hirsch, C.~Profeta, B.~Roynette, and M.~Yor.
\newblock {\em Constructing Self-Similar Martingales via Two Skorokhod
  Embeddings}, pages 451--503.
\newblock Springer Berlin Heidelberg, Berlin, Heidelberg, 2011.

\bibitem{Hobson}
D.~G. Hobson.
\newblock Fake exponential brownian motion.
\newblock {\em Statistics \& Probability Letters}, 83(10):2386 -- 2390, 2013.

\bibitem{Kurtz2011}
Thomas~G. Kurtz.
\newblock {\em Equivalence of Stochastic Equations and Martingale Problems},
  pages 113--130.
\newblock Springer Berlin Heidelberg, Berlin, Heidelberg, 2011.

\bibitem{Lipton:2002}
A.~Lipton.
\newblock The vol smile problem.
\newblock {\em Risk Magazine}, 15:61--65, 2002.

\bibitem{MadanYor}
D.~B. Madan and M.~Yor.
\newblock Making markov martingales meet marginals: with explicit
  constructions.
\newblock {\em Bernoulli}, 8(4):509--536, 08 2002.

\bibitem{Menaldi}
J.~L. Menaldi.
\newblock {\em Stochastic differential equations with jumps}.
\newblock 2014.

\bibitem{menozzi2010}
S.~Menozzi and V.~Lemaire.
\newblock On some non asymptotic bounds for the euler scheme.
\newblock {\em Electron. J. Probab.}, 15:1645--1681, 2010.

\bibitem{Oleszkiewicz}
K.~Oleszkiewicz.
\newblock {On fake Brownian motions}.
\newblock {\em {Statistics and Probability Letters}}, 78(11):1251, March 2010.

\bibitem{Piterbarg}
V.~Piterbarg.
\newblock Markovian projection method for volatility calibration.
\newblock {\em Risk Magazine}, 20:84--89, 2007.

\bibitem{StroockV}
D.~W. Stroock and S.~R.~S. Varadhan.
\newblock {\em Multidimensional diffusion processes}, volume 233 of {\em
  Grundlehren der Mathematischen Wissenschaften [Fundamental Principles of
  Mathematical Sciences]}.
\newblock Springer-Verlag, Berlin, 1979.
\newblock Reprinted in 2006.

\bibitem{temam}
R.~Temam.
\newblock {\em Navier-Stokes Equations}.
\newblock North-Holland, Amsterdam, 1979.

\bibitem{YinRegimeSwitchingProperties}
G.~Yin and C.~Zhu.
\newblock Properties of solutions of stochastic differential equations with
  continuous-state-dependent switching.
\newblock {\em Journal of Differential Equations}, 249(10):2409 -- 2439, 2010.

\bibitem{zhang}
X.~Zhang.
\newblock Degenerate irregular sdes with jumps and application to
  integro-differential equations of fokker-planck type.
\newblock {\em Electron. J. Probab.}, 18:25 pp., 2013.

\end{thebibliography}

\end{document}